\documentclass[]{amsart}

\usepackage{amsmath}
\usepackage{amsthm}
\usepackage{amsfonts}
\usepackage{amssymb}
\usepackage[shortlabels]{enumitem}
\usepackage{todonotes}
\usepackage{float}
\usepackage{mathtools}
\usepackage{csquotes}
\usepackage{changes}
\usepackage{comment}
\usepackage{tikz-cd}
\tikzcdset{every label/.append style = {font = \normalsize}}

\usepackage[style=alphabetic,sorting=nyt,backend=bibtex8, maxbibnames=99]{biblatex}
\bibliography{reference.bib}

\usepackage{hyperref}
\hypersetup{colorlinks=true, urlcolor=black}

\usepackage{graphicx}

\makeatletter
\newcommand{\oast}{\mathbin{\mathpalette\make@circled\ast}}
\newcommand{\make@circled}[2]{%
  \ooalign{$\m@th#1\smallbigcirc{#1}$\cr\hidewidth$\m@th#1#2$\hidewidth\cr}%
}
\newcommand{\smallbigcirc}[1]{%
  \vcenter{\hbox{\scalebox{0.77778}{$\m@th#1\bigcirc$}}}%
}
\makeatother

\newtheorem{theorem}{Theorem}[section]
\newtheorem{maintheorem}{Theorem}

\newtheorem{maincorollary}{Corollary}

\newtheorem*{theorem*}{Theorem}
\newtheorem{proposition}[theorem]{Proposition}
\newtheorem*{Dproblem}{Ditor's Problem}
\newtheorem{corollary}[theorem]{Corollary}
\newtheorem{lemma}[theorem]{Lemma}
\newtheorem{claim}{Claim}[theorem]
\newtheorem*{subclaim*}{Subclaim}
\theoremstyle{definition}

\newtheorem{question}{Question}
\newtheorem*{question*}{Question}
\newtheorem{definition}[theorem]{Definition}

\theoremstyle{remark}
\newtheorem{remark}[theorem]{Remark}

\begin{document}

\title{On Maximal Ladders}

\author{Lorenzo Notaro}
\address{University of Vienna, Institute of Mathematics, Kurt G\"{o}del Research Center, Kolingasse 14-16, 1090 Vienna, Austria}
\curraddr{}
\email{lorenzo.notaro@univie.ac.at}

\begin{abstract}
Given a positive integer $n$, an $n$-ladder is a lower finite lattice whose elements have at most $n$ lower covers. In 1984, Ditor proved that every $n$-ladder has cardinality at most $\aleph_{n-1}$ and asked whether this bound is sharp, i.e., whether for each $n$ there is an $n$-ladder of cardinality $\aleph_{n-1}$. We isolate the notion of maximal $n$-ladder and use it to study Ditor's problem and related questions. We show that $\text{Add}(\omega, \omega_\omega)$ forces every maximal $n$-ladder to have cardinality $\aleph_{n-1}$, and hence forces a positive answer to Ditor's question for every $n$. In particular, it is consistent that there are no maximal $3$-ladders of cardinality $\aleph_1$. However, we show that the existence of such a ladder follows from $\mathfrak{d}=\aleph_1$. Under $\clubsuit$, we construct a maximal $3$-ladder of breadth $2$. Finally, we prove that, consistently (under $\diamondsuit$), there exists a maximal $3$-ladder that is destructible by forcing with a Suslin tree.
\end{abstract}

\thanks{This research was funded in whole or in part by the Austrian Science Fund (FWF) \href{https://www.fwf.ac.at/en/research-radar/10.55776/ESP1829225}{10.55776/ESP1829225}. For open access purposes, the author has applied a CC BY public copyright license to any author accepted manuscript version arising from this submission.}

\subjclass[2020]{Primary 03E05, Secondary 03E35, 06A07}
\keywords{lattice, lower cover, ladders, Ditor's problem, maximal ladders}
\maketitle

{
  \hypersetup{linkcolor=black}
  \tableofcontents
}

\addtocontents{toc}{\protect\setcounter{tocdepth}{1}}
\section{Introduction}

 Given a positive integer $n$, an \emph{$n$-ladder} is a lower finite lattice whose elements have at most $n$ lower covers. The notion of an $n$-ladder was introduced independently by Ditor~\cite{MR0732199} and Dobbertin~\cite{MR862871} under different names\footnote{Ditor called these lattices \emph{$n$-lattices} while Dobbertin called them \emph{$n$-frames}. We follow the terminology of Gr\"{a}tzer, Lakser, and Wehrung \cite{MR1768850}.}. All the results of Ditor we cite are from~\cite{MR0732199}.
 
Ditor showed that every $n$-ladder  has cardinality at most $\aleph_{n-1}$ (for a more general result, see Theorem~\ref{thm:Ditor}). He then asked whether this cardinality bound is sharp (see also \cite[p. 291]{MR2768581}):

\begin{Dproblem}
For every $n > 0$, is there an $n$-ladder of cardinality $\aleph_{n-1}$?
\end{Dproblem}

The case $n=1$ is trivial: $\omega$ with its usual ordering is a $1$-ladder of cardinality $\aleph_0$. Ditor proved that the answer to his question is also positive when $n=2$, i.e., there exists a $2$-ladder of cardinality $\aleph_1$. Since then, $2$-ladders have been used primarily, but not exclusively, in representation problems in universal algebra for structures of cardinality $\leq \aleph_1$ (e.g. \cite{MR862871,MR1768850, MR1800815, MR2309879}).

Recent progress has been made on the remaining cases of Ditor’s Problem. Wehrung \cite{MR2609217} proved that the existence of a $3$-ladder of cardinality $\aleph_2$ follows from either of two independent set-theoretic assumptions: $\mathsf{MA}_{\omega_1}(\aleph_1\text{-precaliber})$, that is, $\mathsf{MA}_{\omega_1}$ restricted to forcings of precaliber $\aleph_1$; and the existence of an $(\omega_1, 1)$-morass. 

Then, the author \cite{MR4993406} proved that for every $n > 2$ the existence of an $n$-ladder of cardinality $\aleph_{n-1}$ follows from $\square_{\omega_1} + \square_{\omega_2} +\ldots + \square_{\omega_{n-2}}$, i.e., from Jensen's $\square_\kappa$ holding at the first $n-2$ uncountable cardinals. In particular, $\square_{\omega_1}$ implies the existence of a $3$-ladder of cardinality $\aleph_2$. Moreover, since the axiom of constructibility $\mathsf{V=L}$ implies $\square_\kappa$ for every uncountable cardinal $\kappa$, we conclude from the author's result that Ditor's problem has a positive solution for every $n$ under $\mathsf{V=L}$.

In this paper we introduce the notion of \emph{maximal} $n$-ladder. An $n$-ladder is maximal if it is not isomorphic to a proper ideal of an $n$-ladder (Definition~\ref{def:maximal}). This notion stems naturally from the work of Ditor and Wehrung \cite{MR0732199, MR2609217}. Although the notion of maximal $n$-ladders has not previously been made explicit, several existing results can already be phrased in those terms. In particular, Ditor's results already imply that every $n$-ladder of cardinality $\aleph_{n-1}$ is maximal (Corollary~\ref{cor:ditornomax}) and that every maximal $n$-ladder, with $n > 1$, is uncountable (Theorem~\ref{thm:Ditormax}). 

Using the notion of maximality, we can also say something more precise about Wehrung's results. In fact, Wehrung not only proved that $\mathsf{MA}_{\omega_1}(\aleph_1\text{-precaliber})$ implies the existence of a $3$-ladder of cardinality $\aleph_2$, but, more generally, that every maximal $n$-ladder, with $n > 2$, has cardinality at least $\aleph_2$ (cf. Theorem~\ref{thm:Ditormax}).

After providing a characterization of maximality (Theorem~\ref{thm:characterization}), we prove our main results. The first builds on Wehrung’s forcing argument and yields, as a corollary, a new proof of the consistency of a positive answer to Ditor’s Problem for every $n$.  Here $\text{\upshape Add}(\omega, \omega_{n})$ is Cohen's forcing for adding $\aleph_{n}$ Cohen reals.

\begin{maintheorem}\label{thm:main1}
For every $m > 1$, $\text{\upshape Add}(\omega, \omega_{m})$ forces that every maximal $n$-ladder has cardinality at least $\aleph_{\min\{n-1, m\}}$. 
\end{maintheorem}
\begin{maincorollary}\label{cor:main1}
$\text{\upshape Add}(\omega, \omega_\omega)$ forces that every maximal $n$-ladder has cardinality $\aleph_{n-1}$ for every $n> 0$.
\end{maincorollary}

In particular, since every $n$-ladder extends to a maximal $n$-ladder, adding $\aleph_\omega$ many Cohen reals forces Ditor's Problem to have a positive answer for every $n > 0$.  

Theorem~\ref{thm:main1} and the above-mentioned result of Wehrung show that, consistently, every maximal $3$-ladder has cardinality $\aleph_2$, or, equivalently, there are no maximal $3$-ladders of cardinality $\aleph_1$. The natural question, asked implicitly by Wehrung \cite[p. 7]{MR2609217}, is whether the existence of a maximal $3$-ladder of cardinality $\aleph_1$ is consistent. Our second result answers this question in the positive.

\begin{maintheorem}\label{thm:main2}
If $\mathfrak{d}=\aleph_1$, then there exists a maximal $3$-ladder of cardinality $\aleph_1$.
\end{maintheorem}
Thus, in particular, the existence of a maximal $3$-ladder of cardinality $\aleph_1$ follows from $\mathsf{CH}$. The proof of Theorem~\ref{thm:main2} isolates a natural class of $3$-ladders that exhibits a particularly strong relationship with the dominating number.

The question about the existence of maximal $3$-ladders of cardinality $\aleph_1$ has a purely order-theoretic version, which we believe has independent interest: is there a maximal $3$-ladder of breadth $2$? Indeed, a maximal $3$-ladder of breadth $2$ must have cardinality $\aleph_1$: every maximal $3$-ladder is uncountable by Theorem~\ref{thm:Ditormax} and every lower finite join-semilattice of breadth $2$ has cardinality at most $\aleph_1$ by Theorem~\ref{thm:Ditor}. Our next theorem shows that, consistently, there are such maximal ladders. In particular, we construct a maximal $3$-ladder of breadth $2$ from a guessing principle, the club principle $\clubsuit$, which is a weakening of Jensen's diamond principle $\diamondsuit$.

\begin{maintheorem}\label{thm:main3}
If $\clubsuit$ holds, then there exists a maximal $3$-ladder of breadth $2$.
\end{maintheorem}

Finally, we study how forcing affects the maximality of $n$-ladders. For a forcing notion $\mathbb{P}$, a maximal $n$-ladder $L$ is $\mathbb{P}$-indestructible if $\mathbb{P}$ forces $L$ to stay maximal in every $\mathbb{P}$-generic extension; otherwise, $L$ is $\mathbb{P}$-destructible. This notion lies at the core of many open questions (see Questions~\ref{q:3} and \ref{q:4}) revolving around the main one (Question~\ref{q:1}): is the existence of a $3$-ladder of cardinality $\aleph_2$ a theorem of $\mathsf{ZFC}$? 


Our last result answers a question that naturally arises given Wehrung's work and Theorem~\ref{thm:main1}, and that is intimately connected to Ditor's problem (see Question~\ref{q:3}): is it necessary to add new generic reals in order to destroy the maximality of an $n$-ladder? The proof of Theorem~\ref{thm:main1}, extending Wehrung's forcing result, shows that we can always destroy the maximality of an $n$-ladder of cardinality $< \aleph_{n-1}$ by adding enough Cohen reals, while Lemma~\ref{lemma:bounding} shows that for certain maximal $3$-ladders adding new reals is indeed necessary. Our final theorem nevertheless gives a negative answer: consistently, there exists a maximal $3$-ladder whose maximality is destroyed by a forcing that does not add new reals.

\begin{maintheorem}\label{thm:main4}
If $\diamondsuit$ holds, then there is a maximal $3$-ladder which is $T$-destructible for some Suslin tree $T$.
\end{maintheorem}

In Section~\ref{sec:prel}, we introduce the basic concepts and notation, discuss Ditor’s results, and prove our characterization theorem for maximal $n$-ladders (Theorem~\ref{thm:characterization}). Sections~\ref{sec:cohen}, \ref{sec:dominating} and \ref{sec:club} are devoted to the proofs of Theorems~\ref{thm:main1}, \ref{thm:main2} and \ref{thm:main3}, respectively. In Section~\ref{sec:diamond} we start by discussing the notion of destructibility with respect to the results of the previous sections and then we prove Theorem~\ref{thm:main4}. Finally, Section~\ref{sec:openquestions} contains a selection of the many questions that remain open.

\section{Preliminaries}\label{sec:prel}

\subsection{Notation and basic concepts}
The monographs \cite{MR1940513} and \cite{MR2768581} are our references for all classical definitions and notation in set theory and lattice theory, respectively. 

A \emph{tree} $(T, \le)$ is a poset such that, for each $x \in T$, the set $\{y \in T \mid y < x\}$ is well-ordered by $\le$. If $x \in T$, the \emph{height} of $x$ in $T$, denoted by $\mathrm{ht}(x)$, is the order-type of $\{y \in T \mid y < x\}$. Moreover, for each ordinal $\alpha$, the $\alpha$-th \emph{level} of $T$, denoted by $T(\alpha)$ is the set of all the elements of $T$ of height $\alpha$. 

A \emph{join-semilattice} is a nonempty set equipped with a binary, associative, commutative, and idempotent operation called \emph{join}, denoted by $\vee$; it induces a partial order via $x \le y \iff x \vee y = y$. Equivalently, a join-semilattice is a partially ordered set in which every pair of elements $x,y$ admits a least upper bound, denoted by $x \vee y$. The dual notion is the \emph{meet-semilattice}, with the operation called \emph{meet}. A \emph{lattice} is both a join- and a meet-semilattice. We treat (semi)lattices as algebraic structures or as posets depending on what representation is better suited for the given context.

Given a poset $(P, \le)$ and some $p \in P$, we denote by $P \downarrow p$ and $P \uparrow p$ the sets $\{q \in P : q \le p\}$ and $\{q \in P : q \ge p\}$, respectively. Sometimes, instead of $P \downarrow p$ we write ${\le} \downarrow p$ or simply ${\downarrow} p$, when no ambiguity arises.  Given a set $X \subseteq P$, we call the set $\bigcup_{p \in X} {\downarrow} p$ the \emph{downward closure} of $X$ and denote it by ${\downarrow }X$. A subset $D \subseteq P$ is \emph{downward closed} if $D = {\downarrow} D$. Moreover, an \emph{ideal} of $P$ is a nonempty subset $I$ of $P$ which is downward closed and upward directed (i.e., every two elements of $I$ have an upper bound in $I$). An ideal $I$ that does not coincide with the whole poset is called a \emph{proper ideal}. Note that ${\downarrow} p$ is an ideal of $P$ for every $p \in P$; such ideals are known as \emph{principal ideals}. A poset is \emph{lower finite} if its principal ideals are finite. 

If $P$ is a join-semilattice, it follows trivially that ideals are the downward closed subsets which are closed under $\vee$. Moreover, given a nonempty subset $X \subseteq P$, the ideal \emph{generated by} $X$ is denoted by $\mathrm{id}(X)$, i.e.,
\[
\mathrm{id}(X) \coloneqq {\downarrow} \big\{x_0 \vee x_1 \vee \ldots \vee x_k : k < \omega \text{ and } x_i \in X \text{ for all }i \le k\big\}.
\]

If $P$ is a meet-semilattice, recall that a nonempty subset $S\subseteq P$ is called a \emph{meet-subsemilattice} of $P$ if it is closed under binary meets---equivalently, every two elements in $S$ have their greatest lower bound in $S$.


Furthermore, given two elements $p,q \in P$, $q$ is a \emph{lower cover of $p$} if $q < p$ and there is no $x \in P$ with $q < x < p$. 

Let us also recall the definition of breadth, a classical numeric invariant of lattice theory. 
\begin{definition}
Let $P$ be a join-semilattice and $n > 0$ be a positive integer. We say that $P$ has \emph{breadth at most $n$} if, for every nonempty finite subset $X$ of  $P$, there exists $Y\subseteq X$ with at most $n$ elements such that $\bigvee X = \bigvee Y$. The \emph{breadth} of $P$ is the least $n > 0$ such that $P$ has breadth at most $n$, if such $n$ exists.
\end{definition}
In fact, there is a more general notion of breadth which is self-dual and purely poset-theoretical \cite[\S 4]{MR0732199}. Note that the class of join-semilattices of breadth $1$ coincides with the class of linear orders. The next lemma is immediate from the definition.

\begin{lemma}
Given a join-semilattice $P$ and an $n > 0$, the following are equivalent:
\begin{enumerate}[label={\upshape (\arabic*)}]
\itemsep0.3em
\item $P$ has breadth at most $n$.
\item For every $X \in [P]^{n+1}$, there exists $Y \in [X]^n$ such that $\bigvee X = \bigvee Y$.
\end{enumerate}
\end{lemma}

Finally, let us introduce a nonstandard notation. Given a lower finite lattice $P$, an ideal $I\subseteq P$, and an element $x \in P$, let
\[
\pi_I(x) \coloneqq \bigvee \{y \in I : y \le x\}.
\]
Since ${\downarrow} x$ is finite, and $P$ has a least element, the set $\{y \in I : y \le x\}$ is finite and nonempty; hence its join is well-defined. Equivalently, $\pi_I(x)$ is the greatest element of $I \cap ({\downarrow} x)$.

The next two lemmas establish properties of the map $\pi_I$.

\begin{lemma}\label{lemma:meetproj}
Given a lower finite lattice $P$, some elements $x,y \in P$ and an ideal $I \subseteq P$, $\pi_I(x \wedge y) = \pi_I(x) \wedge \pi_I(y)$.
\end{lemma}
\begin{proof}
Since $\pi_I(x) \le x$, we have $\pi_I(x) \wedge \pi_I(y) \le x \wedge y$. As $\pi_I(x) \wedge \pi_I(y) \in I$, we conclude  $\pi_I(x) \wedge \pi_I(y) \le \pi_I(x \wedge y)$.

Furthermore, since $\pi_I(x \wedge y) \le x \wedge y \le x,y$ and $\pi_I(x \wedge y) \in I$,  we conclude $\pi_I(x \wedge y) \le \pi_I(x) \wedge \pi_I(y)$. Overall, we have $\pi_I(x \wedge y) = \pi_I(x) \wedge \pi_I(y)$.
\end{proof}

In other words, Lemma~\ref{lemma:meetproj} tells us that $\pi_I : P\rightarrow I$ is a meet-homomorphism. Moreover, if $x,y \in P$ are such that $x \wedge y \in I$, we conclude from Lemma~\ref{lemma:meetproj} that $x \wedge y = \pi_I(x) \wedge \pi_I(y)$.

\begin{lemma}\label{lemma:concproj}
If $P$ is a lower finite lattice and $I,J \subseteq P$ are ideals with $I \subseteq J$, then $\pi_I \circ \pi_J = \pi_I$.
\end{lemma}
\begin{proof}
Fix an $x \in P$. Clearly, $\pi_I \circ \pi_J(x) \le \pi_I(x)$. Furthermore, as $I \subseteq J$ and $J$ is an ideal, we have $\pi_I (x) \vee \pi_J(x) \in J$. We have also $\pi_I(x) \vee \pi_J(x) \le x$, since $\pi_I(x), \pi_J(x) \le x$. Thus, $\pi_I (x) \vee \pi_J(x) \le \pi_J(x)$, or, equivalently, $\pi_I(x) \le \pi_J(x)$. It directly follows from the last statement that $\pi_I(x) \le \pi_I \circ \pi_J (x)$. Overall, $\pi_I \circ \pi_J (x) = \pi_I (x)$.
\end{proof}

\subsection{Ladders and Ditor's results}

An \emph{$n$-ladder} is a lower finite lattice in which every element has at most $n$ lower covers. It is easy to prove that every $n$-ladder has breadth at most $n$ \cite[Proposition 4.1]{MR0732199}. The converse fails: the breadth can be strictly smaller than the lower cover bound; for example, the diamond lattice $M_3$ is a $3$-ladder of breadth $2$.  

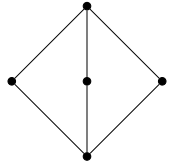
\begin{figure}[H]
\centering
\begin{tikzpicture}
    \node [draw, shape = circle, fill = black, minimum size = 0.1cm, inner sep=0pt] at (0,1) (t){};
    \node [draw, shape = circle, fill = black, minimum size = 0.1cm, inner sep=0pt] at (0,-1) (b){};
     \node [draw, shape = circle, fill = black, minimum size = 0.1cm, inner sep=0pt] at (-1,0) (a0){};
      \node [draw, shape = circle, fill = black, minimum size = 0.1cm, inner sep=0pt] at (0,0) (a1){};
       \node [draw, shape = circle, fill = black, minimum size = 0.1cm, inner sep=0pt] at (1,0) (a2){};
    \draw[ultra thin] (b)--(a0)--(t)--(a1)--(b)--(a2)--(t);
\end{tikzpicture}
\caption{Hasse diagram of $M_3$}
\end{figure}
In 1984, Ditor proved the following result, which gives a cardinal upper bound on the domain of a join-semilattice given its (finite) breadth and the cardinality of its principal ideals. 
\begin{theorem}[{Ditor, \cite{MR0732199}}]\label{thm:Ditor}
Given some $n > 0$ and an infinite cardinal $\kappa$, if $P$ is a join-semilattice of breadth at most $n$ whose principal ideals have cardinality $< \kappa$, then
\begin{enumerate}[label={\upshape (\arabic*)}]
\itemsep0.3em
\item\label{thm:Ditor-1}
\( |P| \le \kappa^{+(n-1)} \), and
\item \label{thm:Ditor-2}
\( |I| < \kappa^{+(n-1)} \) for every proper ideal $I$ of $P$.
\end{enumerate}
\end{theorem}

As noted by Wehrung, Ditor's Theorem~\ref{thm:Ditor}\ref{thm:Ditor-1} is in fact a fairly direct corollary of Kuratowski's Free Set Theorem \cite{MR48518} (see also \cite[Theorem 46.1]{MR795592}). Since every $n$-ladder has breadth at most $n$, a direct consequence of Theorem~\ref{thm:Ditor}\ref{thm:Ditor-1} is that every  $n$-ladder has cardinality at most $\aleph_{n-1}$.

We now introduce the main notion of our paper: maximal $n$-ladders. 

\begin{definition}\label{def:maximal}
An $n$-ladder is \emph{maximal} if it is not isomorphic to a proper ideal of an $n$-ladder.
\end{definition}

In this definition, it does not matter if we consider order-isomorphisms or lattice-isomorphisms because these two notions coincide when the range of the isomorphism is an ideal.

Note that no $n$-ladder is maximal as an $(n+1)$-ladder. In other words, given an $n$-ladder $P$, we can always find an $(n+1)$-ladder $Q$ such that $P$ is isomorphic to a proper ideal of $Q$. The construction is very simple: if $P$ is an $n$-ladder, then the product lattice $P \times \{0,1\}$ equipped with the product ordering is an $(n+1)$-ladder and $P$ is trivially isomorphic to $P \times \{0\}$, which is a proper ideal of $P \times \{0,1\}$.

Even if maximal $n$-ladders have not yet been explicitly introduced, some results of Ditor and Wehrung on $n$-ladders can be naturally recast using this notion. For example, the following is a direct corollary of Theorem~\ref{thm:Ditor}\ref{thm:Ditor-2}.

\begin{corollary}[Ditor]\label{cor:ditornomax}
Every $n$-ladder of cardinality $\aleph_{n-1}$ is maximal.
\end{corollary}

As mentioned in the introduction, Ditor proved that there is a $2$-ladder of cardinality $\aleph_1$---thus giving a positive answer to his problem for $n = 2$. He does so by proving the following:
\begin{theorem}[{Ditor, \cite[\S 6.2]{MR0732199}}]\label{thm:Ditormax}
Every maximal $n$-ladder, with $n > 1$, is uncountable.
\end{theorem}
Therefore, not only is there a $2$-ladder of cardinality $\aleph_1$, but in fact every maximal $2$-ladder has cardinality $\aleph_1$---recall that, by Theorem~\ref{thm:Ditor}\ref{thm:Ditor-1}, every $2$-ladder has cardinality at most $\aleph_{1}$, and, more generally, every $n$-ladder has cardinality at most $\aleph_{n-1}$.

Ditor's proof of Theorem~\ref{thm:Ditormax} relies crucially on the fact that every countable join-semilattice has a cofinal chain: given a countable $P$, fix an enumeration $\langle x_n : n\in\omega\rangle$ of $P$; then the set $\{\bigvee \{x_m : m \le n\} : n\in\omega\}$ is easily seen to be a cofinal chain. We show that the existence of ``nice" cofinal subsets in an $n$-ladder is indeed key to characterizing the maximality of $n$-ladders.

\begin{theorem}\label{thm:characterization}
For every $n > 1$, an $n$-ladder is maximal if and only if it has no cofinal meet-subsemilattice which is also a $(n-1)$-ladder in the induced order.
\end{theorem}
\begin{proof}
Fix an $n$-ladder $(L, \le)$ and cofinal meet-subsemilattice $C\subseteq L$ which is an $(n-1)$-ladder in the induced ordering. We want to show that $L$ is not maximal. Let $C'$ be a set such that $|C'| = |C|$ and $C' \cap L = \emptyset$. Fix a bijection $\phi: C' \rightarrow C$. Now we extend the partial order $\le$ on $L \cup C'$ so that $(L \cup C', \le)$ is an $n$-ladder and $L$ is a proper ideal of $(L\cup C', \le)$. 

Extend $\le$ on $L \cup C'$ as follows:
\begin{itemize}
\item for every $x,y \in C'$ let $x \le y$ if and only if $\phi(x) \le \phi(y)$, and
\item for every $x \in L$ and $y \in C'$, let $x \le y$ if and only if $x \le \phi(y)$.
\end{itemize}
It is clear by our construction, and by the transitivity of $\le$ on $L$, that $(L \cup C', \le)$ is a poset and $L$ is an ideal of it. Moreover, given an $x \in C'$, the set of its $\le$-predecessors is \[\{y \in L : y \le \phi(x)\} \cup \{y \in C' : \phi(y) \le \phi(x)\},\] which is finite, since $L$ is lower finite. 

Let us show that $(L \cup C', \le)$ is a join-semilattice. First note that for every $x \in L$ there is a $\le$-least $y \in C$ such that $x \le y$. Indeed, the set $(L \uparrow x) \cap C$ is nonempty, since $C$ is cofinal in $(L, \le)$, and it is clearly a lower finite meet-semilattice in the induced ordering. As every lower finite meet-semilattice has a least element, we conclude that $(L \uparrow x) \cap C$ has a $\le$-least element.

Pick distinct $x,y \in L \cup C'$, towards proving that they have a least upper bound in $(L \cup C', \le)$. Let $\vee$ be the join operation of $L$.  If both $x$ and $y$ belong to $C'$ then their least upper bound is easily seen to be $\phi^{-1}(\min\{z \in C : \phi(x), \phi(y) \le z\})$---where the minimum is taken with respect to $\le$. Now suppose $x, y \in L$. We claim that $x \vee y$ (i.e., the least upper bound of $x$ and $y$ in $(L, \le)$) is the least upper bound of $x$ and $y$ in $(L \cup C', \le)$. Pick $z \in C'$ with $x, y \le z$. By definition of the extension of $\le$, $x,y \le \phi(z)$. Then, $x \vee y \le \phi(z)$, and therefore $x \vee y \le z$, as we wanted to show. Finally, if $x \in L$ and $y \in C'$, then it is easy to see that $\phi^{-1}(\min\{z \in C: x, \phi(y) \le z\})$ is the least upper bound of $x$ and $y$. 

We have shown that $(L \cup C', \le)$ is a lower finite join-semilattice and that $L$ is an ideal of it. Since $(L \cup C', \le)$ has a least element (namely the least element of $(L, \le)$), $(L \cup C', \le)$ is actually a lattice. Therefore, to show that $(L \cup C', \le)$ is an $n$-ladder it suffices to prove that every $x \in C'$ has at most $n$ lower covers. Since by assumption $(C, \le)$ is an $(n-1)$-ladder, let $p_0, p_1, \dots, p_{k-1} \in C$ be the lower covers of $\phi(x)$ in $(C, \le)$ with $k < n$. Then, it is easy to check that the lower covers of $x$ in $(L \cup C', \le)$ are $\phi^{-1}(p_0), \phi^{-1}(p_1), \dots, \phi^{-1}(p_{k-1})$ and $\phi(x)$. 

Now let us suppose that the $n$-ladder $(L, \le)$ is not maximal, towards showing that there exists a cofinal meet-subsemilattice $C \subseteq L$ which is also an $(n-1)$-ladder. Let $(K, \le)$ be an $n$-ladder such that $L$ is a proper ideal of $K$. Fix $b \in K \setminus L$ and consider the set $C \coloneqq \big\{\pi_L(x) : x \in K \uparrow b\big\}$. We claim that the set $C$ satisfies the desired properties. First, it is clearly cofinal. Indeed, for every $x \in L$, $x \le \pi_L(b \vee x)$. Moreover, $\pi_L$ is a meet-homomorphism by Lemma~\ref{lemma:meetproj} and therefore $C$ is a meet-subsemilattice of $L$. From $C$ being a cofinal meet-subsemilattice of $(L, \le)$ it follows quickly that $C$ is also a lattice in the induced ordering: as above, any two elements of $C$ have a least-upper bound in $C$, namely the least element of $C$ above $x \vee y$.

Now, reasoning by contraposition, let us suppose that $C$ is not an $(n-1)$-ladder, towards showing that $K$ is not an $n$-ladder. Fix $x \in C$ and $p_0,p_1,\dots, p_{n-1} \in C$ distinct lower covers of $x$ in $C$. The set $S_x \coloneqq \{z \in K : b \le z \text{ and } \pi_L(z) = x\}$ is nonempty by definition of $C$, and it is a lower finite meet-subsemilattice because $\pi_L$ is a meet-homomorphism. Hence $S_x$ has a least element; let $y$ be that element. We claim that, for every $i < n$, there exists a lower cover $q_i$ of $y$ in $K$ such that $\pi_L(q_i) = p_i$. Fix an $i < n$. Pick some $q \in K$ such that $b \le q$ and $\pi_L(q) = p_i$. Since $\pi_L$ is a meet-homomorphism, $\pi_L(q \wedge y) = \pi_L(q) = p_i$. In particular, there exists some $q \in K$ such that $b \le q < y$ and $\pi_L(q) = p_i$. As $K$ is lower finite, we can pick a lower cover $q'$ of $y$ such that $q \le q'$. Clearly, $p_i =  \pi_L(q) \le \pi_L(q') \le \pi_L(y) = x$. Since $p_i$ is assumed to be a lower cover of $x$ in $C$, either $\pi_L(q') = p_i$ or $\pi_L(q') = x$. But in the latter case, we would contradict the minimality of $y$. Thus, $\pi_L(q') = p_i$, and indeed there exists a lower cover of $y$ whose image via $\pi_L$ is $p_i$. For each $i < n$ pick a lower cover $q_i$ of $y$ such that $\pi_L(q_i) = p_i$. Clearly, $x \not\le q_i$, as otherwise we would have $\pi_L(q_i) = x$. Let $r$ be a lower cover of $y$ such that $x \le r$. This $r$ is distinct from every $q_i$, because $x \not\le q_i$ for all $i < n$. Hence $y$ has at least $n+1$ lower covers, so $K$ is not an $n$-ladder.
\end{proof}

\begin{remark}\label{rmk:cofinalmeet}
As noted in the proof of Theorem~\ref{thm:characterization}, a cofinal meet-subsemilattice of a lower finite lattice is itself a lattice in the induced ordering. In particular, an $n$-ladder $(P, \le) $ is maximal if and only if it has no cofinal meet-subsemilattice $C \subseteq P$ whose elements have at most $n-1$ lower covers in $(C, \le)$.
\end{remark}

Note that Ditor's Theorem~\ref{thm:Ditormax} is a straightforward consequence of Theorem~\ref{thm:characterization}. In fact, every countable $n$-ladder has a cofinal chain, and a cofinal chain is a cofinal meet-subsemilattice which is a $1$-ladder in the induced ordering.

\section{Maximal ladders and Cohen reals}\label{sec:cohen}

In this section we prove Theorem~\ref{thm:main1}. Before proving it, let us recall some results of Wehrung from \cite{MR2609217}, recasting them via the notion of maximality:

\begin{lemma}[{Wehrung, \cite[Lemma 7.5]{MR2609217}}]\label{prop:orwehrung}
Given a lower finite lattice $L$, there exists a forcing of precaliber $\aleph_1$ that forces the existence of a cofinal meet-subsemilattice of $L$ which is a $2$-ladder.
\end{lemma}

The following theorem, due to Wehrung, follows from the lemma above. 
\begin{theorem}[{Wehrung, \cite[Theorem 7.9]{MR2609217}}]\label{thm:orwherung}
If $\mathsf{MA}_{\omega_1}(\aleph_1\text{-precaliber})$ holds, then there exists a $3$-ladder of cardinality $\aleph_2$.
\end{theorem}
\begin{proof}
We show something stronger: under $\mathsf{MA}_{\omega_1}(\aleph_1\text{-precaliber})$ there is no maximal $n$-ladder, with $n > 2$, of cardinality $\aleph_1$. This assertion is stronger because it implies, in particular, that every maximal $3$-ladder has cardinality $\aleph_2$. Indeed, by Ditor's Theorem~\ref{thm:Ditormax}, there are no countable maximal $3$-ladders, and therefore if maximal $3$-ladders cannot have cardinality $\aleph_1$, they all must have cardinality at least $\aleph_2$. But this actually implies that every maximal $3$-ladder has cardinality precisely $\aleph_2$, since $3$-ladders have cardinality at most $\aleph_2$ by Ditor's Theorem~\ref{thm:Ditor}\ref{thm:Ditor-1}.

Let $L$ be an $n$-ladder of cardinality $\aleph_1$ for some $n > 2$, towards showing that it is not maximal. By Lemma~\ref{prop:orwehrung}, there is an $\aleph_1$-precaliber forcing $\mathbb{P}$ and a $\mathbb{P}$-name $\dot{C}$ such that 
\[
\Vdash \dot{C} \text{ is a cofinal meet-subsemilattice of } L \text{ and it is a } 2\text{-ladder}. 
\]

For every $x \in L$, let 
\begin{multline*}
D_x \coloneqq \big\{p \in \mathbb{P} : \exists y \in L \ \big( x \le y \text{ and } p \Vdash ``y \in \dot{C}" \text{ and} \\p \text{ decides the (finite) set } \dot{C} \cap ({\downarrow} y)\big)\big\}.
\end{multline*}
Since $\mathbb{P}$ forces $\dot{C}$ to be cofinal, each $D_x$ is dense open. By $\mathsf{MA}_{\omega_1}(\aleph_1\text{-precaliber})$ we can fix a filter $G \subseteq \mathbb{P}$ such that $D_x \cap G \neq \emptyset$ for every $x \in L$. Let
\[
C' \coloneqq \big\{x \in L : \exists p \in G \ (p \Vdash x \in \dot{C})\big\}.
\]

We claim that $C'$ is a cofinal meet-subsemilattice of $L$ which is a $2$-ladder. The cofinality of $C'$ directly follows from $G$ intersecting every $D_x$. Now fix  $x \in C'$, and $p \in G \cap D_x$. It quickly follows that $p$ forces $C' \cap ({\downarrow} x) = \dot{C} \cap ({\downarrow} x)$. Moreover, since $p$ forces $\dot{C}$ to be a $2$-ladder, $x$ must have at most $2$ lower covers in $C'$.

 Finally, fix $x, y \in C'$. Let $p, q \in G$ such that $p$ forces $x \in \dot{C}$ and $q$ forces $y \in \dot{C}$. Let $r\in G$ be a lower bound of $p$ and $q$. Since $r$ forces both $x$ and $y$ to belong to $\dot{C}$ and forces $\dot{C}$ to be a meet-subsemilattice of $L$, it must also force $x \wedge y \in \dot{C}$. Thus, $x \wedge y \in C'$.

We conclude that $C'$ is a cofinal meet-subsemilattice of $L$ which is a $2$-ladder, and therefore $L$ is not maximal by Theorem~\ref{thm:characterization}.
\end{proof}

So Wehrung not only proved that, under $\mathsf{MA}_{\omega_1}(\aleph_1\text{-precaliber})$, there is a $3$-ladder of cardinality $\aleph_2$, but actually that every maximal $n$-ladder with $n > 2$ has cardinality at least $\aleph_2$ (cf. Theorem~\ref{thm:Ditormax}).

Our next proposition generalizes Wehrung's Lemma~\ref{prop:orwehrung} to higher cardinals (see Remark~\ref{rmk:Sk}). It is key to the proof of Theorem~\ref{thm:main1}. 


\begin{proposition}\label{prop:wehrung}
Given some $n > 0$ and a lower finite lattice $L$ of cardinality $\aleph_n$, $\text{\upshape Add}(\omega, \omega_n)$ forces the existence of a cofinal meet-subsemilattice of $L$ which is an $(n+1)$-ladder.
\end{proposition}
\begin{proof}
Fix a lower finite lattice $(L, \le)$ of cardinality $\aleph_n$. We first partition $L$ into countable pieces indexed by certain terminal nodes of a tree. Then, using this partition, we show that $\text{Add}(\omega, \omega_n)$ forces the existence of a cofinal meet-subsemilattice of $L$ which is an $(n+1)$-ladder in the induced ordering. 

Let
\begin{multline*}
T \coloneqq \big\{\emptyset\} \cup \{\langle\alpha_0, \alpha_1, \dots, \alpha_{m}\rangle : m < n \text{ and } \alpha_i < \omega_{n-i} \text{ for every } i \le m \text{ and } \\ \alpha_i \text{ is not a limit ordinal for every } i < m\big\}.
\end{multline*}
Given $\vec{\alpha}, \vec{\beta} \in T$, we say that $\vec{\beta}$ \emph{extends} $\vec{\alpha}$, in symbols $\vec{\alpha} \subseteq \vec{\beta}$, when $\vec{\beta} \upharpoonright |\vec{\alpha}| = \vec{\alpha}$. Moreover, we write $\vec{\alpha} <_{lex}^* \vec{\beta}$ to mean that $\vec{\alpha} <_{lex} \vec{\beta}$ and $\vec{\alpha} \nsubseteq \vec{\beta}$.

We call $\vec{\alpha} \in T$ \emph{nonlimit} if its last element is not a limit ordinal. Moreover, we call $\vec{\alpha} \in T$ \emph{maximal} if either it has length $n$ or its last element is a limit ordinal---i.e., if it has no proper extension in $T$.

\begin{claim}\label{prop:wehrung:claim:1}
There is a family $\langle I_{\vec{\alpha}} : \vec{\alpha} \in T\rangle$ of ideals  of $L$ such that:
\begin{enumerate}[label={\upshape (\arabic*)}]
\item $I_{\emptyset} = L$,
\item $I_{\vec{\alpha}} = \bigcup_{\beta < \omega_{n-|\vec{\alpha}|}} I_{\vec{\alpha}^\smallfrown \beta}$ if $\vec{\alpha}$ is not maximal,
\item $I_{\vec{\alpha}^\smallfrown \beta} \subseteq I_{\vec{\alpha}^\smallfrown \gamma}$ for every $\beta \le \gamma < \omega_{n-|\vec{\alpha}|}$ if $\vec{\alpha}$ is not maximal,
\item $I_{\vec{\alpha}^\smallfrown \lambda} = \bigcup_{\beta < \lambda} I_{\vec{\alpha}^\smallfrown \beta}$ if $\lambda$ is a limit ordinal and $\vec{\alpha}$ is not maximal,
\item $I_{\vec{\alpha}} \nsubseteq \bigcup_{\vec{\beta} <_{lex}^* \vec{\alpha}} I_{\vec{\beta}}$ if $\vec{\alpha}$ is nonlimit,
\item $I_{\vec{\alpha}} \supseteq I_{\vec{\gamma}}$ for every nonlimit $\vec{\gamma} \in T$ of length $n$ such that $\vec{\gamma} <_{lex} \vec{\alpha}$ and $(I_{\vec{\alpha}} \cap I_{\vec{\gamma}}) \setminus \bigcup_{\vec{\beta} <_{lex}^* \vec{\gamma}} I_{\vec{\beta}} \neq \emptyset$,
\item $|I_{\vec{\alpha}}| = \aleph_{n-|\vec{\alpha}|}$.
\end{enumerate}
\end{claim}
\begin{proof}
We define $I_{\vec{\alpha}}$ by induction on the lexicographic ordering of the $\vec{\alpha}$s in $T$. As soon as we define $I_{\vec{\alpha}}$ for some $\vec{\alpha} \in T$ we fix an enumeration $\langle x_{\vec{\alpha}}(\delta) : \delta < \omega_{n-|\vec{\alpha}|} \rangle$ of $I_{\vec{\alpha}}$. 

First let $I_{\emptyset} = L$. Fix some nonempty $\vec{\alpha} \in T$ and let $k < n$ be such that $k+1 = |\vec{\alpha}|$. Suppose that we defined $I_{\vec{\beta}}$ for all $\vec{\beta}$ in $T$ with $\vec{\beta} <_{lex} \vec{\alpha}$ such that (3)-(7) hold below $\vec{\alpha}$---i.e., the statements (3)-(7) hold when the elements of $T$ mentioned are $<_{lex} \vec{\alpha}$. Moreover, we also assume that the following weakening of (2) holds:
\begin{enumerate}
\item[(2')] $I_{\vec{\beta}} \supseteq I_{\vec{\gamma}}$ for all $\vec{\beta},\vec{\gamma} <_{lex} \vec{\alpha}$ such that $\vec{\gamma}$ extends $\vec{\beta}$.
\end{enumerate} 
We now define $I_{\vec{\alpha}}$. If $\alpha_k$ is a limit ordinal, let $I_{\vec{\alpha}} = \bigcup_{\beta < \alpha_k} I_{(\vec{\alpha} \upharpoonright k)^\smallfrown \beta}$. Now suppose that $\alpha_k$ is not a limit ordinal. First let us fix a subset $X \subseteq I_{\vec{\alpha} \upharpoonright k}$ depending on the value of $\alpha_k$:\vspace{0.5em}

\noindent \underline{$\alpha_k = 0$}:  for every $\vec{\beta}$, $\vec{\beta} <^*_{lex} \vec{\alpha}$ if and only if $\vec{\beta} <^*_{lex} \vec{\alpha} \upharpoonright k$. By (5) and (7), $I_{\vec{\alpha} \upharpoonright k} \nsubseteq \bigcup_{\vec{\beta} <^*_{lex} \vec{\alpha}}I_{\vec{\beta}}$, and $|I_{\vec{\alpha} \upharpoonright k}| = \aleph_{n-k}$. Thus, we can pick $X\subseteq I_{\vec{\alpha} \upharpoonright k}$ such that $|X| = \aleph_{n-k-1}$ and  $X \nsubseteq \bigcup_{\vec{\beta} <^*_{lex} \vec{\alpha}} I_{\vec{\beta}}$. \vspace{0.5em}

\noindent \underline{$\alpha_k > 0$}: it follows that 
\[
\{\vec{\beta} \in T : \vec{\beta} <^*_{lex} \vec{\alpha}\} = \{(\vec{\alpha} \upharpoonright k)^\smallfrown \gamma : \gamma < \alpha_k\} \cup \{\vec{\beta} \in T : \vec{\beta} <^*_{lex} \vec{\alpha}\upharpoonright k\}.
\]
 By (3), $I_{(\vec{\alpha} \upharpoonright k)^\smallfrown \gamma} \subseteq I_{(\vec{\alpha} \upharpoonright k)^\smallfrown (\alpha_k-1)}$ for every $\gamma < \alpha_k$. By (2') and (7), $I_{\vec{\alpha} \upharpoonright k} \supseteq I_{(\vec{\alpha} \upharpoonright k)^\smallfrown (\alpha_k - 1)}$ and $|I_{\vec{\alpha} \upharpoonright k}| > |I_{(\vec{\alpha} \upharpoonright k)^\smallfrown (\alpha_k - 1)}|$. Moreover, by (5), $I_{\vec{\alpha} \upharpoonright k} \nsubseteq \bigcup_{\vec{\beta} <^*_{lex} \vec{\alpha} \upharpoonright k} I_{\vec{\beta}}$. The set $I_{\vec{\alpha} \upharpoonright k} \setminus \bigcup_{\vec{\beta} <^*_{lex} \vec{\alpha} \upharpoonright k} I_{\vec{\beta}}$, being nonempty and upward closed, must have the same cardinality as $I_{\vec{\alpha} \upharpoonright k}$, that is $\aleph_{n-k}$. Indeed, in general, an infinite lower finite poset has the same cardinality as all its cofinal subsets. Thus, we can fix some $z$ belonging to the set
\[
I_{\vec{\alpha} \upharpoonright k} \setminus \Big ( I_{(\vec{\alpha} \upharpoonright k)^\smallfrown (\alpha_k - 1)} \cup \bigcup_{\vec{\beta} <^*_{lex} \vec{\alpha} \upharpoonright k} I_{\vec{\beta}} \Big).
\] 
Now, let $\delta < \omega_{n-k}$ be the least ordinal such that $
x_{\vec{\alpha} \upharpoonright k}(\delta) \not\in I_{(\vec{\alpha} \upharpoonright k)^\smallfrown (\alpha_k - 1)}$.
Set $X = I_{(\vec{\alpha} \upharpoonright k)^\smallfrown (\alpha_k - 1)} \cup \{z, x_{\vec{\alpha} \upharpoonright k}(\delta)\}$.\vspace{0.5em}

 Now that we have fixed $X$, let us define a $\subseteq$-increasing sequence $\langle F_m : m\in\omega\rangle$ of ideals of $L$ as follows: first let $F_0 = \mathrm{id}(X)$,  that is, $F_0$ is the ideal generated by $X$; now, if $F_m$ is defined, let $R_m$ be the set of all those $\vec{\gamma} \in T$ of length $n$ such that $\vec{\gamma} <_{lex} \vec{\alpha}$ and $(F_m \cap I_{\vec{\gamma}}) \setminus \bigcup_{\vec{\beta} <^*_{lex} \vec{\gamma}} I_{\vec{\beta}} \neq \emptyset$; let $F_{m+1} = \mathrm{id}(F_m \cup \bigcup_{\vec{\gamma} \in R_m} I_{\vec{\gamma}})$. Finally, let $I_{\vec{\alpha}} = \bigcup_{m\in\omega} F_m$. 
 
This ends the definition of $I_{\vec{\alpha}}$. Let us check that (2') and (3)-(7) still hold when all the members of $T$ mentioned are $\le_{lex} \vec{\alpha}$. Properties (2') and (3)-(6) follow directly from the construction of $I_{\vec{\alpha}}$. 

We are left to prove that (7) holds. It suffices to show $|I_{\vec{\alpha}}| = \aleph_{n-|\vec{\alpha}|} = \aleph_{n-k-1}$. If $\alpha_k$ is a limit ordinal, then it follows directly that $|I_{\vec{\alpha}}| = |\alpha_k| \cdot \sup_{\beta < \alpha_k} |I_{(\vec{\alpha} \upharpoonright k)^\smallfrown \beta}| = |\alpha_k| \cdot \aleph_{n-k-1}$, but since $\alpha_k < \omega_{n-k}$, we conclude $|I_{\vec{\alpha}}| = \aleph_{n-k-1}$. Now suppose that $\alpha_k$ is not a limit ordinal. It suffices to prove that $|F_m| = \aleph_{n-k-1}$ for each $m \in \omega$. Since $|X| = \aleph_{n-k-1}$ by choice and $L$ is lower finite, it follows that $|F_0| = \aleph_{n-k-1}$. Now, if $|F_m| = \aleph_{n-k-1}$, in order to prove that $|F_{m+1}| = |F_m|$ it suffices to show that $|R_m| \le |F_m|$. Indeed, it is clear that for each $x \in L$ there exists at most one $\vec{\gamma} \in T$ of length $n$ such that $x \in I_{\vec{\gamma}} \setminus \bigcup_{\vec{\beta} <^*_{lex} \vec{\gamma}} I_{\vec{\beta}}$. Then, it follows directly that $|R_m| \le |F_m|$. 

This ends the inductive definition of the $I_{\vec{\alpha}}$s. Property (1) holds by construction and we have shown that (3)-(7) also hold. Regarding (2), note that it holds by the way we chose $\delta$ in the case where $\alpha_k$ is a successor ordinal. 
\end{proof}

Fix a family satisfying the statement of Claim~\ref{prop:wehrung:claim:1}.
Given some $\vec{\alpha} \in T$, for notational clarity let us denote $\pi_{I_{\vec{\alpha}}}$ by $\pi_{\vec{\alpha}}$. Let $T^s(n)$ be the set of all the nonlimit elements of $T$ which have length $n$. For each $\vec{\alpha} \in T^s(n)$, let
\[
J_{\vec{\alpha}} \coloneqq I_{\vec{\alpha}} \setminus \bigcup_{\vec{\beta} <^*_{lex} \vec{\alpha}}I_{\vec{\beta}}.
\]
By (5) and (7), the $J_{\vec{\alpha}}$s are nonempty, countable, and pairwise disjoint.

By (1) and (2), for every $x \in L$ there exists $\vec{\alpha} \in T$ of length $n$ such that $x \in I_{\vec{\alpha}}$. Let $\vec{\alpha}(x)$ be the $\le_{lex}$-least $\vec{\alpha} \in T$ of length $n$ such that $x \in I_{\vec{\alpha}}$. By (4), $\vec{\alpha}(x) \in T^s(n)$. Also,  clearly $x \in J_{\vec{\alpha}(x)}$. In particular, we have
\begin{equation}\label{eq:wehrung}
L = \bigcup \big\{J_{\vec{\alpha}} : \vec{\alpha} \in T^s(n)\big\},
\end{equation}
that is to say, the $J_{\vec{\alpha}}$s form a partition of $L$.

Notice that for every $x,y \in L$, if $x \le y$, then $\vec{\alpha}(x) \le_{lex}\vec{\alpha}(y)$. Finally, observe that, by (6), for every $x,y \in L$, if $x \le y$, then $I_{\vec{\alpha}(x)} \subseteq I_{\vec{\alpha}(y)}$.

Consider the forcing notion $\mathbb{P}$ defined as
\[
\mathbb{P} \coloneqq \big\{p:D \rightarrow L : D \subseteq T^s(n) \times \omega \text{ is finite and } \forall (\vec{\alpha}, m) \in D \ (p(\vec{\alpha}, m) \in J_{\vec{\alpha}})\big\},
\]
and ordered by reverse extension: $q \le p$ if $\mathrm{dom}(q) \supseteq \mathrm{dom}(p)$ and $q \upharpoonright \mathrm{dom}(p) = p$. Since $T^s(n)$ has cardinality $\aleph_n$ and each $J_{\vec{\alpha}}$ is countable, it is easy to see that $\mathbb{P} \cong \text{Add}(\omega, \omega_n)$, i.e., $\mathbb{P}$ is isomorphic to Cohen's forcing $\text{Add}(\omega, \omega_n)$.

For each $p \in \mathbb{P}$, we now define a subset $C_p \subseteq L$. Fix $p \in \mathbb{P}$. By \eqref{eq:wehrung}, it suffices to define $C_p \cap J_{\vec{\alpha}}$ for each  $\vec{\alpha} \in T^s(n)$. We do so by induction on the lexicographic ordering of $T^s(n)$. 

Fix some $\vec{\alpha} \in T^s(n)$ and suppose that $C_p \cap J_{\vec{\beta}}$ is defined for every $\vec{\beta} \in T^s(n)$ with $\vec{\beta} <_{lex} \vec{\alpha}$, towards defining $C_p \cap J_{\vec{\alpha}}$. For every $x \in J_{\vec{\alpha}}$, we let $x \in C_p$ if and only if there exists some $m \in \omega$ such that:
\begin{enumerate}[label=(\alph*)]
\item $p (\vec{\alpha}, m) = x$, and
\item for all $k < m$,  $(\vec{\alpha}, k) \in \mathrm{dom}(p)$ and $p(\vec{\alpha}, k) \le x$, and
\item for all $\vec{\gamma} \in T$ such that $\vec{\gamma} <_{lex}^* \vec{\alpha}$ ,  $\pi_{\vec{\gamma}}(x) \in C_p$.
\end{enumerate}
Note that if $p,q \in \mathbb{P}$ and $p \le q$, then $C_q \subseteq C_p$. Now, given any filter $G \subseteq \mathbb{P}$, let  $C_G \coloneqq \bigcup_{p \in G} C_p$. We first claim that, for any filter $G$, $C_G$ is a meet-subsemilattice of $L$ whose elements have at most $n+1$ lower covers.

\begin{claim}
Let $G\subseteq \mathbb{P}$ be a filter. Then, every element of $C_G$ has at most $n+1$ lower covers in $C_G$.
\end{claim}
\begin{proof}
Pick $x \in C_G$ and let $\vec{\alpha} = \vec{\alpha}(x)$. It suffices to prove that there exists a set $S \subseteq C_G$ of size at most $n+1$ such that $z < x$ for all $z \in S$ and, for all $y \in C_G$ with $y < x$, there is some $z \in S$ with $y \le z$.

First let 
\[
S' \coloneqq \big\{\pi_{(\vec{\alpha} \upharpoonright i)^\smallfrown (\alpha_i - 1)}(x) : i < n \text{ and } \alpha_i > 0\big\}.
\]
Now notice that, by (b), the set $C_G \cap J_{\vec{\alpha}}$ is a chain in $(L, \le)$. If there is $c \in C_G \cap J_{\vec{\alpha}}$ with $c < x$, let $c_0$ be the $\le$-greatest such $c$ and let $S = S' \cup \{c_0\}$; otherwise let $S = S'$.

We claim that the set $S$ satisfies the desired property. Observe first that $S \subseteq C_G$. Indeed, if $c_0$ is defined, then $c_0 \in C_G$ by definition. The other elements of $S$ belong to $C_G$ by (c). Moreover, all the elements of $S$ are strictly below $x$. The element $c_0$ is strictly below $x$ by definition. The other elements also cannot coincide with $x$ because, by the $\le_{lex}$-minimality of $\vec{\alpha}$, $x \not\in I_{(\vec{\alpha} \upharpoonright i)^\smallfrown (\alpha_i - 1)}$ for every  $i < n$ with $\alpha_i > 0$.

Now pick any $y \in C_G$ with $y < x$, towards showing that $y \le z$ for some $z \in S$. Let $\vec{\beta} = \vec{\alpha}(y)$. We already noted that $\vec{\beta} \le_{lex} \vec{\alpha}$. 

If $\vec{\beta} = \vec{\alpha}$, then $c_0$ is defined and $y \le c_0 \in S$ by definition of $c_0$. If $\vec{\beta} <_{lex} \vec{\alpha}$, let $i$ be such that $\vec{\beta} \upharpoonright i = \vec{\alpha} \upharpoonright i$ and $\beta_i < \alpha_i$. By (2), $I_{\vec{\beta}} \subseteq I_{\vec{\beta} \upharpoonright (i+1)}$. By (3), $I_{\vec{\beta} \upharpoonright (i+1)} \subseteq I_{(\vec{\alpha} \upharpoonright i)^\smallfrown (\alpha_i-1)}$. Thus, $y \in I_{(\vec{\alpha} \upharpoonright i)^\smallfrown (\alpha_i-1)}$. We conclude $y \le \pi_{(\vec{\alpha} \upharpoonright i)^\smallfrown (\alpha_i-1)}(x) \in S$, and we are done.
\end{proof}

\begin{claim}
Let $G\subseteq \mathbb{P}$ be a filter. Then, $C_G$ is a meet-subsemilattice of $L$.
\end{claim}
\begin{proof}
Fix $x,y \in C_G$. Let $\vec{\alpha} = \vec{\alpha}(x \wedge y)$.  By (c), both $\pi_{\vec{\alpha}} (x)$ and $\pi_{\vec{\alpha}} (y)$ belong to $C_G$. Moreover, by Lemma~\ref{lemma:meetproj}, $x \wedge y = \pi_{\vec{\alpha}} (x) \wedge \pi_{\vec{\alpha}} (y)$. In particular, $ \pi_{\vec{\alpha}}(x)$ and $ \pi_{\vec{\alpha}}(y)$ actually lie in $J_{\vec{\alpha}}$, not just in $I_{\vec{\alpha}}$.

As we already noted, the set $C_G \cap J_{\vec{\alpha}}$ is a chain. Thus, $\pi_{\vec{\alpha}}(x)$ and $\pi_{\vec{\alpha}}(y)$ must be comparable. But this means that $x \wedge y = \pi_{\vec{\alpha}}(x) \wedge \pi_{\vec{\alpha}}(y) \in \pi_{\vec{\alpha}}``\{x,y\} \subseteq C_G$. Therefore, $C_G$ is a meet-subsemilattice of $L$.
\end{proof}
  
Now we prove that if $G\subseteq \mathbb{P}$ is a $V$-generic filter, the set $C_G$ is also cofinal in $L$. Note that, by Remark~\ref{rmk:cofinalmeet}, this implies that $C_G$ is an $(n+1)$-ladder in the induced ordering.
\begin{claim}
Let $G\subseteq \mathbb{P}$ be a $V$-generic filter. Then, $C_G$ is cofinal in $L$.
\end{claim}
\begin{proof}
Fix $x \in L$ and $p \in \mathbb{P}$. It suffices by density to show that there are $q \le p$ and $y \in C_q$ with $x \le y$.  

First pick any $p' \le p$ such that, for each $(\vec{\alpha}, m) \in \mathrm{dom}(p')$, $(\vec{\alpha}, k) \in \mathrm{dom}(p')$ for every $k \le m$. 
Let
\[
y \coloneqq x \vee \bigvee \big\{p' (\vec{\alpha}, m) : (\vec{\alpha}, m) \in \mathrm{dom}(p')\big\}.
\]
Let $R = \{\vec{\alpha}(z) : z \le y\}$. Since $L$ is lower finite, $R$ is finite. For every $\vec{\alpha} \in R$, let $m_{\vec{\alpha}}$ be the least $m \in\omega$ such that $(\vec{\alpha}, m) \not\in \mathrm{dom}(p')$. Finally, let 
\[
q \coloneqq p' \cup \big\{((\vec{\alpha}, m_{\vec{\alpha}}), \pi_{\vec{\alpha}}(y)) : \vec{\alpha} \in R\big\}.
\]
Clearly, $q \le p' \le p$. We are done once we prove $y \in C_q$.  We do so by showing that $\pi_{\vec{\alpha}}(y) \in C_q$ for all $\vec{\alpha} \in R$ by induction on the lexicographic ordering of $R$---indeed, $\vec{\alpha}(y) \in R$ and $\pi_{\vec{\alpha}(y)}(y) = y$. 

 Fix $\vec{\alpha} \in R$ and suppose that $\pi_{\vec{\beta}} (y) \in C_q$ for all $\vec{\beta} \in R$ with $\vec{\beta} <_{lex} \vec{\alpha}$. We want to prove that $\pi_{\vec{\alpha}} (y) \in C_q$. In particular, we show that $\pi_{\vec{\alpha}} (y)$ satisfies (a)-(c) in the definition of $C_q \cap J_{\vec{\alpha}}$ for $m = m_{\vec{\alpha}}$.
 
 By definition of $q$, $q(\vec{\alpha}, m_{\vec{\alpha}}) = \pi_{\vec{\alpha}}(y)$. In particular, $\pi_{\vec{\alpha}}(y)$ satisfies (a). Moreover, since  $y \ge q(\vec{\alpha}, k) = p'(\vec{\alpha}, k) $ for all $k  < m_{\vec{\alpha}}$ by definition of $y$, we conclude that $\pi_{\vec{\alpha}}(y) \ge q(\vec{\alpha}, k)$ for all $k  < m_{\vec{\alpha}}$. Thus, $\pi_{\vec{\alpha}}(y)$ also satisfies (b). 
 
Finally, towards showing that $\pi_{\vec{\alpha}}(y)$ satisfies also (c), pick $\vec{\gamma} \in T$ with $\vec{\gamma} <^*_{lex} \vec{\alpha}$. We need to prove $\pi_{\vec{\gamma}} \circ \pi_{\vec{\alpha}}(y) \in C_q$. Denote $\pi_{\vec{\gamma}} \circ \pi_{\vec{\alpha}}(y) $ by $z$. Then
\begin{equation}\label{eq:finalcohen}
\begin{split}
z &= \pi_{\vec{\alpha}(z)} \circ \pi_{\vec{\gamma}} \circ \pi_{\vec{\alpha}} (y)\\
&= \pi_{\vec{\alpha}(z)} \circ \pi_{\vec{\alpha}} (y)\\
&= \pi_{\vec{\alpha}(z)}(y).
\end{split}
\end{equation}
The first equality follows trivially, as $z = \pi_{\vec{\alpha}(z)} (z)$. To argue the second equality, note that $I_{\vec{\alpha}(z)} \subseteq I_{\vec{\gamma}}$. Indeed, there are two cases: either $\vec{\gamma} \subseteq \vec{\alpha}(z)$, and in this case $I_{\vec{\alpha}(z)} \subseteq I_{\vec{\gamma}}$ by (2); or $\vec{\alpha}(z) <_{lex} \vec{\gamma}$ and then by (6) we conclude $I_{\vec{\alpha}(z)} \subseteq I_{\vec{\gamma}}$. So in either case $I_{\vec{\alpha}(z)} \subseteq I_{\vec{\gamma}}$. Then, $\pi_{\vec{\alpha}(z)} \circ \pi_{\vec{\gamma}} = \pi_{\vec{\alpha}(z)}$ follows from Lemma~\ref{lemma:concproj} since $I_{\vec{\alpha}(z)} \subseteq I_{\vec{\gamma}}$. Finally, the third equality also follows from Lemma~\ref{lemma:concproj}. Indeed, $\vec{\alpha}(z)  <_{lex} \vec{\alpha}$, and as such we have $I_{\vec{\alpha}(z)} \subseteq I_{\vec{\alpha}}$ by (6).  

Thus, from \eqref{eq:finalcohen} we conclude that $z = \pi_{\vec{\alpha}(z)} (y)$. But $\vec{\alpha}(z) \in R$ and $\vec{\alpha}(z) <_{lex} \vec{\alpha}$, and thus by induction hypothesis $\pi_{\vec{\alpha}(z)} (y) \in C_q$. We conclude $z \in C_q$ and, overall, that $\pi_{\vec{\alpha}}(y)$ satisfies also (c).
\end{proof}
\end{proof}

\begin{remark}\label{rmk:Sk}
When $n = 1$, the argument used in the proof of Proposition~\ref{prop:wehrung} is essentially the same as the one employed by Wehrung to prove Lemma~\ref{prop:orwehrung} (\cite[Lemma 7.5]{MR2609217}). Indeed, the sequence of countable ideals $\langle I_\alpha \mid \alpha < \omega_1 \rangle$ given by Claim~\ref{prop:wehrung:claim:1} canonically induces what Wehrung calls a \emph{locally countable norm} on $L$ \cite[Definitions 6.1 and 7.3]{MR2609217}. More importantly, the map $p \mapsto C_p$ defined in the proof of Proposition~\ref{prop:wehrung} is a forcing projection (in the sense of \cite[Definition 5.2]{MR2768691}) from a dense subset of $\mathbb{P}$ onto Wehrung's forcing $\text{Sk}$ \cite[Definition 7.1]{MR2609217} that witnesses Lemma~\ref{prop:orwehrung}. In particular, the cofinal meet-semilattice $C_G$ induced by a $V$-generic filter $G \subseteq \mathbb{P}$ is $V$-generic for $\text{Sk}$.
\end{remark}
Now we are ready to prove Theorem~\ref{thm:main1}. 

\begin{theorem}[Theorem~\ref{thm:main1}]\label{thm:wehrung}
For every $m > 1$, $\text{\upshape Add}(\omega, \omega_{m})$ forces that every maximal $n$-ladder has cardinality at least $\aleph_{\min\{n-1, m\}}$. 
\end{theorem}
\begin{proof}
Let $G$ be a $V$-generic filter for $\text{Add}(\omega, \omega_{m})$. By identifying $\text{Add}(\omega, \omega_{m})$ with the poset of all finite partial maps from $\omega_{m}$ to $\{0,1\}$, given a set $X\subseteq \omega_{m}$, we let $G \upharpoonright X = \{p \in G : \mathrm{dom}(p) \subseteq X\}$.

It suffices to prove that in $V[G]$ no maximal $n$-ladder has cardinality less than $\aleph_{\min\{n-1, m\}}$. Pick an infinite $n$-ladder $L$ in $V[G]$ and $|L| = \aleph_k < \aleph_{\min\{n-1, m\}}$, towards showing that $L$ is not maximal.  

By a routine argument, there exists a set $X \subseteq \omega_{m}$ with $|X| \le \aleph_k = |L|$ such that $L \in V[G \upharpoonright X]$.

 Moreover, $V[G] = V[G \upharpoonright X][G \upharpoonright (\omega_{m} \setminus X)]$ and $G \upharpoonright (\omega_{m} \setminus X)$ is $V[G \upharpoonright X]$-generic for $\text{Add}(\omega,\omega_{m} \setminus X) \cong \text{Add}(\omega,\omega_{m})$. As $k < m$, $\text{Add}(\omega, \omega_k)$ is a complete subforcing of $\text{Add}(\omega,\omega_{m})$, and we conclude by Proposition~\ref{prop:wehrung} that in $V[G]$ there exists a cofinal meet-subsemilattice of $L$ which is also an $(k+1)$-ladder in the induced ordering. But since $k < n-1$, this means, in particular, that in $V[G]$ there exists a cofinal meet-subsemilattice of $L$ which is an $(n-1)$-ladder. By Theorem~\ref{thm:characterization} we conclude that $L$ is not maximal in $V[G]$.
\end{proof}

The same reasoning used to prove Theorem~\ref{thm:wehrung} shows the following corollary.
\begin{corollary}[Corollary~\ref{cor:main1}]
$\text{ \upshape Add}(\omega, \omega_\omega)$ forces that every maximal $n$-ladder has cardinality $\aleph_{n-1}$ for every $n > 0$.
\end{corollary}

Combining Proposition~\ref{prop:wehrung} together with the proof of Theorem~\ref{thm:orwherung} allows us to cast the previous two results (i.e., Theorem~\ref{thm:main1} and Corollary~\ref{cor:main1}) in terms of forcing axioms. In particular, the forced statements of the previous two results are actually \emph{implied} by certain restrictions of Martin's Axiom: $\mathsf{MA}_{\omega_{m}}(\text{Add}(\omega, \omega_{m}))$---i.e., $\mathsf{MA}_{\omega_m}$ restricted to the single poset $\text{Add}(\omega, \omega_{m})$---implies that every maximal $n$-ladder has cardinality at least $\aleph_{\min\{n-1,m+1\}}$; therefore, $\mathsf{MA}_{<\omega_\omega}(\text{Add}(\omega, \omega_{\omega}))$ implies that every maximal $n$-ladder has cardinality $\aleph_{n-1}$. 

\section{A maximal $3$-ladder of cardinality $\aleph_1$}\label{sec:dominating}

In the previous section, we showed that $\text{Add}(\omega, \omega_2)$ forces the nonexistence of maximal $3$-ladders of cardinality $\aleph_1$. The natural question is: is the existence of a maximal $3$-ladder of cardinality $\aleph_1$ consistent? In this section we prove Theorem~\ref{thm:main2}, which answers this question in the positive. We begin by recalling the \emph{dominating number} $\mathfrak{d}$, a classical cardinal characteristic of the continuum. 

Given $f,g \in {}^\omega \omega$, we say that $g$ \emph{dominates} $f$, in symbols $f <^* g$, if for all but finitely many integers $k \in \omega$, $f(k) < g(k)$. It is easy to see that the binary relation $<^*$ is a strict partial order and is $\sigma$-directed---i.e., every countable subset of ${}^\omega \omega$ has a $<^*$-upper bound. Replacing $<$ by $\le$ in the definition of $<^*$ yields a preorder on ${}^\omega \omega$ denoted by $\le^*$.

A family $\mathcal{D} \subseteq {}^\omega \omega$ is \emph{dominating} if  every element of ${}^\omega \omega$ is dominated by some element of $\mathcal{D}$. The dominating number is the smallest cardinality of a dominating family, i.e.,
\[
\mathfrak{d} = \min \{|\mathcal{D}| : \mathcal{D}\subseteq {}^\omega\omega \text{ is a dominating family}\}.
\]
Clearly $\mathfrak{d} \le 2^{\aleph_0}$, since ${}^\omega \omega$ is itself a dominating family. Moreover, since $<^*$ is $\sigma$-directed, it follows that $\aleph_1 \le \mathfrak{d}$. In particular, the continuum hypothesis implies $\mathfrak{d} = \aleph_1$. We refer the interested reader to \cite{MR2768685} and \cite{MR4917577} for a comprehensive treatment of this and other classical cardinal characteristics of the continuum.

Now we introduce a class of $3$-ladders whose analysis is the main focus of this section. These $3$-ladders are closely related to the dominating number, and in studying this connection, we will prove Theorem~\ref{thm:main2}.

Fix a map $\varrho: [\omega_1]^2 \rightarrow {}^\omega \omega$. Given $\alpha < \beta < \omega_1$, we simply write $\varrho(\alpha, \beta)$ instead of $\varrho(\{\alpha, \beta\})$. Denote the set $\{0\} \cup (\omega_1 \times \omega \times \omega)$ by $K$.  Every map $\varrho: [\omega_1]^2 \rightarrow {}^\omega \omega$ induces the following binary relation $\trianglelefteq_\varrho$ on $K$: for all $(\alpha, n, m), (\beta, n',m') \in K$,
\begin{enumerate}
\item $0 \trianglelefteq_\varrho (\alpha, n, m)$, and
\item $(\alpha, n, m) \trianglelefteq_\varrho (\beta, n',m')$ if and only if $\alpha \le \beta$  and $n \le n'$ and $m \le m'$ and $\varrho(\alpha, \beta)(n) \le m'$,
\end{enumerate}
where we have implicitly extended the domain $\varrho$ by imposing $\varrho(\alpha, \alpha) = \langle 0\rangle^\omega$ for all $\alpha < \omega_1$. Note that $0$ is the $\trianglelefteq_\varrho$-least element of $K$ and that, for each $\alpha < \omega_1$, the map $(\alpha, n, m) \mapsto (n,m)$ is an isomorphism between  $(\{\alpha\} \times \omega \times \omega, \trianglelefteq_\varrho)$ and  $\omega \times \omega$ with its usual product ordering.

In the remainder of this section, we prove the following result. Theorem~\ref{thm:main2} follows immediately from it.
\begin{theorem}\label{thm:dominatingmain}
$\mathfrak{d} = \aleph_1$ holds if and only if there exists a map $\varrho:[\omega_1]^2 \rightarrow {}^\omega \omega$ such that $(K, \trianglelefteq_\varrho)$ is a maximal $3$-ladder.
\end{theorem}

First let us prove the following lemma, that characterizes the properties of $\varrho$ which are sufficient and necessary for $(K, \trianglelefteq_\varrho)$ to be a join-semilattice.

\begin{lemma}\label{lemma:Kjoin}
Given a map $\varrho:[\omega_1]^2 \rightarrow {}^\omega \omega$, $(K, \trianglelefteq_\varrho)$ is a join-semilattice if and only if $\varrho$ satisfies the following properties: for every $\alpha < \beta < \gamma$ and for every $n\in\omega$,
\begin{enumerate}[label={\upshape ($\varrho$\arabic*)}]
\item $\varrho(\alpha, \beta)$ is non-decreasing,
\item $\varrho(\alpha, \gamma)(n) \le \max \{\varrho(\alpha, \beta)(n), \varrho(\beta, \gamma)(n)\}$,
\item $\varrho(\alpha, \beta)(n) \le \max \{\varrho(\alpha, \gamma)(n), \varrho(\beta, \gamma)(n)\}$,
\item $\varrho(\beta, \gamma)(n) \le \max \{\varrho(\alpha, \gamma)(n), \varrho(\beta, \gamma)(0)\}$.
\end{enumerate}
\end{lemma}
\begin{proof}
Let us fix a map $\varrho:[\omega_1]^2 \rightarrow {}^\omega \omega$ and also drop the subscript from the binary relation $\trianglelefteq_\varrho$ since the map $\varrho$ is fixed. Let us first prove that ($\varrho$1)-($\varrho$4) are sufficient for $(K, \trianglelefteq)$ to be a join-semilattice.

The ordering $\trianglelefteq$ is clearly reflexive and antisymmetric. Let us first prove that it is also transitive. Fix $(\alpha, n, m), (\beta, n', m'), (\gamma, n'',m'') \in K$ with $(\alpha, n, m) \trianglelefteq (\beta, n', m') \trianglelefteq (\gamma, n'',m'')$. Clearly, $\alpha \le \beta \le \gamma$ and $n \le n' \le n''$ and $m \le m' \le m''$. 

If $\alpha \le \beta = \gamma$, then the claim follows directly from the definition of $\trianglelefteq$. If $\alpha = \beta < \gamma$, then the following hold:
\[
\varrho(\alpha, \gamma)(n) = \varrho(\beta, \gamma)(n) \le \varrho(\beta, \gamma)(n') \le m'',
\]
where the first inequality follows from ($\varrho$1) and the second one from assuming $(\beta, n', m') \trianglelefteq (\gamma, n'',m'')$. In particular, we conclude $(\alpha, n, m) \trianglelefteq (\gamma, n'', m'')$.

If $\alpha  < \beta < \gamma$, we have
\begin{align*}
\varrho(\alpha, \gamma)(n) &\le \max \{\varrho(\alpha, \beta)(n) , \varrho(\beta, \gamma)(n)\}\\
&\le \max \{\varrho(\alpha, \beta)(n), \varrho(\beta, \gamma)(n')\} \le \max\{m', m''\} = m'',
\end{align*}
where the first inequality follows from ($\varrho$2) and the second one from ($\varrho$1). In particular, we have $(\alpha, n, m) \trianglelefteq (\gamma, n'', m'')$.  Thus, $\trianglelefteq$ is transitive and $(K, \trianglelefteq)$ is a poset.

 Let us prove now that $(K, \trianglelefteq)$ is a join-semilattice. Pick $(\alpha, n, m), (\beta, n', m') \in K$ with $\alpha \le \beta$. We claim that 
\begin{equation}\label{eq:leastupperbound}
\big(\beta, \max\{n, n'\}, \max\{m, m', \varrho(\alpha, \beta)(n)\}\big)
\end{equation}
is the $\trianglelefteq$-least upper bound of $(\alpha, n, m)$ and $(\beta, n', m')$. Clearly, \eqref{eq:leastupperbound} is a $\trianglelefteq$-upper bound of $(\alpha, n, m)$ and $(\beta, n', m')$, by definition of $\trianglelefteq$. Now fix some $(\gamma, n'', m'')$ with $(\alpha, n, m), (\beta, n', m') \trianglelefteq (\gamma, n'', m'')$, towards showing that \eqref{eq:leastupperbound} is  $\trianglelefteq (\gamma, n'', m'')$. By definition of $\trianglelefteq$, we must have $\beta \le \gamma$ and $\max \{n, n'\} \le n''$ and 
\[
\max \{m, m', \varrho(\alpha, \gamma)(n), \varrho(\beta, \gamma)(n')\} \le m''.
\] 
First note that
\begin{equation}\label{eq:dominating1}
\varrho(\beta, \gamma)(n) \le \max \{\varrho(\alpha, \gamma)(n), \varrho(\beta, \gamma)(0)\} \le \max \{\varrho(\alpha, \gamma)(n), \varrho(\beta, \gamma)(n')\},
\end{equation}
where the first inequality follows from ($\varrho$4) and the second one from ($\varrho$1).  Now consider $\varrho(\alpha, \beta)(n)$. We have
\begin{align*}
\varrho(\alpha, \beta)(n) &\le \max \{\varrho(\alpha, \gamma)(n), \varrho(\beta, \gamma)(n)\}\\ &\le \max \{\varrho(\alpha, \gamma)(n), \varrho(\beta, \gamma)(n')\} \le m'',
\end{align*}
where the first inequality follows from ($\varrho$3) and the second one from \eqref{eq:dominating1}. Moreover, we also have
\[
\varrho(\beta, \gamma)(\max\{n,n'\}) \le \max \{\varrho(\alpha, \gamma)(n), \varrho(\beta, \gamma)(n')\} \le m''.
\]
Indeed, if $n \le n'$ then the above expression holds trivially; otherwise, it follows directly from \eqref{eq:dominating1}. Overall, we conclude
\[
(\beta, \max\{n, n'\}, \max\{m, m', \varrho(\alpha, \beta)(n)\}) \trianglelefteq (\gamma, n'', m''),
\]
as desired. Thus, $(K, \trianglelefteq)$ is a join-semilattice.

Now we prove that properties ($\varrho$1)-($\varrho$4) are actually necessary. So suppose that $(K, \trianglelefteq)$ is a join-semilattice, towards showing that $\varrho$ satisfies ($\varrho$1)-($\varrho$4). Fix $\alpha < \beta < \gamma < \omega_1$ and $n\in\omega$.

Property ($\varrho$1) follows directly from the transitivity of $\trianglelefteq$. By definition of $\trianglelefteq$, we certainly have 
\[
(\alpha, n, 0) \trianglelefteq (\alpha, n+1, 0) \trianglelefteq (\beta, n+1, \varrho(\alpha, \beta)(n+1)).
\]
By transitivity of $\trianglelefteq$,  $(\alpha, n, 0) \trianglelefteq (\beta, n+1, \varrho(\alpha, \beta)(n+1))$. Looking at the definition of $\varrho$, this means that $\varrho(\alpha, \beta)(n) \le \varrho(\alpha, \beta)(n+1)$. Thus, $\varrho(\alpha, \beta)$ is non-decreasing.

Also property ($\varrho$2) follows from the transitivity of $\trianglelefteq$. By definition of $\trianglelefteq$, 
\[
(\alpha, n, 0) \trianglelefteq (\beta, n, \varrho(\alpha, \beta)(n)) \trianglelefteq (\gamma, n, \max\{\varrho(\alpha, \beta)(n), \varrho(\beta, \gamma)(n)\}).
\]
By transitivity, $(\alpha, n, 0) \trianglelefteq (\gamma, n, \max\{\varrho(\alpha, \beta)(n), \varrho(\beta, \gamma)(n)\})$, which means, by definition of $\trianglelefteq$, that $\varrho(\alpha, \gamma)(n) \le \max\{\varrho(\alpha, \beta)(n), \varrho(\beta, \gamma)(n)\}$.

Property ($\varrho$3) follows from $(K, \trianglelefteq)$ being a join-semilattice.  Notice that the least upper bound of $(\alpha, n, 0)$ and $(\beta, n, 0)$ must be $(\beta, n, \varrho(\alpha, \beta)(n))$. Indeed, $(\alpha, n, 0), (\beta, n, 0) \trianglelefteq (\beta, n, \varrho(\alpha, \beta)(n))$, and there is no $x \in K$ such that 
\[
(\alpha, n, 0), (\beta, n, 0) \trianglelefteq x \lhd (\beta, n, \varrho(\alpha, \beta)(n)).
\]
 So, since we are assuming $(K, \trianglelefteq)$ to be a join-semilattice, $(\beta, n, \varrho(\alpha, \beta)(n))$ must be the least upper bound of $(\alpha, n, 0)$ and $(\beta, n, 0)$. Moreover, by definition of $\trianglelefteq$,
\[
(\alpha, n, 0), (\beta, n, 0) \trianglelefteq (\gamma, n, \max\{\varrho(\alpha, \gamma)(n), \varrho(\beta, \gamma)(n)\}).
\]
So,
\[
(\beta, n, \varrho(\alpha, \beta)(n)) \trianglelefteq (\gamma, n, \max\{\varrho(\alpha, \gamma)(n), \varrho(\beta, \gamma)(n)\}),
\]
and this directly implies $\varrho(\alpha, \beta)(n) \le \max\{\varrho(\alpha, \gamma)(n), \varrho(\beta, \gamma)(n)\}$.

Finally, let us show that ($\varrho$4) also holds. By arguing as at the beginning of the previous paragraph, the least upper bound of $(\alpha, n, 0)$ and $(\beta, 0 ,0)$ must be $(\beta, n, \varrho(\alpha, \beta)(n))$. Since
\[
(\alpha, n, 0), (\beta, 0, 0) \trianglelefteq (\gamma, n, \max\{\varrho(\alpha, \gamma)(n), \varrho(\beta, \gamma)(0)\}),
\]
we conclude that 
\[
(\beta, n, \varrho(\alpha, \beta)(n)) \trianglelefteq (\gamma, n, \max\{\varrho(\alpha, \gamma)(n), \varrho(\beta, \gamma)(0)\}),
\]
from which $\varrho(\beta, \gamma)(n) \le \max\{\varrho(\alpha, \gamma)(n), \varrho(\beta, \gamma)(0)\}$ follows.
\end{proof}

So we know precisely when $(K, \trianglelefteq_\varrho)$ is a join-semilattice. Next, we characterize, in terms of the properties of $\varrho$, when $(K, \trianglelefteq_\varrho)$ is also lower finite.

\begin{lemma}\label{lemma:Klowerfinite}
Given a map $\varrho:[\omega_1]^2 \rightarrow {}^\omega \omega$ such that $(K, \trianglelefteq_\varrho)$ is a join-semilattice, then  $(K, \trianglelefteq_\varrho)$ is lower finite if and only if $\{\nu < \alpha : \varrho(\nu, \alpha)(0) \le n\}$ is finite for all $\alpha < \omega_1$ and $n\in\omega$.
\end{lemma}
\begin{proof}
Fix a map $\varrho:[\omega_1]^2 \rightarrow {}^\omega \omega$ such that $(K, \trianglelefteq_\varrho)$ is a join-semilattice. Fix some $(\beta, n', m') \in K$. The following holds:
\begin{multline}\label{eq:lowerfinite}
\{(\alpha, 0, 0) : \alpha \le \beta \text{ and }\varrho(\alpha, \beta)(0) \le m'\} \subseteq {\downarrow} (\beta, n', m')\subseteq \\ \{0\} \cup \{(\alpha, n, m) : \alpha \le \beta \text{ and } n \le n' \text{ and } m\le m' \text{ and } \varrho(\alpha, \beta)(0) \le m'\},
\end{multline}
where the first inclusion follows directly from the definition of $\trianglelefteq_\varrho$, and the second one from the same definition and from ($\varrho$1) of Lemma~\ref{lemma:Kjoin}. We conclude from \eqref{eq:lowerfinite}, that ${\downarrow} (\beta, n', m')$ is finite if and only if $\{\alpha < \beta: \varrho(\alpha, \beta)(0) \le m'\}$ is finite.
\end{proof}

Let us notice at this point that if $(K, \trianglelefteq_\varrho)$ is a lower finite join-semilattice, then it is necessarily a $3$-ladder. For each $\alpha < \omega_1$, denote the set $\{0\} \cup (\alpha \times \omega \times \omega)$ by $K_\alpha$. It directly follows from the definition of $\trianglelefteq_\varrho$ that $K_\alpha$ is an ideal of $(K, \trianglelefteq_\varrho)$ (in case $(K, \trianglelefteq_\varrho)$ is a poset) for every $\alpha < \omega_1$. From this point onward, we denote $\pi_{K_\alpha}$ simply by $\pi_\alpha$.
\begin{lemma}\label{lemma:Kladder}
If $(K, \trianglelefteq_\varrho)$ is a lower finite join-semilattice, then it is a $3$-ladder.
\end{lemma}
\begin{proof}
The fact that  $(K, \trianglelefteq_\varrho)$ is a lattice is clear, since every lower finite join-semilattice with a least element is a lattice. So we just have to prove that every element of $K$ has at most $3$ lower covers.  Pick $(\alpha, n, m) \in K$ and $x \in K$ with $x \lhd_\varrho (\alpha, n, m)$. 

If $x \in K_{\alpha+1} \setminus K_\alpha$, then clearly either $x \trianglelefteq_\varrho (\alpha, n-1, m)$ (if $n > 0$) or $x \trianglelefteq_\varrho (\alpha, n, m-1)$ (if $m > 0$). If on the other hand $x \in K_\alpha$, then clearly $x \trianglelefteq_\varrho \pi_\alpha(\alpha, n, m)$. Overall, $(\alpha, n, m)$ has at most $3$ lower covers.
\end{proof}

The crucial and most difficult step in the proof of Theorem~\ref{thm:dominatingmain} is the next theorem.  One of the two implications of Theorem~\ref{thm:dominatingmain} directly follows from it.

\begin{theorem}\label{thm:Kmaximal}
Given a map $\varrho:[\omega_1]^2 \rightarrow {}^\omega \omega$ such that $(K, \trianglelefteq_\varrho)$ is a $3$-ladder, then $(K, \trianglelefteq_\varrho)$ is maximal if and only if $\{\varrho(0, \alpha) : \alpha < \omega_1\}$ is a dominating family.
\end{theorem}
\begin{proof}
Fix a map $\varrho:[\omega_1]^2 \rightarrow {}^\omega \omega$ such that $(K, \trianglelefteq_\varrho)$ is a $3$-ladder. We drop the subscript from $\trianglelefteq_\varrho$ as the map $\varrho$ is fixed.  

Let us first prove that if $\{\varrho(0, \alpha) : \alpha < \omega_1\}$ is a dominating family, then $(K, \trianglelefteq)$ is maximal.  This is the most difficult direction.

Since $(K, \trianglelefteq)$ is a $3$-ladder by assumption, $\varrho$ must satisfy ($\varrho$1)-($\varrho$4) of Lemma~\ref{lemma:Kjoin}. In what follows, we omit repeated reference to Lemma~\ref{lemma:Kjoin}.

\begin{claim}\label{thm:dominating:claim2}
For each $(\beta, n, m) \in K$ and $\alpha \le \beta$, either $\pi_\alpha (\beta,n, m) = 0$ or $\pi_\alpha (\beta,n, m) = (\nu, n', m)$, where 
\begin{align*}
\nu &= \max \{\mu < \alpha : \varrho(\mu, \beta)(0) \le m\},\\
n' &= \max \{k \le n : \varrho(\nu, \beta)(k) \le m\}.
\end{align*}
\end{claim}
\begin{proof}
If there is no $\mu < \alpha$ such that $\varrho(\mu, \beta)(0) \le m$, then it follows directly from \eqref{eq:lowerfinite} that $\pi_\alpha (\beta, n, m) = 0$. Suppose otherwise, i.e., that there exists a $\mu < \alpha$ such that $\varrho(\mu, \beta)(0) \le m$. Let $\nu$ and $n'$ be as in the statement of the claim. Note that $\nu$ is well defined, as, by Lemma~\ref{lemma:Klowerfinite}, there are only finitely many $\mu < \alpha$ such that $\varrho(\mu, \beta)(0) \le m$. Moreover, it follows directly from our definition of $\trianglelefteq$ that $(\nu, n', m) \trianglelefteq (\beta, n ,m)$.

Now pick $(\mu, a, b) \in K_\alpha$ such that $(\mu, a, b) \trianglelefteq (\beta, n, m)$, towards showing $(\mu, a, b) \trianglelefteq (\nu, n', m)$.

Clearly, we must have $a \le n$ and $b \le m$. Moreover, since $\varrho(\mu, \beta)(a) \le m$, it directly follows from ($\varrho$1) that $\varrho(\mu, \beta)(0) \le m$. Thus, $\mu \le \nu$ by definition of $\nu$. 

Now let us show that $a \le n'$. The following holds:
\begin{equation}\label{eq:domination3}
\begin{split}
\varrho(\nu, \beta)(a) &\le \max \{\varrho(\mu, \beta)(a), \varrho(\nu, \beta)(0)\}\\&\le \max \{\varrho(\mu, \beta)(a), \varrho(\nu, \beta)(n')\} \le m,
\end{split}
\end{equation}
where the first inequality follows from ($\varrho$4), the second one from ($\varrho$1) and the last one from both $\varrho(\mu, \beta)(a)$ and $\varrho(\nu, \beta)(n')$ being less than or equal to $m$.
Thus, we conclude from \eqref{eq:domination3} that $\varrho(\nu, \beta)(a) \le m$, and hence $a \le n'$ by definition of $n'$. 

Finally, we have
\[
\varrho(\mu, \nu)(a) \le \max \{\varrho(\mu, \beta)(a), \varrho(\nu, \beta)(a)\} \le m
\]
where the first inequality follows from ($\varrho$3) and the second one from \eqref{eq:domination3}. Overall, we conclude that $(\mu, a, b) \trianglelefteq (\nu, n', m)$, as desired.
\end{proof}

Now that we have determined the behavior of $\pi_\alpha$, we are ready to prove the main technical claim of the proof, which is key to proving maximality of $(K, \trianglelefteq)$. It will be useful, for readability, to let $\alpha(x) \coloneqq \beta$ for all $x = (\beta, n, m) \in K$---in other words, $\alpha(x)$ is the first coordinate of $x$.

\begin{claim}\label{thm:dominating:claim3}
Given a cofinal meet-subsemilattice $C\subseteq K$ which is also a $2$-ladder, there exists $\gamma < \omega_1$ such that:
\begin{enumerate}[label={\upshape (\roman*)}]
\item  $C \cap K_\gamma$ is cofinal in $(K_\gamma, \trianglelefteq)$, and
\item there is at most one $n$ such that $\{m \in \omega : \pi_\gamma (\gamma, n, m) \in C\}$ is infinite.
\end{enumerate}
\end{claim}
\begin{proof}
Fix a cofinal meet-subsemilattice $C \subseteq K$ which is also a $2$-ladder. It is easy to see that the $\alpha$s such that $C \cap K_\alpha$ is cofinal in $(K_\alpha, \trianglelefteq)$ form a club of $\omega_1$. In particular, we can pick two such ordinals, say $\alpha$ and $\gamma$ with $\alpha < \gamma$.  We claim that $\gamma$ satisfies the desired properties. Clearly, $\gamma$ satisfies (i), by the way we chose it. So we need to prove that  $\gamma$ also satisfies (ii).

Suppose towards a contradiction that there are $N, N' \in \omega$ with $N < N'$ such that $\{m \in \omega : \pi_\gamma (\gamma, N, m) \in C\}$ and $\{m \in \omega : \pi_\gamma (\gamma, N', m) \in C\}$ are both infinite. The goal is to reach a contradiction by defining an infinite decreasing sequence of ordinals. 

For clarity, let $N_0, N_1$ and $N_2$ be $N'$, $N$ and $N'$, respectively.
By hypothesis, we can fix $M_0, M_1, M_2 \in \omega$ such that:
\begin{enumerate}[label=(\alph*)]
\item $\max \{\varrho(0, \gamma)(N'), \varrho(\alpha, \gamma)(0)\} < M_0 < M_1 \le M_2$, and
\item $\pi_\gamma (\gamma, N_i, M_i) \in C$ for $i = 0,1,2$.
\end{enumerate}

For each $i = 0,1,2$, let $\beta_i$ be $\alpha(\pi_\gamma (\gamma, N_i, M_i))$. Let us prove that $\pi_\gamma (\gamma, N_i, M_i) = (\beta_i, N_i, M_i)$. Since $\varrho(\alpha, \gamma)(0) \le M_0 \le M_1 \le M_2$, it follows from Claim~\ref{thm:dominating:claim2} that $\alpha \le \beta_0 \le \beta_1 \le \beta_2$. Moreover, we have 
\begin{align*}
\varrho(\beta_i, \gamma)(N_i) &\le \max \{\varrho(0, \gamma)(N_i), \varrho(\beta_i, \gamma)(0)\}\\&\le \max \{\varrho(0, \gamma)(N'), \varrho(\beta_i, \gamma)(0)\} \le \max\{M_0, M_i\} = M_i
\end{align*}
where the first inequality follows from ($\varrho$4), the second one from ($\varrho$1), and the third one from the choice of $M_0$---specifically, from (a).
Thus, by Claim~\ref{thm:dominating:claim2}, $\pi_\gamma (\gamma, N_i, M_i) = (\beta_i, N_i, M_i)$. 

Since $C \cap K_\gamma$ is cofinal in $(K_\gamma, \trianglelefteq)$ and it is a meet-subsemilattice, it follows that $(C \cap K_\gamma, \trianglelefteq)$ is a lattice (see Remark~\ref{rmk:cofinalmeet}). Given $x,y \in C \cap K_\gamma$, we denote the least upper bound of $x$ and $y$ in $(C \cap K_\gamma, \trianglelefteq)$ by $x \vee_C y$---note that $x \vee_C y$ may be different from $x \vee y$.

Let us inductively define a sequence $\langle \beta^*_k : k \in \omega\rangle$ of ordinals such that, for all $k$:
\begin{enumerate}[label=(\Alph*)]
\item $(\beta^*_k, N, M_1) \in C$,
\item $(\beta^*_k, N, M_1) \trianglelefteq (\gamma, N, M_1)$,
\item $\beta_0 \le \beta^*_{k}$,
\item $\beta^*_{k+1} < \beta^*_k$.
\end{enumerate}

Property (D) tells us that this is a strictly decreasing sequence of ordinals, and thus once we define such a sequence, we reach a contradiction.

First let $\beta^*_0 = \beta_1$. Now, suppose that we have defined $\beta^*_k$ satisfying (A)-(C), towards defining $\beta^*_{k+1}$. We claim that $(\beta_0, N', M_0) \vee_C (\beta^*_{k}, N, M_1)$ must be of the form $(\iota, N', b)$, for some $\iota$ and $b$. Indeed, by (B), $(\beta^*_{k}, N, M_1) \trianglelefteq (\gamma, N, M_1)$. But since $(\gamma, N, M_1) \trianglelefteq (\gamma, N', M_2)$, we have $(\beta^*_{k}, N, M_1) \trianglelefteq (\gamma, N', M_2)$. Thus $(\beta^*_{k}, N, M_1) \trianglelefteq (\beta_2, N', M_2) \in C$. In particular, we have
\begin{equation}\label{eq:dominating4}
(\beta_0, N', M_0) \trianglelefteq (\beta_0, N', M_0) \vee_C (\beta^*_{k}, N, M_1) \trianglelefteq (\beta_2, N', M_2).
\end{equation}
Hence, as claimed, there are $\iota$ and $b$ such that $(\beta_0, N', M_0) \vee_C (\beta^*_{k}, N, M_1) = (\iota, N', b)$. Moreover, it quickly follows from \eqref{eq:dominating4} that  $\beta^*_k \le \iota$ and $M_1 \le b \le M_2$.

Since $(\beta_0, N', M_0)$ and $(\beta^*_k, N, M_1)$ are $\trianglelefteq$-incomparable (as $N' > N$ and $M_0 < M_1$), it must be the case that $(\iota, N', b)$ has two distinct lower covers $c_0, c_1$ in $C$  with $(\beta_0, N', M_0) \trianglelefteq c_0$ and $(\beta^*_k, N, M_1) \trianglelefteq c_1$. Let $\mu_j = \alpha(c_j)$ for $j = 0,1$.

\begin{subclaim*}
$\mu_0 < \beta^*_k$ and $c_0 = (\mu_0, N', b)$.
\end{subclaim*}
\begin{proof}
We first show that $\pi_\alpha (\iota, N', b) = (\nu, N', b)$ for some $\nu < \alpha$.  Let  \[
\nu \coloneqq \max \{\mu < \alpha : \varrho(\mu, \iota)(0) \le b\}.
\]
The following holds:
\begin{align*}
\varrho(\nu, \iota)(N') &\le \max\{\varrho(0, \iota)(N'), \varrho(\nu, \iota)(0)\}
\\&\le \max\{\varrho(0, \beta_0)(N'), \varrho(\beta_0, \iota)(N'), \varrho(\nu, \iota)(0)\}
\\&\le \max\{\varrho(0, \gamma)(N'), \varrho(\beta_0, \gamma)(N'), \varrho(\beta_0, \iota)(N'), \varrho(\nu, \iota)(0)\}\\&\le \max\{M_0, b\} = b,
\end{align*}
where the first inequality follows from ($\varrho$4), the second one follows from ($\varrho$2), the third one from ($\varrho$3), and, finally, the fourth one follows from $\varrho(0, \gamma)(N') \le M_0$ by the choice of $M_0$ and from \eqref{eq:dominating4}. In particular, we conclude by Claim~\ref{thm:dominating:claim2}, that $\pi_\alpha (\iota, N', b) = (\nu, N', b)$.

Since, by hypothesis, $C$ is a $2$-ladder, it follows that $c_0$ and $c_1$ are all the lower covers of $(\iota, N', b)$ in $C$. In particular, $(\nu, N', b) \trianglelefteq c_j$ for some $j \in \{0,1\}$. It follows directly from
\[
(\nu, N', b) \trianglelefteq c_j \trianglelefteq (\iota, N', b)
\]
that $c_j = (\mu_j, N', b)$. We now show that $j = 0$ and that $\mu_0 < \beta^*_k$.

Suppose towards a contradiction that $j = 1$. Since $(\beta^*_k, N, M_1) \trianglelefteq c_1$ we have $\beta^*_k \le \mu_1$. Moreover, by (C), $\beta_0 \le \beta^*_k$, and thus $\beta_0 \le \mu_1$. Furthermore,
\[
\varrho(\beta_0, \mu_1)(N') \le \max\{\varrho(\beta_0, \iota)(N'), \varrho(\mu_1, \iota)(N')\} \le b,
\]
where the first inequality follows from ($\varrho$3) and the second one both $(\beta_0, N', M_0)$ and $c_1  = (\mu_1, N', b)$ being $\trianglelefteq (\iota, N', b)$. But the above expression implies $(\beta_0, N', M_0) \trianglelefteq c_1$, which is a contradiction, as it would mean that $c_0 = c_1$. Thus, $j = 0$. 

Now assume towards a contradiction that $\mu_0 \ge \beta^*_k$. Then, we have
\begin{align*}
\varrho(\beta^*_k, \mu_0)(N) &\le \max\{\varrho(\beta^*_k, \iota)(N), \varrho(\mu_0, \iota)(N)\}\\&\le\max\{\varrho(\beta^*_k, \iota)(N), \varrho(\mu_0, \iota)(N')\} \le b,
\end{align*}
where the first inequality follows from ($\varrho$3), the second one from ($\varrho$1), and the last one from both $(\beta^*_k, N, M_1)$ and $c_0 = (\mu_0, N', b)$ being  $\trianglelefteq (\iota, N', b)$. Thus, $(\beta^*_k, N, M_1) \trianglelefteq c_0 = (\mu_0, N', b)$, which is again a contradiction, as it would imply $c_0 = c_1$. Overall, we have shown that $j = 0$ and $\mu_0 < \beta^*_k$. 
\end{proof}

We are ready to define $\beta^*_{k+1}$. Let
\[
\beta^*_{k+1} \coloneqq \alpha\big((\beta^*_k, N, M_1) \wedge (\mu_0, N', b)\big).
\]
Clearly, $\beta^*_{k+1} < \beta^*_k$, as $\beta^*_{k+1} \le \mu_0 < \beta^*_k$, where the last inequality comes from our Subclaim. We are left to prove that $\beta^*_{k+1}$ also satisfies (A)-(C).

By induction hypothesis, $(\beta^*_k, N, M_1) \in C$. Moreover, by our Subclaim, $c_0 = (\mu_0, N', b) \in C$. Since $C$ is a meet-subsemilattice of $(K, \trianglelefteq)$, we conclude that $(\beta^*_k, N, M_1) \wedge (\mu_0, N', b) \in C$. Thus, in order to show that $\beta^*_{k+1}$ satisfies (A)-(C), it suffices to prove that $\beta_0 \le \beta^*_{k+1}$ and that $(\beta^*_k, N, M_1) \wedge (\mu_0, N', b) = (\beta^*_{k+1}, N, M_1)$. 

First note that 
\begin{align*}
\varrho(\beta_0, \beta^*_k)(N) &\le  \max\{\varrho(\beta_0, \gamma)(N), \varrho(\beta^*_k, \gamma)(N)\}\\&\le \max\{\varrho(\beta_0, \gamma)(N'), \varrho(\beta^*_k, \gamma)(N)\} \le \max\{M_0, M_1\} = M_1,
\end{align*}
where the first inequality follows from ($\varrho$3), the second one from ($\varrho$1), and the last one from $(\beta_0, N', M_0) \trianglelefteq (\gamma, N', M_0)$ and $(\beta^*_k, N, M_1) \trianglelefteq (\gamma, N, M_1)$ (see (B)). Thus, we conclude that $(\beta_0, N, M_1) \trianglelefteq (\beta^*_k, N, M_1)$.

Moreover we have
\[
\varrho(\beta_0, \mu_0)(N) \le \varrho(\beta_0, \mu_0)(N') \le b,
\]
where the first inequality follows from ($\varrho$1) and the second one from $(\beta_0, N', M_0) \trianglelefteq c_0 = (\mu_0, N', b)$.
We conclude that $(\beta_0, N, M_1) \trianglelefteq (\mu_0, N', b)$---recall that $M_1 \le b$. Overall,
\[
(\beta_0, N, M_1) \trianglelefteq (\beta^*_k, N, M_1) \wedge (\mu_0, N', b) \trianglelefteq (\beta^*_k, N, M_1).
\]
This already implies $\beta_0 \le \beta^*_{k+1}$ and that $(\beta^*_k, N, M_1) \wedge (\mu_0, N', b) = (\beta^*_{k+1}, N, M_1)$. This ends the inductive definition, and with it we reach the contradiction.
\end{proof}

We are ready to prove that $(K, \trianglelefteq)$ is maximal. By Theorem~\ref{thm:characterization}, we must show that there is no cofinal meet-subsemilattice of $K$ which is also a $2$-ladder. Towards a contradiction, assume that there exists one, and denote it by $C$. Let $\gamma$ be as in the statement of Claim~\ref{thm:dominating:claim3}. Moreover, let $N$ be the unique natural number such that $\{m \in \omega : \pi_\gamma (\gamma, N, m) \in C\}$ is infinite, if it exists; otherwise, let $N = 0$.

Define $r \in {}^\omega \omega$ as follows: for each $n \in \omega$, let
\[
r(n) = \begin{cases}
\max (\{m\in\omega : \pi_\gamma (\gamma, n, m) \in C\} \cup \{0\}) &\text{if } n \neq N,\\
0 &\text{otherwise}.
\end{cases}
\]

By hypothesis $\{\varrho(0, \delta):\delta < \omega_1\}$ is a dominating family. By the $\sigma$-directedness of $<^*$, there is a $\delta$ such that $\varrho(0, \delta)$ dominates $r$ and $\varrho(0, \alpha)$ for all $\alpha \le \gamma$. Fix one such $\delta$. Clearly, we must have $\delta > \gamma$. By ($\varrho$2), for every $n \in \omega$,
\[
\varrho(0, \delta)(n) \le \max\{\varrho(0, \gamma)(n), \varrho(\gamma, \delta)(n)\}.
\]
Since $\varrho(0, \gamma) <^* \varrho(0, \delta)$, we conclude that $\varrho(0, \delta) \le^* \varrho(\gamma, \delta)$, and thus $r <^* \varrho(\gamma, \delta)$. 

Fix $M \in \omega$ such that $M > N$ and $r(n) < \varrho(\gamma, \delta)(n)$ for all $n \ge M$. Since we are assuming $C$ to be cofinal, there is $(\lambda, a, b) \in C$ such that $(\delta, M, \varrho(\gamma, \delta)(M)) \trianglelefteq (\lambda, a, b)$. We have
\begin{align*}
\varrho(\gamma, \lambda)(0) &\le \max\{\varrho(\gamma, \delta)(0), \varrho(\delta, \lambda)(0)\}\\&\le \max\{\varrho(\gamma, \delta)(M), \varrho(\delta, \lambda)(M)\} \le b,
\end{align*}
where the first inequality follows from ($\varrho$2), the second one from ($\varrho$1), and the last one from $(\delta, M, \varrho(\gamma, \delta)(M)) \trianglelefteq (\lambda, a, b)$.
In particular, by Claim~\ref{thm:dominating:claim2}, there exists $a'$ such that $\pi_{\gamma+1} (\lambda, a, b) = (\gamma, a', b)$. 

We claim that $a' \neq N$ and $r(a') < b$. By definition of $M$, it suffices to prove $\varrho(\gamma, \delta)(a') \le b$ and $a' \ge M$. The following holds:
\begin{align*}
\varrho(\gamma, \delta)(a')  &\le \max \{\varrho(\gamma, \lambda)(a'), \varrho(\delta, \lambda)(a')\}\\&\le \max \{\varrho(\gamma, \lambda)(a'), \varrho(\delta, \lambda)(0)\}
\\&\le \max \{\varrho(\gamma, \lambda)(a'), \varrho(\delta, \lambda)(M)\} \le b,
\end{align*}
where the first inequality follows from ($\varrho$3), the second one from ($\varrho$4) and the third one from ($\varrho$1). So we are left to show $a' \ge M$. By Claim~\ref{thm:dominating:claim2}, it suffices to prove that $\varrho(\gamma, \lambda)(M) \le b$. 

By ($\varrho$3), we have
\[
\varrho(\gamma, \lambda)(M) \le \max\{\varrho(\gamma, \delta)(M), \varrho(\delta, \lambda)(M)\}.
\]
We already know that $\varrho(\gamma, \delta)(M) \le b$. Moreover, since $(\delta, M, \varrho(\gamma, \delta)(M)) \trianglelefteq (\lambda, a, b)$, we also have $\varrho(\delta, \lambda)(M) \le b$. Overall, $\varrho(\gamma, \lambda)(M)\le b$, as desired. 

We have shown that $a' \neq N$ and $r(a') < b$. By definition of $r$, this means that $\pi_\gamma (\gamma, a', b) \not\in C$. However, by Lemma~\ref{lemma:concproj}, 
\[
\pi_\gamma(\lambda, a, b) = \pi_\gamma \circ \pi_{\gamma+1}(\lambda, a, b) = \pi_\gamma (\gamma, a', b).
\] 
Thus, $\pi_\gamma(\lambda, a, b) \not\in C$. Moreover, since $C \cap K_\gamma$ is cofinal in $K_\gamma$, we can fix $x \in C \cap K_\gamma$ such that $\pi_\gamma (\lambda, a ,b)  \trianglelefteq x$. It follows from Lemma~\ref{lemma:meetproj} that 
\[
(\lambda, a, b) \wedge x = \pi_\gamma(\lambda, a, b) \wedge x = \pi_\gamma(\lambda, a, b).
\]
However, since $C$ is a meet-subsemilattice, and since both $x$ and $(\lambda, a, b)$ belong to $C$, we would have $(\lambda, a, b) \wedge x = \pi_\gamma(\lambda, a, b) \in C$, which is a contradiction, as we have just shown that $\pi_\gamma(\lambda, a, b) \not\in C$. Thus, $(K, \trianglelefteq)$ is  a maximal $3$-ladder. 

Now let us prove the other implication of our theorem, which is much easier. We do so by contraposition, so assume  $\{\varrho(0, \alpha): \alpha < \omega_1\}$ is not dominating, towards showing that $(K, \trianglelefteq)$ is not maximal. 

Fix some $f \in {}^\omega \omega$ such that $f \not <^* \varrho(0, \alpha)$ for every $\alpha < \omega_1$.  By Theorem~\ref{thm:characterization}, we need to prove that $(K, \trianglelefteq_\varrho)$ has a cofinal meet-subsemilattice which is a $2$-ladder in the induced ordering. We can suppose without loss of generality that $f$ is increasing. Let
\[
C \coloneqq \{0\} \cup \{(\alpha, n, f(n)) \in K \mid \alpha < \omega_1 \text{ and } \varrho(0, \alpha)(n) \le f(n)\}.
\]
We claim that $C$ is a cofinal meet-subsemilattice of $(K, \trianglelefteq_\varrho)$ and that it is a $2$-ladder. First note that, by definition of $\trianglelefteq_\varrho$, $(0, n, f(n)) \trianglelefteq (\alpha, n, f(n))$ for every $(\alpha, n, f(n)) \in C$.

The cofinality of $C$ is a direct consequence of our assumptions on $f$. Indeed, for every $(\alpha, n, m) \in K$, since by assumption $f \not <^* \varrho(0, \alpha)$ and $f$ is increasing, there exists an $n' \ge n$ such that $\max\{m, \varrho(0, \alpha)(n')\} \le f(n')$, and therefore $(\alpha, n, m) \trianglelefteq_\varrho (\alpha, n', f(n')) \in C$.

Now note that for every $(\alpha, n, f(n)) \in C$, $\pi_\alpha (\alpha, n, f(n))$ also belongs to $C$. Indeed, we have
\[
(0, n, f(n)) \trianglelefteq_\varrho \pi_\alpha (\alpha, n, f(n)) \trianglelefteq_\varrho (\alpha, n, f(n)),
\]
from which it follows directly that $\pi_\alpha(\alpha, n, f(n)) = (\gamma, n, f(n))$ for some $\gamma < \alpha$ and $\varrho(0, \gamma)(n) \le f(n)$. Thus,  $\pi_\alpha (\alpha, n, f(n)) \in C$. This implies that every element of $C$ has at most $2$ lower covers: given $(\alpha, n, f(n))$ in $C$, let $n^-$ be the greatest $m < n$ such that $\varrho(0, \alpha)(m) \le f(m)$, if such an $m$ exists; then, the lower covers of $(\alpha, n, f(n))$ in $C$ are $(\alpha, n^-, f(n^-))$ (if $n^-$ is defined) and $\pi_\alpha (\alpha, n, f(n))$.

We are left to prove that $C$ is a meet-subsemilattice of $(K, \trianglelefteq_\varrho)$. Pick $(\alpha, n, f(n))$ and $(\beta, m, f(m))$ in $C$, with $m \le n$, towards showing that their greatest lower bound belongs to $C$ as well. If $(\alpha, n, f(n)) \wedge (\beta, m, f(m)) = 0$ then there is nothing to prove since $0 \in C$ by definition. So suppose otherwise. Then
\[
S \coloneqq \big\{\mu \le \min \{\alpha, \beta\} : \varrho(\mu, \alpha)(0) \le f(n) \text{ and } \varrho(\mu, \beta)(0) \le f(m)\big\}
\]
is nonempty by assumption and finite by Lemma~\ref{lemma:Klowerfinite}. Let $\nu = \max S$. By Lemma~\ref{lemma:meetproj}, 
\[
(\alpha, n, f(n)) \wedge (\beta, m, f(m)) = \pi_{\nu + 1} (\alpha, n, f(n)) \wedge \pi_{\nu + 1} (\beta, m, f(m)).
\]
The argument in the previous paragraph implies $\pi_{\nu + 1} (\alpha, n, f(n))  = (\nu, n, f(n)) \in C$ and $\pi_{\nu + 1} (\beta, m, f(m))  = (\nu, m, f(m)) \in C$. In particular, we conclude that 
\[
(\alpha, n, f(n)) \wedge (\beta, m, f(m)) = (\nu, m, f(m)) \in C.
\]
Thus, $C$ is also a meet-subsemilattice and we are done. 
\end{proof}

Finally, let us show the following theorem, which shows the remaining implication of Theorem~\ref{thm:dominatingmain}.

\begin{theorem}\label{thm:Kexistence}
If $\mathfrak{d}= \aleph_1$, then there exists a map $\varrho:[\omega_1]^2 \rightarrow {}^\omega \omega$ such that $(K, \trianglelefteq_\varrho)$ is a maximal $3$-ladder.
\end{theorem}
\begin{proof}
By Lemmas~\ref{lemma:Kjoin}-\ref{lemma:Kladder} and Theorem~\ref{thm:Kmaximal}, we need to construct a map $\varrho:[\omega_1]^2 \rightarrow {}^\omega \omega$ that satisfies the following properties: for every $\alpha < \beta < \gamma$ and for every $n\in\omega$,
\begin{enumerate}[label={\upshape ($\varrho$\arabic*)}]
\item $\varrho(\alpha, \beta)$ is non-decreasing,
\item $\varrho(\alpha, \gamma)(n) \le \max \{\varrho(\alpha, \beta)(n), \varrho(\beta, \gamma)(n)\}$,
\item $\varrho(\alpha, \beta)(n) \le \max \{\varrho(\alpha, \gamma)(n), \varrho(\beta, \gamma)(n)\}$,
\item $\varrho(\beta, \gamma)(n) \le \max \{\varrho(\alpha, \gamma)(n), \varrho(\beta, \gamma)(0)\}$,
\item there are finitely many $\nu < \alpha$ such that $\varrho(\nu, \alpha)(0) \le n$,
\item $\{\varrho(0, \delta) : \delta < \omega_1\}$ is a dominating family.
\end{enumerate}

Fix a dominating family $\{f_\alpha \in {}^\omega \omega : \alpha < \omega_1\}$, which exists by our hypothesis $\mathfrak{d}= \aleph_1$.  We define $\varrho \upharpoonright [\delta]^2$ by induction on $\delta < \omega_1$, and in doing so, we make sure that conditions ($\varrho$1)-($\varrho$5) are satisfied by $\varrho \upharpoonright [\delta]^2$. For clarity, when we say ``$\varrho \upharpoonright [\delta]^2$ satisfies ($\varrho$1)-($\varrho$5)'' we mean that the map $\varrho \upharpoonright [\delta]^2$ satisfies statements ($\varrho$1)-($\varrho$5) for all $\alpha, \beta, \gamma < \delta$ and $n\in\omega$.

Assume that we defined $\varrho \upharpoonright [\delta]^2$  for some $0 <\delta < \omega_1$ and that $\varrho \upharpoonright [\delta]^2$ satisfies ($\varrho$1)-($\varrho$5), towards extending $\varrho$ on $[\delta +1]^2$.

Fix a non-decreasing sequence $(\delta_n)_{n\in\omega}$ such that $\sup_n (\delta_n+1) = \delta$. Now, for each $\alpha < \delta$, let $n_\alpha = \min\{ n \in \omega : \alpha \le \delta_n\}$. Fix a non-decreasing $f^*_\delta \in {}^\omega \omega$ such that $f_\delta \le^* f^*_\delta$ and $\varrho(\delta_m, \delta_n) \le^* f^*_\delta$ for all $n, m \in \omega$ with $m < n$---we can find such $f_\delta^*$ because $<^*$ is $\sigma$-directed. 

Now, fix an increasing sequence $\langle k_n : n \in \omega\rangle$ of natural numbers such that $\varrho(\delta_m, \delta_n)(k) \le f^*_\delta (k)$ for all $m < n$ and $k \ge k_n$.

Extend $\varrho$ on $[\delta+1]^2$ as follows: for each $\alpha < \delta$ and $k \in \omega$ let
\[
\varrho(\alpha, \delta)(k) \coloneqq \max (\{n_\alpha, f^*_\delta(k), \varrho(\alpha, \delta_{n_\alpha})(k)\} \cup \{\varrho(\delta_m, \delta_{n_\alpha})(k_{n_\alpha}) : m < n_\alpha\}).
\]

We now prove that $\varrho$ satisfies ($\varrho$1)-($\varrho$5) also on $[\delta+1]^2$. Clearly, $\varrho(\alpha, \delta)$ is non-decreasing for all $\alpha < \delta$, since $f^*_\delta$ is non-decreasing by choice and $\varrho(\alpha, \delta_{n_\alpha})$ is non-decreasing by induction hypothesis. Thus, ($\varrho$1) holds.

\begin{claim}\label{thm:dominating:claim4}
For all $\alpha < \delta$, $k \in \omega$ and $m < n_\alpha$, $\varrho(\delta_m, \delta_{n_\alpha})(k) \le \varrho(\alpha, \delta)(k)$.
\end{claim}
\begin{proof}
First suppose $k \ge k_{n_\alpha}$. Then $\varrho(\delta_{m}, \delta_{n_\alpha})(k) \le f^*_\delta (k)$. By definition of $\varrho(\alpha, \delta)$, $f^*_\delta(k) \le \varrho(\alpha, \delta)(k)$. Thus, $\varrho(\delta_{m}, \delta_{n_\alpha})(k) \le \varrho(\alpha, \delta)(k)$.

Now suppose $k < k_{n_\alpha}$. By induction hypothesis, $\varrho(\delta_m, \delta_{n_\alpha})$ is non-decreasing. Therefore, $\varrho(\delta_{m}, \delta_{n_\alpha})(k) \le \varrho(\delta_m, \delta_{n_\alpha})(k_{n_\alpha})$. Moreover, $\varrho(\delta_m, \delta_{n_\alpha})(k_{n_\alpha}) \le \varrho(\alpha, \delta)(k)$ by definition of $\varrho(\alpha, \delta)$. Once again, we obtain $\varrho(\delta_{m}, \delta_{n_\alpha})(k) \le \varrho(\alpha, \delta)(k)$.
\end{proof}

\begin{claim}
$\varrho \upharpoonright [\delta + 1]^2$ satisfies {\upshape ($\varrho$2)}.
\end{claim}
\begin{proof}
Fix $\alpha, \beta$ such that $\alpha < \beta < \delta$ and $k \in \omega$, towards showing $\varrho(\alpha, \delta)(k) \le \max\{\varrho(\alpha, \beta)(k) , \varrho(\beta, \delta)(k)\}$. Clearly, $n_\alpha \le n_\beta$, since $\alpha < \beta$. Looking at the definition of $\varrho(\alpha, \delta)(k)$, it is easy to see that, in order to prove our main inequality, it suffices to check that 
\begin{enumerate}[label={(\roman*)}]
\item $\varrho(\alpha, \delta_{n_\alpha})(k) \le \max\{\varrho(\alpha, \beta)(k) , \varrho(\beta, \delta)(k)\}$, and
\item $\varrho(\delta_m, \delta_{n_\alpha})(k_{n_\alpha}) \le \max\{\varrho(\alpha, \beta)(k) , \varrho(\beta, \delta)(k)\}$ for all $m < n_\alpha$.
\end{enumerate}
Let us first take care of the inequality (ii). Given some $m < n_{\alpha}$, we have
\begin{align*}
\varrho(\delta_m, \delta_{n_\alpha})(k_{n_\alpha}) &\le \max\{\varrho(\delta_m, \delta_{n_\beta})(k_{n_\alpha}), \varrho(\delta_{n_\alpha}, \delta_{n_\beta})(k_{n_\alpha})\}\\&\le \max\{\varrho(\delta_m, \delta_{n_\beta})(k_{n_\beta}), \varrho(\delta_{n_\alpha}, \delta_{n_\beta})(k_{n_\beta})\}\\&\le \varrho(\beta, \delta)(k) \le \max\{\varrho(\alpha, \beta)(k) , \varrho(\beta, \delta)(k)\},
\end{align*}
where the first inequality follows from assuming ($\varrho$3) below $\delta$; the second one follows from $k_{n_\alpha} \le k_{n_\beta}$ and from assuming ($\varrho$1) below $\delta$; finally, the third inequality follows directly from the definition of $\varrho(\beta, \delta)$. 

Now let us prove the inequality (i). First suppose that $n_\alpha = n_\beta$. In particular, $\alpha < \beta \le \delta_{n_\alpha} = \delta_{n_\beta}$.  Then,
\begin{align*}
\varrho(\alpha, \delta_{n_\alpha})(k) &\le \max\{\varrho(\alpha, \beta)(k), \varrho(\beta, \delta_{n_\alpha})(k)\}\\ &= \max\{\varrho(\alpha, \beta)(k), \varrho(\beta, \delta_{n_\beta})(k)\}\\&\le \max\{\varrho(\alpha, \beta)(k), \varrho(\beta, \delta)(k)\},
\end{align*}
where the first inequality follows from assuming ($\varrho$2) below $\delta$, and the second one is directly implied by the definition of $\varrho(\beta, \delta)(k)$.

Now suppose that $n_\alpha < n_\beta$. By the minimality of $n_\beta$, we must have $\delta_{n_\alpha} < \beta$. The following holds:
\begin{align*}
\varrho(\alpha, \delta_{n_\alpha})(k) &\le \max \{\varrho(\alpha, \beta)(k), \varrho(\delta_{n_\alpha}, \beta)(k)\}\\& \le \max \{\varrho(\alpha, \beta)(k), \varrho(\delta_{n_\alpha}, \delta_{n_\beta})(k), \varrho(\beta, \delta_{n_\beta})(k)\}\\&\le \max\{\varrho(\alpha, \beta)(k), \varrho(\beta, \delta)(k)\}.
\end{align*}
The first two inequalities follow from assuming ($\varrho$3) below $\delta$. The last one holds because $\varrho(\delta_{n_\alpha}, \delta_{n_\beta})(k) \le \varrho(\beta, \delta)(k)$ by Claim~\ref{thm:dominating:claim4} and because $\varrho(\beta, \delta_{n_\beta})(k) \le \varrho(\beta, \delta)(k)$ by definition of $\varrho(\beta, \delta)(k)$.  Overall, we have shown that also the inequality (i) is satisfied.
\end{proof}

\begin{claim}
$\varrho \upharpoonright [\delta + 1]^2$ satisfies {\upshape ($\varrho$3)}.
\end{claim}
\begin{proof}
Fix $\alpha, \beta$ such that $\alpha < \beta < \delta$ and $k \in \omega$, towards showing $\varrho(\alpha, \beta)(k) \le \max\{\varrho(\alpha, \delta)(k) , \varrho(\beta, \delta)(k)\}$.

First suppose that $n_\alpha = n_\beta$. We have
\begin{align*}
\varrho(\alpha, \beta)(k) &\le  \max \{\varrho(\alpha, \delta_{n_\alpha})(k), \varrho(\beta, \delta_{n_\alpha})(k)\}\\ &=  \max \{\varrho(\alpha, \delta_{n_\alpha})(k), \varrho(\beta, \delta_{n_\beta})(k)\}\\&\le \max\{\varrho(\alpha, \delta)(k) , \varrho(\beta, \delta)(k)\},
\end{align*}
where the first inequality follows from assuming ($\varrho$3) below $\delta$, and the second one is a direct consequence of the definition of $\varrho(\alpha, \delta)(k)$ and $\varrho(\beta, \delta)(k)$. 

Now suppose $n_\alpha < n_\beta$. By minimality of $n_\beta$, we must have $\delta_{n_\alpha} < \beta$. The following holds:
\begin{align*}
\varrho(\alpha, \beta)(k) &\le \max\{\varrho(\alpha, \delta_{n_\alpha})(k), \varrho(\delta_{n_\alpha}, \beta)(k)\}\\&\le \max\{\varrho(\alpha, \delta_{n_\alpha})(k), \varrho(\delta_{n_\alpha}, \delta_{n_\beta})(k), \varrho(\beta, \delta_{n_\beta})(k)\}\\&\le \max\{\varrho(\alpha, \delta)(k) , \varrho(\beta, \delta)(k)\}.
\end{align*}
The first inequality follows from assuming ($\varrho$2) below $\delta$, while the second from assuming ($\varrho$3) below $\delta$. The last one holds because $\varrho(\delta_{n_\alpha}, \delta_{n_\beta})(k) \le \varrho(\beta, \delta)(k)$ by Claim~\ref{thm:dominating:claim4} and because $\varrho(\alpha, \delta_{n_\alpha})(k)$ and $\varrho(\beta, \delta_{n_\beta})(k)$ are less than or equal to $\varrho(\alpha, \delta)(k)$ and $\varrho(\beta, \delta)(k)$, respectively, by definition of $\varrho(\alpha, \delta)(k)$ and $\varrho(\beta, \delta)(k)$.
\end{proof}
\begin{claim}
$\varrho \upharpoonright [\delta + 1]^2$ satisfies {\upshape ($\varrho$4)}.
\end{claim}
\begin{proof}
Fix $\alpha, \beta$ such that $\alpha < \beta < \delta$ and $k \in \omega$, towards showing $\varrho(\beta, \delta)(k) \le \max\{\varrho(\alpha, \delta)(k) , \varrho(\beta, \delta)(0)\}$. Looking at the definition of $\varrho(\beta, \delta)(k)$, it is easy to see that, in order to prove our inequality, it suffices to check that 
\begin{enumerate}[label={(\roman*)}]
\item $\varrho(\beta, \delta_{n_\beta})(k) \le \max\{\varrho(\alpha, \delta)(k) , \varrho(\beta, \delta)(0)\}$, and
\item $\varrho(\delta_m, \delta_{n_\beta})(k_{n_\beta}) \le \max\{\varrho(\alpha, \delta)(k) , \varrho(\beta, \delta)(0)\}$ for all $m < n_\beta$.
\end{enumerate}
Inequality (ii) is trivial. Indeed,  $\varrho(\delta_m, \delta_{n_\beta})(k_{n_\beta}) \le \varrho(\beta, \delta)(0)$ follows directly from the definition of $\varrho(\beta, \delta)(0)$. So we are left to prove (i).

First suppose that $n_\alpha = n_\beta$. Then,
\begin{align*}
\varrho(\beta, \delta_{n_\beta})(k) &\le \max \{ \varrho(\alpha, \delta_{n_\beta})(k), \varrho(\beta, \delta_{n_\beta})(0)\}\\&= \max \{ \varrho(\alpha, \delta_{n_\alpha})(k), \varrho(\beta, \delta_{n_\beta})(0)\}\\&\le  \max \{ \varrho(\alpha, \delta)(k), \varrho(\beta, \delta)(0)\},
\end{align*}
where the first inequality follows from assuming ($\varrho$4) below $\delta$, and the second one follows directly from the definition of $\varrho(\alpha, \delta)(k)$ and $\varrho(\beta, \delta)(0)$.

Now suppose $n_\alpha < n_\beta$. From the minimality of $n_\beta$ it follows $\delta_{n_\alpha} < \beta$. Then,
\begin{align*}
\varrho(\beta, \delta_{n_\beta})(k) &\le \max \{ \varrho(\delta_{n_\alpha}, \delta_{n_\beta})(k), \varrho(\beta, \delta_{n_\beta})(0)\} \\
&\le \max \{\varrho(\delta_{n_\alpha}, \delta_{n_\beta})(k), \varrho(\beta, \delta)(0)\},
\end{align*}
where the first inequality follows from assuming ($\varrho$4) below $\delta$ and the second one from the definition of $\varrho(\beta, \delta)(0)$. Given the above inequality, we are done proving (i) if we are able to show that 
\begin{equation}\label{eq:dominating5}
\varrho(\delta_{n_\alpha}, \delta_{n_\beta})(k) \le \max\{\varrho(\alpha, \delta)(k), \varrho(\beta, \delta)(0)\}.
\end{equation}

If $k \ge k_{n_\beta}$, then $\varrho(\delta_{n_\alpha}, \delta_{n_\beta})(k) \le f^*_\delta (k)$, and by definition of  $\varrho(\alpha, \delta)$, $f^*_\delta (k) \le \varrho(\alpha, \delta)(k)$. In particular, we conclude $\varrho(\delta_{n_\alpha}, \delta_{n_\beta})(k) \le \varrho(\alpha, \delta)(k)$. 

Now consider instead the case $k < k_{n_\beta}$. By assuming ($\varrho$1) below $\delta$, we have $\varrho(\delta_{n_\alpha}, \delta_{n_\beta})(k) \le \varrho(\delta_{n_\alpha}, \delta_{n_\beta})(k_{n_\beta})$ and, by definition of $\varrho(\beta, \delta)$, $ \varrho(\delta_{n_\alpha}, \delta_{n_\beta})(k_{n_\beta}) \le \varrho(\beta, \delta)(0)$. 

In either case, we have shown that \eqref{eq:dominating5} holds, and we are done proving (i).
\end{proof}
\begin{claim}
$\varrho \upharpoonright [\delta + 1]^2$ satisfies {\upshape ($\varrho$5)}.
\end{claim}
\begin{proof}
We must show that, for each $n \in \omega$, there are finitely many $\alpha < \delta$ such that  $\varrho(\alpha, \delta)(0) \le n$. For every such $\alpha$, we must have $n_\alpha \le n$ and $\varrho(\alpha, \delta_{n_\alpha})(0) \le n$---this follows directly from the definition of $\varrho(\alpha, \delta)(0)$. Thus,
\[
\{\alpha < \delta : \varrho(\alpha, \delta)(0) \le n\} \subseteq \bigcup_{m \le n}\{\alpha \le \delta_m : \varrho(\alpha, \delta_m)(0) \le n\}.
\]
From assuming ($\varrho$5) below $\delta$, it follows that the set on the right hand side is a finite union of finite sets. Hence our claim holds.
\end{proof}

We have shown that the map $\varrho$ we have defined satisfies ($\varrho$1)-($\varrho$5). We are left to prove that it also satisfies ($\varrho$6). But this follows directly from our definition of $\varrho$. Indeed, for every $\delta < \omega_1$, we have $f_\delta \le^* f^*_\delta$, by the choice of $f^*_\delta$, and $f^*_\delta \le^* \varrho(0, \delta)$ by definition of $\varrho(0, \delta)$. In particular, $f_\delta \le^* \varrho(0, \delta)$ for every $\delta < \omega_1$. Since $\{f_\delta : \delta < \omega_1\}$ is a dominating family, so is $\{\varrho(0, \delta) : \delta < \omega_1\}$. 
\end{proof}

Note that every join-semilattice of the form $(K, \trianglelefteq_\varrho)$ for some map $\varrho:[\omega_1]^2 \rightarrow {}^\omega \omega$ has breadth $3$. Consider the three elements $(0, 0, \varrho(0,1)(1)+1), (0, 1, 0), (1,0, 0)$  of $K$. It is easy to check (using \eqref{eq:leastupperbound}) that none of them is $\trianglelefteq_\varrho$-below  the $\trianglelefteq_\varrho$-least upper bound of the other two. In the next section we address the natural question of whether there can be a maximal $3$-ladder of breadth $2$.

\section{A maximal $3$-ladder of breadth $2$}\label{sec:club}

In the previous section, we constructed a maximal $3$-ladder of cardinality $\aleph_1$ under $\mathfrak{d} = \aleph_1$. That $3$-ladder has breadth $3$. In this section we prove Theorem~\ref{thm:main3}, showing that, consistently, there exists a maximal $3$-ladder of breadth $2$. Note that such a ladder must have cardinality $\aleph_1$. In fact, it follows from Theorem~\ref{thm:Ditor}\ref{thm:Ditor-1} that a $3$-ladder of breadth $2$ (more generally, a lower finite join-semilattice of breadth $2$) has cardinality at most $\aleph_1$. Furthermore, maximal $3$-ladders are necessarily uncountable by Ditor's Theorem~\ref{thm:Ditormax}. Thus, a maximal $3$-ladder of breadth $2$ must indeed have cardinality $\aleph_1$.

We show that the existence of a maximal $3$-ladder of breadth $2$ follows from a guessing principle, the club principle, denoted by $\clubsuit$, introduced by Ostaszewski in \cite{MR454941}. In the literature this principle is also called Ostaszewski's principle and tiltan. Let us recall its definition.
\begin{definition}
The club principle $\clubsuit$ holds if there exists a sequence $\langle A_\alpha : \alpha < \omega_1 \text{ and } \alpha \text{ is limit}\rangle$ such that:
\begin{enumerate}
\item $A_\alpha$ is an unbounded subset of $\alpha$, for every limit $\alpha < \omega_1$,
\item for every uncountable $X\subseteq \omega_1$ the set $\{\alpha < \omega_1 : A_\alpha \subseteq X\}$ is stationary.
\end{enumerate}
\end{definition}
The club principle is a (strict) weakening of Jensen's diamond principle $\diamondsuit$. Indeed, $\diamondsuit$ is equivalent to $\mathsf{CH}+\clubsuit$, but the club principle is consistent with $\neg \mathsf{CH}$ \cite[Theorem 7.4]{MR1623206} and even with arbitrarily large continuum by an unpublished result of Baumgartner---for a proof see \cite{MR1900391}.

\begin{theorem}[Theorem~\ref{thm:main3}]
If $\clubsuit$ holds, then there exists a maximal $3$-ladder of breadth $2$.
\end{theorem}
\begin{proof}
Let $E = \omega \times \{0,1\}$ and consider the following strict partial order $\prec$ on $E$: given $(n,i), (m,j) \in E$, we let $(n,i) \prec (m,j)$ if and only if $n < m$ and $j = 1$. Let $\preceq$ be its reflexive closure. It is easy to see that $(E, \preceq)$ is a lower finite join-semilattice whose elements have at most $2$ lower covers.

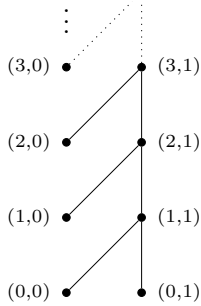
\begin{figure}[H]
\centering
\begin{tikzpicture}
    \node [draw, shape = circle, fill = black, minimum size = 0.1cm, inner sep=0pt, label={[shift={(0.5,-0.3)}]$\scriptstyle (0, 1)$},] at (1,0) (a0){};
    
    \node [draw, shape = circle, fill = black, minimum size = 0.1cm, inner sep=0pt, label={[shift={(-0.5,-0.3)}]$\scriptstyle (0, 0)$}] at (0,0) (a1){};
    \node [draw, shape = circle, fill = black, minimum size = 0.1cm, inner sep=0pt, label={[shift={(0.5,-0.3)}]$\scriptstyle (1, 1)$}] at (1,1) (a2){};
    \node [draw, shape = circle, fill = black, minimum size = 0.1cm, inner sep=0pt, label={[shift={(-0.5,-0.3)}]$ \scriptstyle (1, 0)$}] at (0,1) (a3){};
    
    \node [draw, shape = circle, fill = black, minimum size = 0.1cm, inner sep=0pt, label={[shift={(0.5,-0.3)}]$\scriptstyle (2, 1)$}] at (1,2) (a4){};
    \node [draw, shape = circle, fill = black, minimum size = 0.1cm, inner sep=0pt, label={[shift={(-0.5,-0.3)}]$\scriptstyle (2, 0)$}] at (0,2) (a5){};
    \node [draw, shape = circle, fill = black, minimum size = 0.1cm, inner sep=0pt, label={[shift={(0.5,-0.3)}]$\scriptstyle (3, 1)$}] at (1,3) (a6){};
    \node [draw, shape = circle, fill = black, minimum size = 0.1cm, inner sep=0pt, label={[shift={(-0.5,-0.3)}]$\scriptstyle (3, 0)$}] at (0,3) (a7){};
    
  \node[] at (1, 4) (a9){};
  \node[label={$\vdots$}] at (0,3.2) (a8){};

    \draw[ultra thin] (a0)--(a2)--(a4)--(a6);
    \draw[ultra thin] (a1)--(a2);
    \draw[ultra thin] (a3)--(a4);
    \draw[ultra thin] (a5)--(a6);
    \draw[dotted] (a7)--(a9);
    \draw[dotted] (a6)--(a9);
\end{tikzpicture}
\caption{Hasse diagram of $(E, \preceq)$}
\end{figure}

Fix a $\clubsuit$-sequence $\langle A_\alpha : \alpha < \omega_1 \text{ and } \alpha \text{ is limit}\rangle$. Denote the set $\{0\} \cup (\omega_1 \times E)$ by $K$ and, for each $\alpha < \omega_1$, the set $\{0\} \cup (\alpha \times E)$ by $K_\alpha$. Moreover, given $x = (\beta, n, i) \in \omega_1 \times E$, we denote $\beta$ by $\alpha(x)$.

We claim that there exists a partial order $\trianglelefteq$ on $K$ such that, for every $\alpha$:
\begin{enumerate}
\item $(K, \trianglelefteq)$ is a lower finite lattice,
\item $K_\alpha$ is an ideal of $(K, \trianglelefteq)$,
\item for every $n,m\in\omega$ and $i,j\in\{0,1\}$,
\[
(\alpha,n,i)\trianglelefteq(\alpha,m,j)
\iff (n,i)\preceq(m,j),
\]
\item $\pi_{K_\alpha}(x) \neq \pi_{K_\alpha}(y)$ for all distinct $x,y\in K$ with $\alpha(x) = \alpha(y) \ge \alpha > 0$,

\item if $\alpha$ is limit and all the elements of $A_\alpha$ have the same parity, then, letting $i \in \{0,1\}$ denote the parity of the elements of $A_\alpha$, for every $n\in\omega$ and  $j \in \{0,1\}$ there is $\beta$  with $2\beta + i \in A_\alpha$ and there is $m \ge n$ such that $\pi_{K_\alpha} (\alpha, n, j) = (\beta, m, i)$.

\end{enumerate}

Let us assume that a partial order $\trianglelefteq$ satisfying (1)-(5) exists and let us fix one, towards showing that $(K, \trianglelefteq)$ is a maximal $3$-ladder of breadth $2$. We will then define such a partial order. Note that properties (1) and (2) directly imply that $0$ is the minimum of $(K, \trianglelefteq)$.

Furthermore, let us denote $\pi_{K_\alpha}$ simply by $\pi_\alpha$, for clarity.

\begin{claim}
$(K, \trianglelefteq)$ is a $3$-ladder of breadth $2$.
\end{claim}
\begin{proof}
Let us first prove that $(K, \trianglelefteq)$ is a $3$-ladder. Since by (1) $(K, \trianglelefteq)$ is a lower finite lattice, it suffices to prove that every element of $K$ has at most three lower covers. Fix some $\alpha < \omega_1$ and some $n \in\omega$. It follows directly from (3), that $(\alpha, n, 0)$ has only one lower cover in $(K, \trianglelefteq)$, namely  $\pi_\alpha(\alpha, n, 0)$. On the other hand, $(\alpha, n, 1)$ has at most three lower covers: $(\alpha, n-1, 0)$ and $(\alpha, n-1, 1)$ (that is, if $n > 0$), and $\pi_\alpha(\alpha, n, 1)$. 
Overall, $(K, \trianglelefteq)$ is a $3$-ladder.

In order to show that $(K, \trianglelefteq)$ has breadth $2$, we need to prove that given any three elements of $\omega_1 \times E$, one of them is less than or equal to the least upper bound of the other two. The only nontrivial case to be checked is when the three elements $x,y,z \in \omega_1\times E$ satisfy $\alpha(x) \le \alpha(y) \le \alpha(z)$ and $\alpha(x) < \alpha(z)$ and the three of them are mutually $\trianglelefteq$-incomparable. By (2), the sets $K_{\alpha(z)}$ and $K_{\alpha(z)+1}$ are  ideals of $(K, \trianglelefteq)$. Then, both $x \vee z$ and $y \vee z$ must belong to $K_{\alpha(z)+1} \setminus K_{\alpha(z)}$. Equivalently, $\alpha(x \vee z) = \alpha(y \vee z) = \alpha(z)$. Moreover, it follows from (3) and from our definition of $\preceq$ that $x \vee z = (\alpha(z), n, 1)$ and $y \vee z = (\alpha(z), m, 1)$ for some $n,m\in\omega$. By (3) again, we conclude that either $x \trianglelefteq y \vee z$ (if $n \le m$) or $y \trianglelefteq x \vee z$ (if $m \le n$). Thus, $(K, \trianglelefteq)$ has breadth $2$.
\end{proof}

The next technical claim is needed to prove that $(K, \trianglelefteq)$ is a maximal $3$-ladder.
\begin{claim}\label{thm:main3:claim2}
Given a meet-subsemilattice $C$ of $(K, \trianglelefteq)$ which is also a $2$-ladder, there is at most one $\alpha < \omega_1$ such that $\{n \in \omega \mid (\alpha, n, i) \in C\}$ is infinite for both $i = 0,1$.
\end{claim}
\begin{proof}
Suppose otherwise, towards a contradiction, and fix two distinct $\alpha, \beta$ with $\alpha < \beta$ such that $\{n \in \omega \mid (\alpha, n, i) \in C\}$ and $\{n \in \omega \mid (\beta, n, i) \in C\}$ are infinite for both $i = 0,1$.

By the property of $\beta$, we can fix an $n\in\omega$ such that $(\beta, n, 0) \in C$ and there exists $m \le n$ with $(\beta, m, 1) \in C$. Let $m$ be the greatest integer $\le n$ such that $(\beta, m, 1) \in C$. Now, again by the property of $\beta$, there is some $k > n$ such that $(\beta, k, 1) \in C$. Let $k$ be the least such integer. Note that $m \le n < k$ and that $(\beta, n, 0), (\beta, m, 1)$ are lower covers of $(\beta, k, 1)$ in $C$. 

Clearly, since $(\beta, n, 0), (\beta, m, 1) \trianglelefteq (\beta, k, 1)$, we have $\pi_{\alpha+1} (\beta, n, 0), \pi_{\alpha+1} (\beta, m, 1) \trianglelefteq \pi_{\alpha+1} (\beta, k, 1)$. Moreover, by (4), $\pi_{\alpha+1} (\beta, k, 1)$, $\pi_{\alpha+1} (\beta, n, 0)$ and $\pi_{\alpha+1} (\beta, m, 1)$ are all distinct. Thus, we conclude that $\pi_{\alpha+1} (\beta, k, 1) \ntrianglelefteq (\beta, n, 0),(\beta, m, 1)$. We now claim that $\pi_{\alpha+1} (\beta, k, 1)$ belongs to $C$, which would result in a contradiction, as it would mean that $(\beta, k ,1)$, which belongs to $C$, has more than $2$ lower covers in $C$. Indeed, since $(K, \trianglelefteq)$ is lower finite by (1), there must be a lower cover $b$ of $(\beta, k, 1)$ in $C$ such that $\pi_{\alpha+1} (\beta, k, 1) \trianglelefteq b$. But since $\pi_{\alpha+1} (\beta, k, 1)$ is neither below $(\beta, n, 0)$ nor $(\beta, m, 1)$, we conclude that $b$ is distinct from both $(\beta, n, 0)$ nor $(\beta, m, 1)$, which are also lower covers of $(\beta, k, 1)$. Thus, $(\beta, k, 1)$ has at least three lower covers in $C$, a contradiction.

By (3) and the way we chose $\alpha$, $C \cap (\{\alpha\} \times E)$ is cofinal in $(\{\alpha\} \times E, \trianglelefteq)$. By (2), this in turn implies that $C \cap (\{\alpha\} \times E)$ is cofinal in $(K_{\alpha+1}, \trianglelefteq)$. 
Since $C \cap (\{\alpha\} \times E)$ is cofinal in $K_{\alpha+1}$, there exists a $y \in C \cap K_{\alpha+1}$ such that $\pi_{\alpha+1} (\beta, k, 1) \trianglelefteq y$. By assumption $C$ is a meet-subsemilattice, and thus $y \wedge (\beta, k, 1) \in C$. By Lemma~\ref{lemma:meetproj}, $y \wedge (\beta, k, 1) = y \wedge \pi_{\alpha+1}(\beta, k, 1)$. By the way we chose $y$, $y \wedge (\beta, k, 1) = \pi_{\alpha+1}(\beta, k, 1)$, and thus $\pi_{\alpha+1}(\beta, k, 1)$ belongs to $C$, as we wanted to show. The contradiction is reached.
\end{proof}

\begin{claim}\label{thm:main3:claimmaximality}
$(K, \trianglelefteq)$ is a maximal $3$-ladder.
\end{claim}
\begin{proof}
By Theorem~\ref{thm:characterization}, we need to prove that there is no cofinal meet-subsemilattice of $(K, \trianglelefteq)$ which is also a $2$-ladder. Suppose, towards a contradiction that there is such a cofinal subset and denote it by $C$.

The following is immediately from Claim~\ref{thm:main3:claim2}: for all sufficiently large $\alpha < \omega_1$ there is an $n \in \omega$ and an $i \in \{0,1\}$ such that $(\alpha, m, i) \not\in C$  for all $m \ge n$. In particular, we can pick $n\in\omega$ and $i \in \{0,1\}$ such that 
\[
X \coloneqq \{\alpha < \omega_1 : \forall m \ge n\ ((\alpha, m, i) \not\in C)\}
\] 
is uncountable. Now let
\[
Z \coloneqq \{2\alpha+i : \alpha \in X\}.
\]

It is easy to see that the set $\{\alpha < \omega_1 : C \cap K_\alpha \text{ is cofinal in }(K_\alpha, \trianglelefteq)\}$ is a club. By the properties of our $\clubsuit$-sequence, there exists a $\gamma < \omega_1$ such that $C \cap K_\gamma$ is cofinal in $(K_\gamma, \trianglelefteq)$ and $A_{\gamma} \subseteq Z$.  Fix one such $\gamma$.

By (5), for every $m \in \omega$ and $j \in \{0,1\}$ there is $\beta \in X$ and $k \ge m$ such that $\pi_{\gamma}(\gamma, m, j) = (\beta, k, i)$. In particular, $\pi_{\gamma}(\gamma, m, j) \not\in C$ for every $m \ge n$ and every $j \in \{0,1\}$.

Since $C$ is cofinal in $(K, \trianglelefteq)$ there must be  $x \in C$, such that $(\gamma, n, 0) \trianglelefteq x$. In particular, there are $m \ge n$ and $j \in \{0,1\}$ such that  $\pi_{\gamma+1}(x) = (\gamma, m, j)$. By the previous paragraph, we conclude that $\pi_\gamma \circ \pi_{\gamma+1} (x) \not\in C$. But by Lemma~\ref{lemma:concproj}, $\pi_\gamma \circ \pi_{\gamma+1} (x) = \pi_\gamma (x)$. Thus, $\pi_\gamma (x) \not\in C$.

Since $C \cap K_\gamma$ is cofinal in $(K_\gamma, \trianglelefteq)$, there exists $y \in C \cap K_\gamma$ such that $\pi_\gamma (x) \trianglelefteq y$. Since $C$ is assumed to be a meet-subsemilattice of $K$, we have  $x \wedge y \in C$. By Lemma~\ref{lemma:meetproj}, $x \wedge y = \pi_\gamma (x) \wedge y$. Thus, by the way we chose $y$, $x \wedge y = \pi_\gamma (x)$. In particular, this means that $\pi_\gamma (x)$ belongs to $C$, which is a contradiction, as in the previous paragraph we have shown that $\pi_\gamma (x)$ does not belong to $C$.
\end{proof}

We are left to construct the partial order $\trianglelefteq$ on $K$ so as to satisfy (1)-(5). Actually, our partial order $\trianglelefteq$ will satisfy an additional technical property, which is necessary to carry out the construction:
\begin{enumerate}
\item[(6)] for every $\alpha$ and $x \in K_\alpha$, $x \trianglelefteq (\alpha, n, 0)$ for infinitely many $n$.
\end{enumerate}

We define a partial order $\trianglelefteq$ on $K$ by defining ${\trianglelefteq} \upharpoonright K_\alpha$ recursively on $\alpha < \omega_1$. Assume that $\trianglelefteq$ has been defined on $K_\alpha$ so that (1)-(6) hold below $\alpha$, and we now describe how to extend $\trianglelefteq$ on $K_{\alpha+1}$. 

First suppose that $\alpha = 0$. Set $0 \trianglelefteq (0, n, i)$ for every $(n, i) \in E$. Then, the ordering $\trianglelefteq$ on $K_1 \setminus \{0\}$ is determined by (3). It immediately follows from $(E, \preceq)$ being a lower finite join-semilattice that $(K_1, \trianglelefteq)$ satisfies (1)-(3) and (6), with (4) and (5) vacuously holding. 

So we can suppose $\alpha > 0$. To proceed, we need to prove the following claim.

\begin{claim}\label{thm:main3:claim4}
There exists a $\trianglelefteq$-increasing sequence $(x_n)_{n\in\omega}$ of elements of $\alpha \times E$ which is cofinal in $K_\alpha$ and such that:
\begin{enumerate}[label={\upshape(\alph*)}]
\item $\pi_{\alpha(x_n)+1}(x_{n+1}) \neq x_n$ for every $n$, and
\item if $\alpha$ satisfies the hypotheses of {\upshape (5)}, then, letting $i \in \{0,1\}$ denote the parity of $A_\alpha$, for every $n$ there is $\beta$ with $2\beta+i \in A_\alpha$ and there is $m \ge n$ such that $x_n = (\beta, m, i)$.
\end{enumerate}
\end{claim}
\begin{proof}
We define the sequence $(x_n)_{n\in\omega}$ by induction on $n$. Suppose that $\alpha$ and $A_\alpha$ satisfy the hypotheses of (5)---otherwise the construction is simpler, as (b) vacuously holds. Let $i$ be the parity of the elements of $A_\alpha$. Moreover, fix an enumeration $(z_n)_{n\in\omega}$ of $K_\alpha$. 

First let $x_0 = (\beta_0, 0, i)$ for some $\beta_0$ such that $2\beta_0 + i \in A_\alpha$. Now suppose that $x_n$ is defined, towards defining $x_{n+1}$. Since $A_\alpha$ is unbounded in $\alpha$, we can pick a $\beta_{n+1}$ such that $2\beta_{n+1}+i \in A_\alpha$ and $\beta_{n+1} > \max\{\alpha (x_n), \alpha (z_n)\}$. 

Fix any $z \in K_\beta$ such that $z_n \trianglelefteq z$ and $x_n \lhd \pi_{\alpha (x_n)+1}(z)$---for example, pick any $w \rhd x_n$ with $\alpha(w) = \alpha(x_n)$ and let $z = w \vee z_n$.  We claim that there exists an $m > n$ such that $z \trianglelefteq (\beta_{n+1}, m, i)$. If $i = 1$, such an $m$ exists because the set $\{(\beta_{n+1}, m, 1) \mid m \in \omega\}$ is a cofinal chain in $K_{\beta_{n+1}+1}$. If $i = 0$ instead, such an $m$ exists because we are assuming that (6) holds in $K_\alpha$ which implies, in particular, that there are infinitely many $m$ with $z \trianglelefteq (\beta_{n+1}, m, 0)$. Fix an $m$ that satisfies our claim and let $x_{n+1} = (\beta_{n+1}, m, i)$. Note that $x_n \lhd \pi_{\alpha(x_n) + 1} (z)  \trianglelefteq \pi_{\alpha(x_n) +1}(x_{n+1})$. In particular, $x_n \neq \pi_{\alpha(x_n) +1}(x_{n+1})$. This ends the inductive definition of the sequence  $(x_n)_{n\in\omega}$.

The sequence $(x_n)_{n\in\omega}$ is $\trianglelefteq$-increasing and cofinal in $K_\alpha$ by construction. Moreover, we have argued that (a) holds at the end of the previous paragraph, and (b) holds again by construction. 
\end{proof}

Fix a sequence as in Claim~\ref{thm:main3:claim4}.  Note that (a) implies that the $x_n$s are actually strictly increasing, that is $x_n \lhd x_{n+1}$ for every $n$. Now extend $\trianglelefteq$ on $K_{\alpha+1}$ by letting:
\begin{itemize}
\item $(\alpha, n, i) \trianglelefteq (\alpha, m, j)$ if and only if $(n, i) \preceq (m, j)$, and
\item for every $y \in K_\alpha$, $y \trianglelefteq (\alpha, n, i)$ if and only if $y \trianglelefteq x_{2n+i}$.
\end{itemize}

Let us show  that the extension of $\trianglelefteq$ on $K_{\alpha+1}$ still satisfies properties (1)-(6).

First note that $(K_{\alpha+1}, \trianglelefteq)$ is a poset. This follows easily from the way we defined the extension of $\trianglelefteq$ on $K_{\alpha+1}$ and from $(K_\alpha, \trianglelefteq)$ being a poset by induction hypothesis. 

Properties (2), (3) and (6) follow directly from the definition of our extension. Moreover, $\trianglelefteq$ is lower finite on $K_{\alpha+1}$: indeed, for every $(n,i) \in E$,
\[
{\trianglelefteq} \downarrow (\alpha, n, i) = ({\trianglelefteq} \downarrow x_{2n+i}) \cup \{(\alpha, m, j) : (m,j) \preceq (n, i)\},
\]
and the set on the right hand side is finite since $(E, \preceq)$ is lower finite and, by induction hypothesis, ${\trianglelefteq} \upharpoonright K_\alpha$ is lower finite.

We now show that $(K_{\alpha+1}, \trianglelefteq)$ is a join-semilattice. Pick $x,y \in K_{\alpha+1}$, towards showing that they have a least upper bound. There are two non-trivial cases to check: when $\alpha(x), \alpha(y) < \alpha$ and when $\alpha(x) < \alpha(y) = \alpha$. If $\alpha(x),\alpha(y) < \alpha$, then $x$ and $y$ have a least upper bound $x \vee y$ in $(K_\alpha, \trianglelefteq)$, by induction hypothesis. Now suppose that $x,y \trianglelefteq (\alpha, n, i)$ for some $n\in\omega$ and $i \in \{0,1\}$. By definition, $x, y \trianglelefteq x_{2n+i}$, which means $x \vee y \trianglelefteq x_{2n+i}$, and therefore $x \vee y \trianglelefteq (\alpha, n, i)$. Hence, the least upper bound of $x$ and $y$ in $(K_\alpha, \trianglelefteq)$ is still their least upper bound in $(K_{\alpha+1}, \trianglelefteq)$. 

Now suppose that $\alpha(x) < \alpha(y) = \alpha$. If $x \trianglelefteq y$ then there is nothing to prove, so we can assume $x \ntrianglelefteq y$. Let $(n,i) \in E$ be such that $y = (\alpha, n, i)$. Let $m$ be the least positive integer strictly greater than $n$ such that $x \trianglelefteq x_{2m+1}$---note that such an $m$ exists because the set $\{x_k \mid k \in \omega\}$ is a cofinal chain in $K_\alpha$.  We claim that $(\alpha, m, 1)$ is the $\trianglelefteq$-least upper bound of $x$ and $y = (\alpha, n, i)$. Since $(n, i) \prec (m, 1)$, we conclude that $(\alpha, m, 1)$ is an $\trianglelefteq$-upper bound of $x$ and $y$. Now pick $(k, j) \in E$ with $x,y \trianglelefteq (\alpha, k, j)$, towards showing $(\alpha, m, 1) \trianglelefteq (\alpha, k, j)$ or, equivalently, $(m, 1) \preceq (k, j)$. 

From $y = (\alpha, n, i) \trianglelefteq (\alpha, k ,j)$---equivalently, $(n, i) \preceq (k, j)$---it follows that either $(n, i) = (k, j)$ or that $j = 1$ and $k > n$. The former possibility is excluded by $x \trianglelefteq (\alpha, k, j)$, as it would mean $x \trianglelefteq y$. Thus, it must be the case that $j = 1$ and $k > n$. By the minimality of $m$, $m \le k$, and therefore $(\alpha, m, 1) \trianglelefteq (\alpha, k, j) = (\alpha, k, 1)$ as we wanted to show.

Being a lower finite join-semilattice with a least element, $(K_{\alpha+1}, \trianglelefteq)$ is a lower finite lattice, and therefore (1) is also satisfied. 

To show that condition (4) is satisfied by our extension, we just need to prove that, given some $\beta \le \alpha$ and two distinct $(n,i), (m,j) \in E$, $\pi_\beta(\alpha, n, i) \neq \pi_\beta(\alpha, m, j)$. By Lemma~\ref{lemma:concproj}, $\pi_\beta = \pi_\beta \circ \pi_\alpha$. Since, by definition of $\trianglelefteq$, $\pi_\alpha(\alpha, n, i) = x_{2n+i}$, we must prove that $\pi_\beta(x_{2n+i}) \neq \pi_\beta(x_{2m+j})$. 

Suppose without loss of generality that $2n+i > 2m+j$. Let $\gamma = \alpha(x_{2m+j})$ and $\delta = \alpha(x_{2n+i})$. Note that $\gamma \le \delta$, since $x_{2m+j} \trianglelefteq x_{2n+i}$. There are three cases to consider:
\begin{description}
\itemsep0.3em
\item[$\beta > \delta$] in this case $\pi_\beta(x_{2n+i}) = x_{2n+i}$ and $\pi_\beta(x_{2m+j}) = x_{2m+j}$. Since $x_{2m+j} \lhd x_{2n+i}$, it holds trivially that $\pi_\beta(x_{2n+i}) \neq \pi_\beta(x_{2m+j})$.

\item[$\gamma < \beta \le \delta$] since $\gamma < \beta$, we have $\pi_\beta(x_{2m+j}) = x_{2m+j}$. By repeated applications of (a) from Claim~\ref{thm:main3:claim4}, $x_{2m+j} \lhd \pi_{\gamma+1}(x_{2n+i})$. From $\gamma+1 \le \beta$ it follows $\pi_{\gamma+1}(x_{2n+i}) \trianglelefteq \pi_\beta(x_{2n+i})$. Thus, $x_{2m+j} \lhd \pi_{\beta}(x_{2n+i})$. In particular, $\pi_\beta(x_{2m+j}) \neq \pi_\beta(x_{2n+i})$.

\item[$\beta \le \gamma$] by repeated applications of (a) from Claim~\ref{thm:main3:claim4}, $x_{2m+j} \lhd \pi_{\gamma+1}(x_{2n+i})$. Note that  $\alpha(\pi_{\gamma+1}(x_{2n+i})) = \alpha(x_{2m+j}) = \gamma$. By Lemma~\ref{lemma:concproj}, $\pi_\beta = \pi_\beta \circ \pi_{\gamma+1}$. Thus, $\pi_\beta(x_{2n+i}) = \pi_\beta \circ \pi_{\gamma+1}(x_{2n+i})$. But since, as we already observed, $\pi_{\gamma+1}(x_{2n+i}) \neq x_{2m+j}$ and $\alpha(\pi_{\gamma+1}(x_{2n+i})) = \alpha(x_{2m+j}) \ge \beta$, we conclude \[\pi_\beta(x_{2n+i}) = \pi_\beta \circ \pi_{\gamma+1}(x_{2n+i}) \neq \pi_\beta (x_{2m+j}),\] since by induction hypothesis (4) holds in $K_\alpha$.
\end{description}
Thus, in all three cases, $\pi_\beta(x_{2n+i}) \neq \pi_\beta(x_{2m+j})$, as we wanted to show. Overall, (4) holds also in $(K_{\alpha+1}, \trianglelefteq)$.

Finally, also (5) is satisfied by $(K_{\alpha+1}, \trianglelefteq)$. Indeed, as we already noted, for every $n$ and $j$, $\pi_\alpha(\alpha, n, j) = x_{2n+j}$ holds by construction. Moreover, if the hypotheses of (5) are satisfied by $\alpha$ and $A_\alpha$, then, by (b), letting $i$ be the parity of the elements of $A_\alpha$, for every $n \in \omega$ and $j \in \{0,1\}$, there is a $\beta$ with $2\beta + i \in A_\alpha$ and $m \ge n$ such that  $x_{2n+j} = (\beta, m, i)$.

This finishes the inductive definition of $\trianglelefteq$ and the proof.
\end{proof}

\section{Maximality and destructibility}\label{sec:diamond}

\begin{definition}
Let $L$ be a maximal $n$-ladder, where $n > 0$, and let $\mathbb{P}$ be a forcing notion. We say that $L$ is \emph{$\mathbb{P}$-indestructible} if $\Vdash_\mathbb{P} ``\check{L}$ is maximal"; otherwise, $L$ is \emph{$\mathbb{P}$-destructible}.
\end{definition}

The notion of destructibility by a forcing stems naturally from the work of Wehrung \cite{MR2609217} and from our proof of Theorem~\ref{thm:main1}. Our Proposition~\ref{prop:wehrung} together with Theorem~\ref{thm:characterization} directly imply that every maximal $n$-ladder of cardinality $\aleph_k$ with $n > k+1$ is $\text{Add}(\omega, \omega_k)$-destructible. In particular, every maximal $3$-ladder of cardinality $\aleph_1$ is $\text{Add}(\omega, \omega_1)$-destructible. 

What about maximal $3$-ladders of cardinality $\aleph_2$? By the preceding paragraph, these maximal $3$-ladders are $\text{Coll}(\omega_1, \omega_2) \times \text{Add}(\omega, \omega_1)$-destructible. In other words, we can destroy any maximal $3$-ladder, regardless of the cardinality, by first collapsing $\omega_2$ to $\omega_1$ and then adding $\aleph_1$-many Cohen reals. 

In general, we cannot destroy the maximality of a $3$-ladder of cardinality $\aleph_1$ by adding less than $\aleph_1$-many Cohen reals. Indeed, a slight modification of Claim~\ref{thm:main3:claimmaximality} shows that the maximal $3$-ladder (of breadth $2$) constructed in Section~\ref{sec:club} under $\clubsuit$ is $\text{Add}(\omega, 1)$-indestructible. 

On the other hand, the existence of an $\text{Add}(\omega, 1)$-destructible maximal $3$-ladder is consistent. In other words, it is consistent that there is a maximal $3$-ladder which is destructible by adding a single Cohen real. Consider for example the following statement, which is a direct corollary of Theorem~\ref{thm:Kmaximal} (see Section~\ref{sec:dominating} for the definition of $\trianglelefteq_\varrho$):

\begin{lemma}\label{lemma:bounding}
If $(K, \trianglelefteq_\varrho)$ is a maximal $3$-ladder, then $(K, \trianglelefteq_\varrho)$  is $\mathbb{P}$-destructible if and only if $\mathbb{P}$ is not ${}^\omega \omega$-bounding\footnote{A forcing $\mathbb{P}$ is ${}^\omega \omega$-bounding if $V \cap {}^\omega \omega$ is dominating in every $\mathbb{P}$-generic extension $V[G]$.}.
\end{lemma}
\begin{proof}
Fix a map $\varrho:[\omega_1]^2 \rightarrow {}^\omega \omega$ such that $(K, \trianglelefteq_\varrho)$ is a maximal $3$-ladder. By Theorem~\ref{thm:Kmaximal}, $(K, \trianglelefteq_\varrho)$ is $\mathbb{P}$-destructible if and only if there is  $p \in \mathbb{P}$ such that
\begin{equation}\label{eq:bounding}
p \Vdash \{\varrho(0, \alpha): \alpha < \omega_1^V\} \text{ is not dominating}.
\end{equation}
By Theorem~\ref{thm:Kmaximal}, we know that $\{\varrho(0, \alpha): \alpha < \omega_1\}$ is dominating. In particular, \eqref{eq:bounding} is equivalent to 
\[
p \Vdash V \cap {}^\omega \omega \text{ is not dominating}.
\]
This shows that $(K, \trianglelefteq_\varrho)$ is $\mathbb{P}$-destructible if and only if $\mathbb{P}$ is not ${}^\omega \omega$-bounding.
\end{proof}

Now, by Theorem~\ref{thm:Kexistence}, $\mathfrak{d}=\aleph_1$ implies the existence of a map $\varrho:[\omega_1]^2\rightarrow {}^\omega \omega$ such that $(K, \trianglelefteq_\varrho)$ is a maximal $3$-ladder. Since $\text{Add}(\omega, 1)$ is well-known not to be ${}^\omega \omega$-bounding, we conclude from Lemma~\ref{lemma:bounding} that, consistently, there is a maximal $3$-ladder which is $\text{Add}(\omega, 1)$-destructible.

So, to summarize what we have said until this point in this section: 1) every maximal $3$-ladder of cardinality $\aleph_1$ is $\text{Add}(\omega, \omega_1)$-destructible; 2) every maximal $3$-ladder is $\text{Coll}(\omega_1, \omega_2) \times \text{Add}(\omega, \omega_1)$-destructible; 3) there are consistent examples of $\text{Add}(\omega, 1)$-indestructible maximal $3$-ladders of cardinality $\aleph_1$; 4) there are consistent examples of   $\text{Add}(\omega, 1)$-destructible maximal $3$-ladders.

A natural question arises at this point: do we need to generically add new reals in order to destroy the maximality of an $n$-ladder? Lemma~\ref{lemma:bounding} tells us that sometimes we really do need to add new reals. However, in this section we prove Theorem~\ref{thm:main4}, which answers the question in the negative: consistently, there exists a maximal $3$-ladder which is destructible by an $\aleph_0$-distributive forcing.

\begin{theorem}
If $\diamondsuit$ holds, then there exists a maximal $3$-ladder which is $T$-destructible for some Suslin tree $T$.
\end{theorem}
\begin{proof}
Denote the set $\omega \times \{\bot, 0, 1\}$ by $D$. We define the following strict partial order $\prec$ on $D$: for all distinct $(n, a), (m, b) \in D$ with $n \le m$, let $(n, a) \prec (m, b)$ if and only if $a = \bot$ or $n < m$. Let $\preceq$ be the reflexive closure of $\prec$. Note that $(D, \preceq)$ is a countable $2$-ladder.

\begin{figure}[H]
\centering
\begin{tikzpicture}
    \node [draw, shape = circle, fill = black, minimum size = 0.1cm, inner sep=0pt, label={[shift={(0,-0.6)}]$\scriptstyle (0, \bot)$},] at (0,-2) (a0){};
    
    \node [draw, shape = circle, fill = black, minimum size = 0.1cm, inner sep=0pt, label={[shift={(-0.3,0)}]$\scriptstyle (0, 0)$}] at (-1.2,-1) (a1){};
    \node [draw, shape = circle, fill = black, minimum size = 0.1cm, inner sep=0pt, label={[shift={(0.3,0)}]$\scriptstyle (0, 1)$}] at (1.2,-1) (a2){};
    \node [draw, shape = circle, fill = black, minimum size = 0.1cm, inner sep=0pt, label={[shift={(0,-0.75)}]$ \scriptstyle (1, \bot)$}] at (0,0) (a3){};
    
    \node [draw, shape = circle, fill = black, minimum size = 0.1cm, inner sep=0pt, label={[shift={(-0.3,0)}]$\scriptstyle (1, 0)$}] at (-1.2,1) (a4){};
    \node [draw, shape = circle, fill = black, minimum size = 0.1cm, inner sep=0pt, label={[shift={(0.3,0)}]$\scriptstyle (1, 1)$}] at (1.2,1) (a5){};
    
     \node [draw, shape = circle, fill = black, minimum size = 0.1cm, inner sep=0pt, label={[shift={(0,-0.75)}]$\scriptstyle (2, \bot)$}] at (0,2) (a6){};
     \node [ label={$\vdots$}] at (0,2.3) (a7){};
     \node [ ] at (-1.2,3) (a8){};
     \node [ ] at (1.2,3) (a9){};

    \draw[ultra thin] (a0)--(a1)--(a3)--(a2)--(a0);
	\draw[ultra thin] (a3)--(a4)--(a6)--(a5)--(a3);
	\draw[dotted] (a6)--(a8);
	\draw[dotted] (a6)--(a9);
\end{tikzpicture}
\caption{Hasse diagram of $(D, \preceq)$}
\end{figure}

Denote the sets $\omega_1 \times D$ and $\alpha \times D$, for some $\alpha < \omega_1$, by $K$ and $K_\alpha$, respectively. Fix a $\diamondsuit$-sequence $\langle A_\alpha : \alpha \in \text{Lim}(\omega_1)\rangle$ and a surjection $\psi: \omega_1 \rightarrow K$. Moreover, given $z = (\beta, n, a)\in K$, we let $\alpha(z)$ be $\beta$---i.e., it is the canonical projection to the first coordinate. 

We define a partial order $\trianglelefteq$ on $K$ such that, for every $\alpha$:
\begin{enumerate}
\item $(K, \trianglelefteq)$ is a lower finite lattice,
\item if $\alpha > 0$, then $K_\alpha$ is an ideal of $(K, \trianglelefteq)$,
\item for every $(n, b), (m ,c) \in D$, 
\[
(\alpha, n, b) \trianglelefteq (\alpha, m, c) \iff (n, b) \preceq (m, c),
\]
\item for every $x \in K_\alpha$ and $(n, b) \in D$, if $x \lhd (\alpha, n, b)$, then $x \trianglelefteq (\alpha, n, \bot)$,
\item if $\alpha$ is limit and  $\psi'' A_\alpha \subseteq K_\alpha$ is a cofinal meet-subsemilattice of $(K_\alpha, \trianglelefteq)$ which is also a $2$-ladder, then $\pi_{K_\alpha}(\alpha, n, \bot) \not\in \psi`` A_\alpha$ for all $n\in\omega$,
\end{enumerate}

We proceed in three steps: first we show that any $\trianglelefteq$ satisfying (1)–(5) yields a maximal $3$-ladder; next we introduce a Suslin tree $T$ and a map $\Gamma$ whose properties ensure $T$-destructibility; finally, we construct all three objects simultaneously.

For clarity, let us denote $\pi_{K_\alpha}$ simply by $\pi_\alpha$.

\begin{claim}
$(K, \trianglelefteq)$ is a maximal $3$-ladder.
\end{claim}
\begin{proof}
Let us first show that $(K, \trianglelefteq)$ is a $3$-ladder. By (1), it suffices to prove that every element has at most three lower covers. Fix some $\alpha < \omega_1$ and some $n \in\omega$. It follows directly from (2) and (4) that for $i= 0, 1$ the element $(\alpha, n, i)$ has only one lower cover in $(K, \trianglelefteq)$, namely  $(\alpha, n, \bot)$. Moreover, $(\alpha, n, \bot)$ has at most three lower covers: $(\alpha, n-1, 0)$ and $(\alpha, n-1, 1)$ (if $n > 0$), and $\pi_\alpha(\alpha, n, \bot)$ (if either $n = 0$ or $\pi_\alpha(\alpha, n, \bot) \ntrianglelefteq (\alpha, n-1, \bot)$). Overall, $(K, \trianglelefteq)$ is a $3$-ladder.

We are left to prove that $(K, \trianglelefteq)$ is maximal. Suppose towards a contradiction that it is not. By Theorem~\ref{thm:characterization}, we can fix a cofinal meet-subsemilattice $C\subseteq K$ which is also a $2$-ladder in its induced ordering. It is easy to see that the set $\{\alpha < \omega_1: C \cap K_\alpha \text{ is cofinal in } (K_\alpha, \trianglelefteq)\}$ is a club. By the properties of our $\diamondsuit$-sequence, there is a limit $\alpha < \omega_1$ such that $\psi`` A_\alpha = C \cap K_\alpha$ and $C \cap K_\alpha$ is cofinal in $(K_\alpha, \trianglelefteq)$. Fix one such $\alpha$. 

Since $C$ is cofinal in $K$, there must be some $p \in C$ such that $(\alpha, 0, \bot) \trianglelefteq p$.  Moreover, since $C \cap K_\alpha$ is cofinal in $(K_\alpha, \trianglelefteq)$, there exists some $q \in C \cap  K_\alpha$ such that $\pi_{\alpha}(p) \trianglelefteq q$. By Lemma~\ref{lemma:meetproj}, $p \wedge q = \pi_\alpha(p) \wedge q$. But since $\pi_\alpha (p) \trianglelefteq q$, we have $p \wedge q = \pi_\alpha (p)$.  Now, since we are assuming $C$ to be a meet-subsemilattice, and since both $p$ and $q$ belong to $C$, we conclude that $\pi_\alpha(p) \in C$. However, as we are going to show, this contradicts (5).

Let $n\in\omega$  be such that $\pi_{\alpha+1}(p) =  (\alpha, n, a)$ for some $a \in \{\bot, 0, 1\}$. By Lemma~\ref{lemma:concproj}, $\pi_\alpha (p) = \pi_\alpha \circ \pi_{\alpha+1} (p) = \pi_\alpha (\alpha, n, a)$. By (4), $\pi_{\alpha}(\alpha, n, a) = \pi_{\alpha}(\alpha, n, \bot)$. Thus, $\pi_\alpha (p) = \pi_\alpha (\alpha, n, \bot)$. By (5), $\pi_\alpha (\alpha, n, \bot) \not\in C$ and thus $\pi_\alpha (p) \not\in C$. However, in the previous paragraph, we have shown that $\pi_\alpha(p) \in C$. Contradiction.
\end{proof}

Let us also prove the following auxiliary claim, which is useful for our construction.

\begin{claim}\label{claim:5}
If $C\subseteq K$ is a meet-subsemilattice which is also a $2$-ladder, then there is at most one $\alpha < \omega_1$, such that $K_{\alpha} \setminus C$ is not cofinal in $(K_{\alpha}, \trianglelefteq)$.
\end{claim}
\begin{proof}
Suppose towards a contradiction that this is not the case. Pick $\alpha, \beta$ with $\alpha < \beta$ such that $K_{\alpha} \setminus C$ is not cofinal in $(K_{\alpha}, \trianglelefteq)$, and the analogous statement holds for $\beta$. Note that $C \cap K_\alpha$ must be cofinal in $(K_\alpha, \trianglelefteq)$---otherwise, $K_\alpha \setminus C$ would certainly be cofinal in  $(K_\alpha, \trianglelefteq)$. 

Since $K_{\beta} \setminus C$ is not cofinal in $(K_{\beta}, \trianglelefteq)$, there is $p \in K_\beta$ such that $q \in C$ for all $q \in K_\beta$ with $q \trianglerighteq p$. If we let $\nu$ be $\max \{\alpha(p), \alpha\}$, it is easy to see that we can pick $N \in\omega$ such that $(\nu, m, b) \in C$ for all $b \in \{\bot, 0, 1\}$ and $m \ge N$. 

Let us show that there exists an $M > N$  such that $\pi_{\alpha}(\nu, M, \bot) \ntrianglelefteq (\nu, M-1, 0)$---note that by (4) this is equivalent to $\pi_{\alpha}(\nu, M, \bot) \ntrianglelefteq (\nu, M-1, 1)$. Fix any $q \in K_\alpha$ such that $q \ntrianglelefteq (\nu, N, \bot)$---such a $q$ certainly exists because $K_\alpha$ is infinite and $\trianglelefteq$ is lower finite. Properties (2) and (4) of $\trianglelefteq$ imply that there is $M > N$ such that $q \vee (\nu, N, \bot) = (\nu, M, \bot)$. Since $q \trianglelefteq (\nu, M, \bot)$  and $q \ntrianglelefteq (\nu, M-1, 0)$, we conclude that $\pi_{\alpha}(\nu, M, \bot) \ntrianglelefteq (\nu, M-1, 0)$.

Fix an $M$ as in the previous paragraph. We now claim that $\pi_\alpha (\nu, M, \bot) \in C$. Since $M > N$, $(\nu, M, \bot) \in C$, by the way we picked $N$. Moreover, $C \cap K_\alpha$ is cofinal in $(K_\alpha, \trianglelefteq)$. Thus, we can fix $p' \in C \cap K_\alpha$ such that $\pi_\alpha (\nu, M, \bot) \trianglelefteq p'$. By Lemma~\ref{lemma:meetproj}, $p' \wedge (\nu, M, \bot) = p' \wedge \pi_\alpha (\nu, M, \bot)$. By the way we chose $p'$, we conclude $p' \wedge (\nu, M, \bot) = \pi_\alpha (\nu, M, \bot)$. Since $C$ is a meet-subsemilattice, we conclude that $p' \wedge (\nu, M, \bot) \in C$. Therefore, $\pi_\alpha (\nu, M, \bot) \in C$.

We claim that $(\nu, M, \bot)$ has more than two lower covers in $C$, contradicting our assumptions on $C$. Indeed, note that both $(\nu, M-1, 0)$ and  $(\nu, M-1, 1)$ belong to $C$, since $M > N$. So $(\nu, M-1, 0)$ and $(\nu, M-1, 1)$ are two lower covers of $(\nu, M, \bot)$ in $C$. Moreover, $\pi_\alpha (\nu, M, \bot) \ntrianglelefteq (\nu, M-1, 0),(\nu, M-1, 1)$. But we showed in the previous paragraph that $\pi_\alpha (\nu, M, \bot)$ belongs to $C$. Hence, $(\nu, M, \bot)$ has more than two lower covers in $C$.
\end{proof}

Along with $\trianglelefteq$, we also define a Suslin tree $T$ and a map $\Gamma : T \rightarrow \mathcal{P}(K)$ that satisfies the following properties: for every $\alpha < \omega_1$ and $x \in T$,
\begin{enumerate}
\setcounter{enumi}{5}
\item $T(\alpha) = \{\omega\alpha+n : n\in\omega\}$,
\item if $\alpha = \omega\alpha$ and $A_\alpha$ is a maximal antichain of $(\alpha, \le_T)$, then $A_\alpha$ is also maximal in $(\alpha + \omega, \le_T)$,
\item $\Gamma(x) \subseteq K_{\mathrm{ht}(x)+1}$,
\item for all $z \in \Gamma(x)$ with $\alpha(z) > 0$,  $\pi_{\alpha(z)}(z) \in \Gamma(x)$,
\item if $\alpha \le \mathrm{ht}(x)$, then there is an $n\in\omega$ such that for all $m \ge n$, $(\alpha, m, \bot) \in \Gamma(x)$,
\item if $\alpha \le \mathrm{ht}(x)$, then for every $n\in\omega$ at most one between $(\alpha, n, 0)$ and $(\alpha, n, 1)$ belong to $\Gamma(x)$,
\item for all $y >_T x$, $\Gamma(y) \supseteq \Gamma(x)$.
\item if $\alpha$ is a successor ordinal and $x \in \omega\alpha$, then there exists an $n \in\omega$ such that for every $m \ge n$ and for every $i \in \{0,1\}$ there is a $y \in T(\alpha)$ with $x \le_T y$ and $(\alpha, m, i) \in \Gamma(y)$.
\end{enumerate}

Now let us assume, on top of $(K, \trianglelefteq)$ satisfying (1)-(5), that the tree $T = (\omega_1, \le_T)$ and the map $\Gamma : T \rightarrow \mathcal{P}(K)$ satisfy (6)-(13). Note that, by (11) and (13), the tree $T$ is ever-branching \cite[Definition 7.3]{MR756630}. This last fact, together with property (7), directly implies that $T$ is a Suslin tree \cite[Lemma 7.7]{MR756630}. We now claim that the maximal $3$-ladder $(K, \trianglelefteq)$ is $T$-destructible. 

\begin{claim}
$(K, \trianglelefteq)$ is $T$-destructible.
\end{claim}
\begin{proof}

We first claim that it suffices to show that, for every $x \in T$, $\Gamma(x)$ is a cofinal meet-subsemilattice of $(K_{\mathrm{ht}(x)+1}, \trianglelefteq\nobreak)$ and it is a $2$-ladder in its induced ordering. Indeed, this together with (12) would imply that for any (generic) cofinal branch $B \subseteq T$, the set $\bigcup_{x \in B} \Gamma(x)$ is a cofinal meet-subsemilattice of $(K, \trianglelefteq)$ and it is a $2$-ladder in the induced ordering. By Theorem~\ref{thm:characterization}, the latter fact implies, in particular, that $(K, \trianglelefteq)$ is not maximal in every $T$-generic extension.

So let us show that for every $x \in T$, $\Gamma(x)$ is a cofinal meet-subsemilattice of $(K_{\mathrm{ht}(x)+1}, \trianglelefteq\nobreak)$ and it is a $2$-ladder.
By (8), $\Gamma(x) \subseteq K_{\mathrm{ht}(x)+1}$. For every $(\alpha, n, b) \in K_{\mathrm{ht}(x)+1}$, we know by (10) that there exists $m > n$ such that $(\alpha, m, \bot) \in \Gamma(x)$. Since by (3) $(\alpha, n, b) \trianglelefteq (\alpha, m, \bot)$, we conclude that $\Gamma(x)$ is indeed cofinal in $(K_{\mathrm{ht}(x)+1}, \trianglelefteq)$.

Let us prove that every element of $\Gamma(x)$ has at most $2$ lower covers in $\Gamma(x)$. Pick $z \in \Gamma(x)$. By (11), the set $\Gamma (x) \cap (\{\alpha(z)\} \times D)$ is a chain. Let $z^-$ be the $\trianglelefteq$-greatest element of $\{z' \in \Gamma (x) : z' \lhd z \text{ and } \alpha(z') = \alpha(z)\}$, if such a set is nonempty. Thus, the lower covers of $z$ are $z^-$ (if it is defined) and $\pi_{\alpha(z)}(z)$ (if $\pi_{\alpha(z)}(z) \ntrianglelefteq z^-$), which belongs to $\Gamma (x)$ by (9).

To see that $\Gamma (x)$ is a meet-subsemilattice, pick $p, q \in \Gamma (x)$, with $\alpha(p) \le \alpha(q)$, towards showing that $p \wedge q \in \Gamma (x)$. Let us denote $\alpha(p \wedge q)$ by $\gamma$. By Lemma~\ref{lemma:meetproj}, $p \wedge q = \pi_{\gamma+1}(p) \wedge \pi_{\gamma+1}(q)$. Moreover, by Lemma~\ref{lemma:concproj} and repeated applications of (9), both $\pi_{\gamma+1}(p)$ and $\pi_{\gamma+1}(q)$ belong to $\Gamma (x)$. We have already observed in the previous paragraph that, by (11),  $\pi_{\gamma+1}(p)$ and $\pi_{\gamma+1}(q)$ are $\trianglelefteq$-comparable. Thus, we conclude that $p \wedge q \in \{\pi_{\gamma+1}(p),\pi_{\gamma+1}(q)\} \subset \Gamma(x)$. In particular, $p \wedge q \in \Gamma(x)$.
\end{proof}

It remains to define the partial order $\trianglelefteq$ on $K$, a tree $T = (\omega_1, \le_T)$ and a map $\Gamma : T \rightarrow \mathcal{P}(K)$ satisfying (1)-(13).

We define ${\trianglelefteq} \upharpoonright K_{\alpha}, {\le_T} \upharpoonright (\omega\alpha)$ and $\Gamma \upharpoonright (\omega\alpha)$ simultaneously by induction on $\alpha$. 

If $\alpha$ is either $0$ or is a limit ordinal, there is nothing to do. So let us assume that $\trianglelefteq$ has been defined on $K_\alpha$, and that both $\le_T$ and $\Gamma$ have been defined on $\omega \alpha$, and that the three of them satisfy (1)-(13)---i.e., we assume that ${\trianglelefteq} \upharpoonright K_\alpha, {\le_T} \upharpoonright (\omega\alpha)$ and $\Gamma \upharpoonright (\omega\alpha)$ satisfy properties (1)-(13) with the universal quantifiers on countable ordinals and on elements of $T$ bounded by $\alpha$ and $\omega\alpha$, respectively. We now describe how to extend $\trianglelefteq$ on $K_{\alpha+1}$ and $\le_T$, $\Gamma$ on $\omega(\alpha+1)$.

There are two cases to consider: $\alpha$ successor, and $\alpha$ limit. In both cases, we will define an increasing sequence $(p_n)_{n\in\omega}$ of elements of $K_\alpha$ such that $\{p_n : n \in\omega\}$ is cofinal in $(K_\alpha, \trianglelefteq)$. The sequence $(p_n)_{n\in\omega}$ induces the following extension of $\trianglelefteq$ on $K_{\alpha+1}$:
\begin{itemize}
\item for every $(n, b), (m,c) \in D$, \[(\alpha, n, b) \trianglelefteq (\alpha, m, c) \iff (n, b) \preceq (m, c),\] 
\item for every $p \in K_\alpha$ and $(n, b) \in D$, $p \trianglelefteq (\alpha, n, b)$ if and only if $p \trianglelefteq p_n$.
\end{itemize}
The extension of $\le_T$ and $\Gamma$, on the other hand, are more sensitive to the two cases.\vspace{0.5em}

\underline{$\alpha$ is a successor ordinal}: First let $p_n = (\alpha-1, n, \bot)$ for each $n$. The $p_n$s defined in this way are clearly cofinal in $(K_\alpha, \trianglelefteq)$. Then, fix a bijection $\langle \cdot, \cdot \rangle: \omega \times \omega \rightarrow \omega$. 

For each $n,m\in\omega$, let $(\omega(\alpha-1)+n) \le_T (\omega\alpha+\langle n,m \rangle)$. Moreover, for each $n \in \omega$ pick $h_n \in \omega$ such that for every $h \ge h_n, (\alpha-1, h, \bot) \in \Gamma(\omega(\alpha-1) + n)$---such an $h_n$ exists because by induction hypothesis $\Gamma \upharpoonright \alpha$ satisfies (10). Fix a map $f: \omega \times \omega \rightarrow \omega \times \{0,1\}$ such that $f``(\{n\} \times \omega) = \{(h, i) \mid h \ge h_n \text{ and } i \in \{0,1\}\}$ for every $n$. 
Then, for each $n,m \in \omega$, let
\begin{equation}\label{eq:main4:defgammasucc}
\Gamma(\omega\alpha + \langle n, m \rangle) \coloneqq \Gamma (\omega(\alpha-1) + n ) \cup \big\{(\alpha, h, \bot) : h \ge h_n\big\} \cup \{(\alpha, f(n,m))\}.
\end{equation}

\vspace{0.5em}
\underline{$\alpha$ is a limit ordinal}: Assume that the hypotheses of (5) and (7) are both satisfied---otherwise the construction is simpler, as at least one of these two properties would vacuously hold.

Fix an enumeration $(x_n)_{n\in\omega}$ of $\alpha$ and an enumeration $(q_n)_{n\in\omega}$ of $K_\alpha$. Fix also an increasing sequence $(\alpha_n)_{n\in\omega}$ of successor ordinals which is cofinal in $\alpha$ such that, for each $n$, $\alpha(q_n) \le \alpha_n$ and there exists $z \in A_\alpha$ with $x_n$ and  $z$ comparable (with respect to $\le_T$) and $\max\{\mathrm{ht}(z), \mathrm{ht}(x_n)\} < \alpha_n$. Moreover, we require that if there is a (unique, by Claim~\ref{claim:5}) $\beta < \alpha$ such that $K_\beta \setminus \psi`` A_\alpha$ is not cofinal in $(K_\beta, \trianglelefteq)$, then $\beta < \alpha_0$.

We define the following objects: a sequence $(k_n)_{n\in\omega}$ of natural numbers; a sequence $(b_n)_{n\in\omega}$ of elements of $\{0,1\}$; a $\le_T$-chain $B(x_n) \subseteq \alpha$ for each $n\in\omega$ such that $x_n \in B(x_n)$ and $B(x_n)$ intersects every $T(\beta)$ with $\beta < \alpha$. Then, we are going to set $p_n \coloneqq (\alpha_n, k_n, b_n)$ for each $n$ and set, for each $x \in \alpha$, $x \le_T \alpha+n$ if and only if $x \in B(x_n)$.

In order to define the chains $B(x_n)$, we construct a sequence $\langle y^n_m : n,m\in\omega\rangle$ such that, for all $n,m \in \omega$:
\begin{enumerate}[label = (\alph*)]
\item $x_n \le_T y^n_0$,
\item there is $z \in A_\alpha$ such that $z \le_T y_0^n$,
\item $y^n_m \le_T y^n_{m+1}$,
\item $\mathrm{ht}(y^n_m) = \alpha_{n+m}$.
\end{enumerate}

Once we have defined this sequence, we set, for each $n \in \omega$, 
\begin{equation}\label{eq:main4:defbranch}
B(x_n) \coloneqq \{y \in \alpha : \exists m \ (y \le_T y^n_m)\}.
\end{equation}

Suppose that we defined $k_m$, $b_m$ and $y^m_{n-m}$ for all $m < n$, towards defining $k_{n}$, $b_{n}$ and $y^m_{n-m}$ for all $m \le n$. 

First, by the properties of our sequence $(\alpha_n)_{n\in\omega}$, there is a $y \in T$ such that $\mathrm{ht}(y) = \alpha_{n}$ and $x_{n}, z \le_T y$ for some $z \in A_\alpha$. Let $y_0^{n}$ be such a $y$. Suppose $n =0$. In this case, we just need to define $k_0$ and $b_0$. By the way we chose $\alpha_0$ (and by Claim~\ref{claim:5}), the set $K_{\alpha_{0}+1} \setminus \psi`` A_\alpha$ is cofinal in $(K_{\alpha_{0}+1}, \trianglelefteq)$. In particular, it is nonempty, and  we can pick $k_{0}$ and $b_{0}$ such that $(\alpha_{0}, k_{0}, b_{0}) \not\in \psi`` A_\alpha$.

Now suppose $n > 0$. Since  $(\alpha, \le_T)$ and $\Gamma \upharpoonright \alpha$ satisfy (13) by induction hypothesis, for every $m < n$ we can pick $h_m \in \omega$ such that for every $h \ge h_m$, for every $i \in \{0,1\}$, there is some $y \in T(\alpha_n)$ with $y_{n-m-1}^m \le_T y$ and $(\alpha_{n}, h, i) \in \Gamma (y)$. Let $\bar{h} \coloneqq \max_{m < n} h_m$. 

By the way we chose $\alpha_0$ (and by Claim~\ref{claim:5}), and since $\alpha_{n} > \alpha_0$, the set $K_{\alpha_{n}+1} \setminus \psi`` A_\alpha$ is cofinal in $(K_{\alpha_{n}+1}, \trianglelefteq)$. In particular,  we can pick $k_{n} \in \omega$ with $k_{n} > \bar{h}$ and $b_{n} \in \{0, 1\}$ such that $(\alpha_{n}, k_{n}, b_{n}) \not\in \psi`` A_\alpha$ and $q_{n-1}, (\alpha_{n-1}, k_{n-1}, b_{n-1}) \trianglelefteq (\alpha_{n}, k_{n}, b_{n})$.

By the way we defined $\bar{h}$, and since $k_{n} > \bar{h}$, we can pick, for each $m < n$, $y_{n-m}^m \in T(\alpha_n)$ such that $y_{n-m-1}^m \le_T  y_{n-m}^m$ and $(\alpha_{n}, k_{n}, b_{n}) \in \Gamma(y_{n-m}^m)$. 

This ends the definition of the sequences $\langle y^n_m : m,n\in\omega\rangle$, $(k_n)_{n\in\omega}$, and $(b_n)_{n\in\omega}$. Recall that we set $p_n \coloneqq (\alpha_n, k_n, b_n)$ for each $n$. Moreover, recall that for each $x \in \alpha$ we set $x \le_T \alpha + n$ if and only if $x \in B(x_n)$. We are left to define the extension of $\Gamma$ on $\omega(\alpha+1)$. For each $n$, let
\begin{equation}\label{eq:main4:defgammalimit}
\Gamma (\alpha+ n) \coloneqq \{(\alpha, m, \bot) : m \ge n\} \cup \bigcup_{m\in\omega} \Gamma (y_m^n). 
\end{equation}

This ends the definition of ${\trianglelefteq} \upharpoonright K_{\alpha+1}$, ${\le_T} \upharpoonright \omega(\alpha+1)$ and $\Gamma \upharpoonright \omega(\alpha+1)$.  We claim that these extensions still satisfy (1)-(13).

\begin{claim}
$(K_{\alpha+1}, \trianglelefteq)$ satisfies {\upshape (1)-(5)}.
\end{claim}
\begin{proof}
 First of all, properties (3) and (4) are trivially satisfied by construction. We now claim that $(K_{\alpha+1}, \trianglelefteq)$ is lower finite: for every $(n, b) \in D$, 
\begin{equation}\label{eq:2}
{\trianglelefteq} \downarrow (\alpha, n, b) = ({\trianglelefteq} \downarrow p_n) \cup \{(\alpha, m, c) : (m, c) \preceq (n, b)\}.
\end{equation}
Since we are assuming $(K_\alpha, \trianglelefteq)$ to be lower finite, ${\trianglelefteq} \downarrow p_n$ is finite. Since also $(D, \preceq)$ is lower finite, we conclude that the set in \eqref{eq:2} is finite. 

Now let us show that $(K_{\alpha+1}, \trianglelefteq)$ is a poset. Clearly, $\trianglelefteq$ is reflexive, since $\preceq$ is. Moreover, it follows quickly from the transitivity of $\trianglelefteq$ on $K_\alpha$ and from the increasingness of the $p_n$s that $\trianglelefteq$ is also transitive on $K_{\alpha+1}$. 

We now prove that $(K_{\alpha+1}, \trianglelefteq)$ is also a join-semilattice. Fix $p,q \in K_{\alpha+1}$, towards showing that they have a $\trianglelefteq$-least upper bound. The only non-trivial case to consider is when $p \in K_\alpha$, $q \in K_{\alpha+1} \setminus K_\alpha$ and $p \ntrianglelefteq q$. Let $(n, b) \in D$ be such that $q = (\alpha, n, b)$. Since we are assuming $p \ntrianglelefteq q$, we have $p \ntrianglelefteq p_n$, by definition of the extension of $\trianglelefteq$. Let $m = \min\{k \in \omega : p \trianglelefteq p_k\}$---the latter set is nonempty because the $p_n$s are cofinal in $(K_\alpha, \trianglelefteq)$. Clearly, $m > n$. Then, the $\trianglelefteq$-least upper bound of $p$ and $q$ is easily seen to be $(\alpha, m, \bot)$. Moreover, (2) holds, i.e., $K_\alpha$ is an ideal of $(K_{\alpha+1}, \trianglelefteq)$.

Since we have proven $(K_{\alpha+1}, \trianglelefteq)$ to be a lower finite join-semilattice with a least element (namely $(0, 0, \bot)$), we conclude that $(K_{\alpha+1}, \trianglelefteq)$ is a (lower finite) lattice, and therefore also (1) is satisfied. 

Finally, let us show that (5) is satisfied. If the hypotheses of (5) are not satisfied, then (5) vacuously holds. So suppose that the hypotheses of (5) are satisfied. By definition of the extension of $\trianglelefteq$, $\pi_\alpha (\alpha, n, \bot) = p_n$ for every $n\in\omega$. It directly follows from the inductive construction of the $p_n$s in the case where $\alpha$ is a limit ordinal that $p_n \not\in \psi`` A_\alpha$. Therefore, $\pi_\alpha(\alpha, n, \bot) \not\in \psi`` A_\alpha$ for every $n\in\omega$ and (5) is satisfied.
\end{proof}

\begin{claim}
${\le_T} \upharpoonright \omega(\alpha+1)$ satisfies {\upshape (6)} and {\upshape (7)}.
\end{claim}
\begin{proof}
If $\alpha$ is a successor ordinal, then (7) holds vacuously and (6) holds directly from the definition of ${\le_T} \upharpoonright \omega(\alpha+1)$ . 

If $\alpha$ is limit, then (6) holds because, by property (d) of $y^n_m$, the branches $B(x_n)$ we have defined \eqref{eq:main4:defbranch} intersect every $T(\beta)$ with $\beta < \alpha$. Moreover, if the hypotheses of (7) hold, then (7) is guaranteed by the way we chose $y^n_0$ for each $n$. Indeed, for each $n$ there is an element $z \in A_\alpha$ such that $z \le_T y^n_0$. In particular $A_\alpha$ is still maximal in $(\alpha+\omega, \le_T)$.
\end{proof}

\begin{claim}
$\Gamma \upharpoonright \omega(\alpha+1)$ satisfies {\upshape (8)-(13)}.
\end{claim}
\begin{proof}
Suppose that $\alpha$ is successor. Properties (8), (10)-(12) directly follow from \eqref{eq:main4:defgammasucc} and from $\Gamma \upharpoonright \omega\alpha$ satisfying (8), (10)-(12) by induction hypothesis. Property (13) also holds. Indeed, for each $n\in\omega$, $i \in \{0,1\}$ and every $h \ge h_n$ there is an $m$ such that $f(n,m) = (h, i)$ and $\omega (\alpha-1)+n \le_T \omega \alpha + \langle n, m\rangle$ and $(\alpha, h, i) \in \Gamma (\omega\alpha + \langle n, m\rangle)$. 

As for property (9), it suffices to prove that for every $x \in T(\alpha)$ and every $z \in \Gamma(x) \cap K_{\alpha+1}$ with $\alpha(z)  = \alpha$, $\pi_\alpha (z) \in \Gamma(x)$. Let $n,m$ be such that $x = \omega\alpha + \langle n, m\rangle$. By \eqref{eq:main4:defgammasucc}, $z$ is either $(\alpha, h ,\bot)$ for some $h \ge h_n$ or it is $(\alpha, f(n,m))$. In the first case, we have $\pi_\alpha (\alpha, h, \bot) = p_h = (\alpha-1, h, \bot)$. By the way we picked $h_n$, this means that $\pi_\alpha(\alpha, h, \bot) \in \Gamma (\omega(\alpha-1) + n)$. But by \eqref{eq:main4:defgammasucc} $\Gamma (\omega(\alpha-1) + n) \subseteq \Gamma (\omega(\alpha) + \langle n, m\rangle)$. Thus, $\pi_\alpha (z) \in \Gamma(x)$. On the other hand, if $z = (\alpha, f(n,m))$, then $z = (\alpha, h, i)$ for some $h \ge h_n$. Reasoning as before, we have $\pi_\alpha (z)  = p_h \in \Gamma(x)$.

Now suppose that $\alpha$ is limit. Property (13) trivially holds, since we are assuming $\Gamma \upharpoonright \omega\alpha$ to satisfy (13) by induction hypothesis. Properties (8), (10)-(12) directly follow from \eqref{eq:main4:defgammalimit} and from $\Gamma \upharpoonright \omega \alpha$ satisfying (8), (10)-(12) by induction hypothesis. To show that (9) holds, it suffices to show that for every $x \in T(\alpha)$, and every $z \in \Gamma(x)$ with $\alpha(z) = \alpha$, $\pi_\alpha (z) \in \Gamma(x)$. Let $x = \alpha + n$. By \eqref{eq:main4:defgammalimit}, $z = (\alpha, m, \bot)$ for some $m \ge n$. It follows directly from the definition of the extension of $\trianglelefteq$ on $K_{\alpha+1}$, that $\pi_\alpha(\alpha, m, \bot) = p_m$. Moreover, by construction we have $p_m \in \Gamma (y^n_{m-n})$. Therefore, by \eqref{eq:main4:defgammalimit}, $p_m \in \Gamma (\alpha + n) = \Gamma (x)$. So also (9) holds.
\end{proof}
This finishes the inductive definition of $\trianglelefteq$, $\le_T$ and $\Gamma$.
\end{proof}

\section{Open questions}\label{sec:openquestions}
\subsection{Ditor's problem}
The main question that remains open is the following \cite{MR2768581, MR2926318}:
\begin{question}\label{q:1}
Is the existence of a $3$-ladder of cardinality $\aleph_2$ a theorem of $\mathsf{ZFC}$?
\end{question}
We expect this question to have a negative answer, assuming the consistency of large cardinals. The large cardinal assumption is necessary since the existence of a $3$-ladder of cardinality $\aleph_2$ follows from $\square_{\omega_1}$ \cite{MR4993406}.

We argued in the final paragraph of Section~\ref{sec:cohen} that $\mathsf{MA}_{\omega_1}(\text{Add}(\omega, \omega_1))$ implies the existence of a $3$-ladder of cardinality $\aleph_2$. In fact, it implies that every maximal $3$-ladder must have cardinality $\aleph_2$. It remains open whether $\mathsf{cov}(\mathcal{B}) > \aleph_1$ suffices. Notice that $\mathrm{cov}(\mathcal{B}) > \aleph_1$ is equivalent to $\mathsf{MA}_{\omega_1}(\text{Add}(\omega, 1))$ \cite[Theorem 7.13]{MR2768685}.

\begin{question}
Does $\mathrm{cov}(\mathcal{B}) > \aleph_1$ imply the existence of a $3$-ladder of cardinality $\aleph_2$?
\end{question}

The next two open questions are motivated by the following observation: if every maximal $3$-ladder is indestructible by a $\sigma$-closed forcing, then it follows from standard arguments that collapsing a Mahlo cardinal $\kappa$ via $\text{Col}(\omega_1, {<}\kappa)$ would result in a model in which there are no $3$-ladders of cardinality $\aleph_2$, thus answering Question~\ref{q:1} in the negative.

Similarly, if every maximal $3$-ladder is indestructible by a countable support iteration of Sacks forcing $\mathbb{S}$ \cite{MR556894}, then the iteration of $\kappa$-many Sacks forcings, where $\kappa$ is Mahlo, would again result in a model of $\mathsf{ZFC}$ in which there are no $3$-ladders of cardinality $\aleph_2$.

\begin{question}\label{q:3}
Is every maximal $3$-ladder indestructible by $\sigma$-closed forcings?
\end{question}
\begin{question}\label{q:4}
Is every maximal $3$-ladder $\mathbb{S}$-indestructible?
\end{question}

Another question related to Question~\ref{q:3} is the following:
\begin{question}\label{q:5}
Does $\mathsf{CH}$ imply the existence of a $3$-ladder of cardinality $\aleph_2$?
\end{question}
Indeed, a positive answer to Question~\ref{q:3} would entail, by the argument sketched above, that Question~\ref{q:5} has a negative answer (assuming the consistency of a Mahlo cardinal).

\subsection{Maximal $3$-ladders of breadth $2$}
Theorem~\ref{thm:main2} implies, in particular, that the existence of a maximal $3$-ladder of cardinality $\aleph_1$ follows from $\mathsf{CH}$. But the maximal $3$-ladder constructed in the proof of Theorem~\ref{thm:main2} has breadth $3$. On the other hand, Theorem~\ref{thm:main3} shows that, assuming $\clubsuit$, we can construct a maximal $3$-ladder of breadth $2$. In other words, Theorem~\ref{thm:main3} provides a maximal $3$-ladder of cardinality $\aleph_1$ which is ``tamer" than the one provided by Theorem~\ref{thm:main2}.

The natural question is whether Theorem~\ref{thm:main2} can be improved by showing that  $\mathsf{CH}$ suffices to prove the existence of a maximal $3$-ladder of breadth $2$.

\begin{question}
Does $\mathsf{CH}$ imply the existence of a maximal $3$-ladder of breadth $2$?
\end{question}

\subsection{Spectra of maximality}
For each positive integer $n > 0$,  the \emph{spectrum} of maximal $n$-ladders  (in symbols, $\mathrm{lsp}_n$) is the set of cardinalities of maximal $n$-ladders, that is
\[
\mathrm{lsp}_n \coloneqq \big\{|L| : L \text{ is a maximal }n\text{-ladder}\big\}.
\]
These spectra encode a sizable portion of the set-theoretic behavior of $n$-ladders. By Ditor's Theorems~\ref{thm:Ditor} and \ref{thm:Ditormax}, we know that for every $n > 1$,
\[
\mathrm{lsp}_n \subseteq \{\aleph_1, \aleph_2, \dots, \aleph_{n-1}\}.
\]
We can also recast Corollary~\ref{cor:main1} and Theorem~\ref{thm:main2} as follows:
\begin{itemize}
\item $\text{Add}(\omega, \omega_\omega)$ forces $\mathrm{lsp}_n = \{\aleph_{n-1}\}$ for every $n$ (Corollary~\ref{cor:main1}).
\item If $\mathfrak{d} = \aleph_1$, then $\aleph_1 \in \mathrm{lsp}_3$ (Theorem~\ref{thm:main2}).
\end{itemize}

Moreover, once one observes that every $n$-ladder extends to a maximal $n$-ladder, the existence of an $n$-ladder of cardinality $\aleph_k$ is equivalent to $\max \mathrm{lsp}_n \ge \aleph_k$. As a direct consequence, the sequence $\langle \max \mathrm{lsp}_n : n >  0\rangle$ is non-decreasing. 

It is natural to investigate the range of possible consistent behaviors of these spectra.
Are there structural---that is, provable in $\mathsf{ZFC}$---constraints on their mutual arrangements? For example:

\begin{question}
If there is a $4$-ladder of cardinality $\aleph_2$, does it follow that there is also a $3$-ladder of cardinality $\aleph_2$?
\end{question}
\begin{question}
More generally, for a given $m > 2$, is it true that either $\aleph_{m-1} \in \mathrm{lsp}_m$ or $\max \mathrm{lsp}_n < \aleph_{m-1}$ for every $n \ge m$?
\end{question}
\begin{question}
Is it always true that $\mathrm{lsp}_{n} \cap \{\aleph_1, \dots, \aleph_{m-1}\} \subseteq \mathrm{lsp}_m$ holds for every $n > m > 1$?
\end{question}
\begin{question}
Is it consistent that $\mathrm{lsp}_n = \{\aleph_1, \aleph_2, \dots, \aleph_{n-1}\}$ for every $n > 1$?
\end{question}
\begin{question}\label{q:9}
Assuming the consistency of large cardinals, is it consistent that $\mathrm{lsp}_n = \{\aleph_1\}$ for every $n > 1$?
\end{question}

Note that a positive answer to Question~\ref{q:9} would entail a strong negative answer to Ditor's Problem: for every $n > 2$, there are no $n$-ladders of cardinality $\aleph_2$, let alone $\aleph_{n-1}$.

\printbibliography

@incollection {MR2768691,
    AUTHOR = {Cummings, James},
     TITLE = {Iterated forcing and elementary embeddings},
 BOOKTITLE = {Handbook of set theory. Vols. 1, 2, 3},
     PAGES = {775--883},
 PUBLISHER = {Springer, Dordrecht},
      YEAR = {2010},
      ISBN = {978-1-4020-4843-2}
}

@book {MR795592,
    AUTHOR = {Erd\H{o}s, Paul and Hajnal, Andr\'{a}s and M\'{a}t\'{e}, Attila and
              Rado, Richard},
     TITLE = {Combinatorial set theory: partition relations for cardinals},
    SERIES = {Studies in Logic and the Foundations of Mathematics},
    VOLUME = {106},
 PUBLISHER = {North-Holland Publishing Co., Amsterdam},
      YEAR = {1984},
     PAGES = {347},
      ISBN = {0-444-86157-2}
}

@article {MR48518,
    AUTHOR = {Kuratowski, Casimir},
     TITLE = {Sur une caract\'{e}risation des alephs},
   JOURNAL = {Fund. Math.},
  FJOURNAL = {Polska Akademia Nauk. Fundamenta Mathematicae},
    VOLUME = {38},
      YEAR = {1951},
     PAGES = {14--17},
      ISSN = {0016-2736, 1730-6329}
}

@book {MR756630,
    AUTHOR = {Kunen, Kenneth},
     TITLE = {Set theory},
    SERIES = {Studies in Logic and the Foundations of Mathematics},
    VOLUME = {102},
 PUBLISHER = {North-Holland Publishing Co., Amsterdam},
      YEAR = {1983},
      ISBN = {0-444-86839-9}
}

@book {MR1623206,
    AUTHOR = {Shelah, Saharon},
     TITLE = {Proper and improper forcing},
    SERIES = {Perspectives in Mathematical Logic},
   EDITION = {Second},
 PUBLISHER = {Springer-Verlag, Berlin},
      YEAR = {1998},
      ISBN = {3-540-51700-6}
}

@incollection {MR1900391,
    AUTHOR = {Hru\v{s}\'{a}k, Michal},
     TITLE = {Life in the Sacks model},
      NOTE = {29th Winter School on Abstract Analysis},
   JOURNAL = {Acta Univ. Carolin. Math. Phys.},
  FJOURNAL = {Acta Universitatis Carolinae. Mathematica et Physica},
    VOLUME = {42},
      YEAR = {2001},
    NUMBER = {2},
     PAGES = {43--58},
      ISSN = {0001-7140}
}

@article {MR454941,
    AUTHOR = {Ostaszewski, Adam J.},
     TITLE = {A perfectly normal countably compact scattered space which is
              not strongly zero-dimensional},
   JOURNAL = {J. London Math. Soc. (2)},
  FJOURNAL = {Journal of the London Mathematical Society. Second Series},
    VOLUME = {14},
      YEAR = {1976},
    NUMBER = {1},
     PAGES = {167--177},
      ISSN = {0024-6107, 1469-7750}
}

@book {MR4917577,
    AUTHOR = {Halbeisen, Lorenz J.},
     TITLE = {Combinatorial set theory---with a gentle introduction to
              forcing},
    SERIES = {Springer Monographs in Mathematics},
   EDITION = {Third},
 PUBLISHER = {Springer, Cham},
      YEAR = {2025},
      ISBN = {978-3-031-91751-6; 978-3-031-91752-3}
}

@incollection {MR2768685,
    AUTHOR = {Blass, Andreas},
     TITLE = {Combinatorial cardinal characteristics of the continuum},
 BOOKTITLE = {Handbook of set theory. {V}ols. 1, 2, 3},
     PAGES = {395--489},
 PUBLISHER = {Springer, Dordrecht},
      YEAR = {2010},
      ISBN = {978-1-4020-4843-2}
}

@article {MR4993406,
    AUTHOR = {Notaro, Lorenzo},
     TITLE = {Ladders and squares},
   JOURNAL = {Adv. Math.},
  FJOURNAL = {Advances in Mathematics},
    VOLUME = {485},
      YEAR = {2026},
     PAGES = {Paper No. 110714, 36},
      ISSN = {0001-8708,1090-2082}
}

@article {MR2309879,
    AUTHOR = {R\r{u}\v{z}i\v{c}ka, Pavel and T\r{u}ma, Ji\v{r}\'{i} and Wehrung, Friedrich},
     TITLE = {Distributive congruence lattices of congruence-permutable
              algebras},
   JOURNAL = {J. Algebra},
  FJOURNAL = {Journal of Algebra},
    VOLUME = {311},
      YEAR = {2007},
    NUMBER = {1},
     PAGES = {96--116},
      ISSN = {0021-8693}
}

@article {MR1800815,
    AUTHOR = {Wehrung, Friedrich},
     TITLE = {Representation of algebraic distributive lattices with
              {$\aleph_1$} compact elements as ideal lattices of regular
              rings},
   JOURNAL = {Publ. Mat.},
  FJOURNAL = {Publicacions Matem\`atiques},
    VOLUME = {44},
      YEAR = {2000},
    NUMBER = {2},
     PAGES = {419--435},
      ISSN = {0214-1493}
}

@article {MR1768850,
    AUTHOR = {Gr\"{a}tzer, George and Lakser, Harry and Wehrung, Friedrich},
     TITLE = {Congruence amalgamation of lattices},
   JOURNAL = {Acta Sci. Math. (Szeged)},
  FJOURNAL = {Acta Universitatis Szegediensis. Acta Scientiarum
              Mathematicarum},
    VOLUME = {66},
      YEAR = {2000},
    NUMBER = {1-2},
     PAGES = {3--22},
      ISSN = {0001-6969}
}

@book {MR2768581,
    AUTHOR = {Gr{\"{a}}tzer, George},
     TITLE = {Lattice theory: foundation},
 PUBLISHER = {Birkh\"{a}user/Springer Basel AG, Basel},
      YEAR = {2011},
      ISBN = {978-3-0348-0017-4}
}

@article {MR0732199,
    AUTHOR = {Ditor, Seymour Z.},
     TITLE = {Cardinality questions concerning semilattices of finite
              breadth},
   JOURNAL = {Discrete Math.},
  FJOURNAL = {Discrete Mathematics},
    VOLUME = {48},
      YEAR = {1984},
    NUMBER = {1},
     PAGES = {47--59},
      ISSN = {0012-365X}
}

@article {MR2609217,
    AUTHOR = {Wehrung, Friedrich},
     TITLE = {Large semilattices of breadth three},
   JOURNAL = {Fund. Math.},
  FJOURNAL = {Fundamenta Mathematicae},
    VOLUME = {208},
      YEAR = {2010},
    NUMBER = {1},
     PAGES = {1--21},
      ISSN = {0016-2736}
}

@book{MR1940513,
	author = {Jech, Thomas},
	isbn = {3-540-44085-2},
	note = {The third millennium edition, revised and expanded},
	publisher = {Springer-Verlag},
	series = {Springer Monographs in Mathematics},
	title = {Set theory},
	year = {2003}
}

@article {MR862871,
    AUTHOR = {Dobbertin, Hans},
     TITLE = {Vaught measures and their applications in lattice theory},
   JOURNAL = {J. Pure Appl. Algebra},
  FJOURNAL = {Journal of Pure and Applied Algebra},
    VOLUME = {43},
      YEAR = {1986},
    NUMBER = {1},
     PAGES = {27--51},
      ISSN = {0022-4049}
}

@article {MR2926318,
    AUTHOR = {Wehrung, Friedrich},
     TITLE = {Infinite combinatorial issues raised by lifting problems in
              universal algebra},
   JOURNAL = {Order},
  FJOURNAL = {Order. A Journal on the Theory of Ordered Sets and its
              Applications},
    VOLUME = {29},
      YEAR = {2012},
    NUMBER = {2},
     PAGES = {381--404},
      ISSN = {0167-8094}
}

@article {MR556894,
    AUTHOR = {Baumgartner, James E. and Laver, Richard},
     TITLE = {Iterated perfect-set forcing},
   JOURNAL = {Ann. Math. Logic},
  FJOURNAL = {Annals of Mathematical Logic},
    VOLUME = {17},
      YEAR = {1979},
    NUMBER = {3},
     PAGES = {271--288},
      ISSN = {0003-4843}
}

\end{document}